\documentclass[12pt,a4paper]{amsart}
\usepackage[shortlabels]{enumitem}
\usepackage{ifthen,latexsym,amsthm,amssymb,amsmath,bbm,fixmath}
\usepackage{hyperref}
\usepackage[nobysame]{amsrefs}
\usepackage{tikz,graphicx}
\usetikzlibrary{calc,shapes,arrows}
\usetikzlibrary{arrows.meta}

\numberwithin{equation}{section}

\setlist{leftmargin=3\parindent,labelindent=3\parindent}
\setlist[enumerate]{%
  leftmargin=3\parindent,%
  align=left,%
  labelwidth=3\parindent,%
  labelsep=0pt%
}
\setlist[enumerate,1]{%
  label={\normalfont (\thesection.\arabic{equation})}, ref={\normalfont \thesection.\arabic{equation}},
  resume%
}

\usepackage{fullpage}

\newcommand{\C}[1]{{\protect\mathcal{#1}}}

\newcommand{\I}[1]{{\mathbbm #1}}

\newcommand{\bSigma}{\boldsymbol{\Sigma}}

\newcommand{\bDelta}{\boldsymbol{\Delta}}
\newcommand{\bB}{\boldsymbol{\mathcal{B}}}

\newcommand{\e}{\varepsilon}

\renewcommand{\mid}{:}
\renewcommand{\ldots}{\hspace{0.9pt}.\hspace{0.3pt}.\hspace{0.3pt}.\hspace{1.5pt}}
\renewcommand{\ge}{\geqslant}
\renewcommand{\geq}{\geqslant}
\renewcommand{\le}{\leqslant}
\renewcommand{\leq}{\leqslant}

\newcommand{\actson}{{\curvearrowright}}
\newcommand{\range}{\mathrm{range}}

\newcommand{\hide}[1]{}

\newcommand{\beq}[1]{\begin{equation}\label{#1}}
\newcommand{\eeq}{\end{equation}}

\newtheorem{theorem}{Theorem}
\numberwithin{theorem}{section}
\newtheorem{lemma}[theorem]{Lemma}

\newtheorem{corollary}[theorem]{Corollary}
\newtheorem{obs}[theorem]{Observation}
\newtheorem{claim}{Claim}[theorem]

\theoremstyle{definition}
\newtheorem{defn}[theorem]{Definition}

\newtheorem{remark}[theorem]{Remark}

\renewcommand{\qed}{\nolinebreak\mbox{\hspace{5 true pt}%
  \rule[-0.85 true pt]{3.9 true pt}{8.1 true pt}}}

\newtheorem{manuallemmainner}{Lemma}
\newenvironment{manuallemma}[1]{%
\renewcommand\themanuallemmainner{#1}%
  \manuallemmainner
}{\endmanuallemmainner}

\title{Circle Squaring with Pieces of Small Boundary and Low Borel Complexity}

\author{Andr\'{a}s M\'{a}th\'{e}$^{\ast}$}
\address{Mathematics Institute\\
University of Warwick\\
Coventry CV4 7AL, UK}
\thanks{$^{\ast}$\,Supported by the Hungarian National Research, Development and Innovation Office -- NKFIH, 124749.}

\author{Jonathan A. Noel$^{\dagger}$}
\address{Department of Mathematics and Statistics\\ University of Victoria\\ Victoria, B.C., Canada V8P 5C2}
\thanks{$^{\dagger}$\,Supported by Leverhulme Early Career Fellowship ECF-2018-534, NSERC Discovery Grant RGPIN-2021-02460 and NSERC Early Career Supplement DGECR-2021-00024.} 

\author{Oleg Pikhurko$^{\ddagger}$}
\address{Mathematics Institute and DIMAP\\
University of Warwick\\
Coventry CV4 7AL, UK}
\thanks{$^{\ddagger}$\,Supported by ERC Advanced Grant 101020255 and Leverhulme Research Project Grant RPG-2018-424.}

\date{\today}

\begin{document}


\newcommand{\boxdim}{\operatorname{dim}_\square}
\newcommand{\dist}{\operatorname{dist}}
\newcommand{\diam}{\operatorname{diam}}
\renewcommand{\deg}{\operatorname{deg}}
\newcommand{\lex}{\prec_{\operatorname{lex}}}
\newcommand{\eps}{\varepsilon} 
\newcommand{\ind}{\mathbbm{1}}
\newcommand{\induced}{\upharpoonright}
\newcommand{\cl}{\operatorname{cl}}
\newcommand{\interior}{\operatorname{int}}
\newcommand{\vvec}{\boldsymbol}
\newcommand{\comp}{\operatorname{comp}}
\newcommand{\holes}{\operatorname{holes}}
\newcommand{\fout}[1]{#1^{\operatorname{out}}}
\newcommand{\fin}[1]{{#1}_{\operatorname{in}}}

\newcommand{\N}{M}
\newcommand{\FE}[2]{{\C E}^{#1}_{#2}}
\newcommand{\FO}[2]{{\C O}^{#1}_{#2}}
\renewcommand{\Tilde}[1]{\widetilde{#1}}
\renewcommand{\Hat}[1]{\widehat{#1}}

\tikzset{
  smallblack/.style={circle, draw=black!100,fill=black!100,thick, inner sep=0pt, minimum size=1.3mm},
    smallcirc/.style={circle, draw=black!100,thick, inner sep=1pt, minimum size=1.3mm},
    smallsq/.style={regular polygon, regular polygon sides=4, draw=black!100,thick, inner sep=-1pt}
}

\begin{abstract}
Tarski's Circle Squaring Problem from 1925 asks whether it is possible to partition a disk in the plane 
into finitely many pieces and reassemble them via isometries to yield a partition of a square of the same area. It was finally resolved by Laczkovich in 1990 in the affirmative.
Recently, several new proofs have emerged which achieve circle squaring with better structured pieces: namely, pieces which are Lebesgue measurable and have the property of Baire (Grabowski--M\'ath\'e--Pikhurko) or even are Borel (Marks--Unger). 

In this paper, we show that circle squaring is possible with Borel pieces of positive Lebesgue measure whose boundaries have upper Minkowski dimension less than~2 (in particular, each piece is Jordan measurable). We also improve the Borel complexity of the pieces: namely,
we show that each piece can be taken to be a Boolean combination of $F_\sigma$ sets.
This is a consequence of our more general result that applies 
to any two bounded subsets of $\I R^k$, $k\ge 1$, of equal positive measure whose boundaries have upper Minkowski dimension smaller than~$k$.
\end{abstract}

\maketitle

\section{Introduction}

Tarski's Circle Squaring Problem~\cite{Tarski25} from 1925 asks if a circle (i.e.,\ a circular disk) and a square of the same area in $\mathbbm{R}^2$
are \emph{equidecomposable}, that is, whether we can partition the circle  into finitely many pieces and apply some isometry to each piece to get a partition of the square.
This question was answered affirmatively some 65 years later by Laczkovich who showed in a deep and groundbreaking paper~\cite{Laczkovich90} that, in fact, it is possible to square a circle  using translations only. 

The Axiom of Choice plays a crucial role in his proof and, consequently, the pieces of his circle squaring could not be guaranteed to have any discernible regularity properties. A notable problem (mentioned by 
e.g.,~Wagon~\cite{Wagon81}*{Appendix~C} or Laczkovich~\cite{Laczkovich90}*{Section~10}) has been to determine whether there exist circle squarings using ``better'' structured pieces. Recently, Grabowski, M\'{a}th\'{e} and Pikhurko~\cite{GrabowskiMathePikhurko17} proved that the pieces of a circle squaring can simultaneously be Lebesgue measurable and have the property of Baire. Then, Marks and Unger~\cite{MarksUnger17} proved that the pieces can be made Borel. (Let us assume in this paper that the disk and the square are closed and thus Borel sets.)
In fact, Marks and Unger~\cite{MarksUnger17}*{Section~7} showed that the pieces of a circle squaring can be chosen to be in $\bB(\bSigma_4^0)$, where $\bSigma_i^0$ is the $i$-th additive class of the standard Borel hierarchy (see e.g.,\ \cite{Kechris:cdst}*{Section 11.B}) and $\bB(\C F)$ denotes the algebra generated by a set family $\C F$ (that is, the family of all Boolean combinations of elements from~$\C F$). For some generalisations and simplifications of the above results, we refer the reader to Cie\'sla and Sabok~\cite{CieslaSabok22} and Bowen, Kun and Sabok~\cite{BowenKunSabok21,BowenKunSabok25}. 

Shortly after his circle-squaring paper, Laczkovich~\cites{Laczkovich92,Laczkovich92b} proved a far-reaching generalisation.  Before stating it, let us set up some notation.
Fix $k\ge 1$. Let $\lambda$ denote the Lebesgue measure on $\I R^k$ and let $\partial X$ denote the (topological) boundary of $X\subseteq \I R^k$. Recall that the \emph{upper Minkowski dimension}, sometimes called \emph{box} or \emph{grid dimension}, of  $X\subseteq \mathbbm{R}^k$ is
\begin{equation}
\label{eq:DefBoxDim}
\boxdim(X):=\limsup_{\delta\to0^+}\frac{\log(N_{\delta}(X))}{\log(\delta^{-1})},
\end{equation}
where $N_\delta(X)$ is the number of boxes from the regular grid in $\I R^k$ of side-length~$\delta$ that intersect~$X$.

\begin{theorem}[Laczkovich~\cites{Laczkovich92,Laczkovich92b}]
\label{th:L}
If $k\ge1$ and $A,B\subseteq \mathbbm{R}^k$ are bounded sets such that $\lambda(A)=\lambda(B)>0$, $\boxdim(\partial A)<  k$ and $\boxdim(\partial B)<k$, then $A$ and $B$ are equidecomposable by translations.\end{theorem}

The subsequent papers~\cites{GrabowskiMathePikhurko17,MarksUnger17} on circle squaring in fact prove  appropriate ``constructive'' versions of Theorem~\ref{th:L} and derive the corresponding circle squaring results as special cases. The aim of this paper is to extend this line of research.  
In the context of circle squaring, we prove the following result which decreases the Borel complexity of the pieces by two hierarchy levels and ensures that the boundary of each piece is ``small'' in a strong sense.

\begin{theorem}
\label{th:circleSquare}
In $\I R^2$, a closed disk and a closed square of the same area can be equidecomposed  using translations so that every piece has
boundary of upper Minkowski dimension at most~$1.987$, belongs to $\bB(\bSigma_2^0)$ (i.e.,\ is a Boolean combination of $F_\sigma$ sets), and has positive Lebesgue measure.
\end{theorem}

Recall that a subset $X\subseteq \mathbbm{R}^k$ is \emph{Jordan measurable} if its indicator function is Riemann integrable. An equivalent definition is that $X$ is bounded and $\lambda(\partial X)=0$. It easily follows that any bounded set $X\subseteq\mathbbm{R}^k$ with $\boxdim(\partial X)<k$ is Jordan measurable. Therefore, Theorem~\ref{th:circleSquare} implies that circle squaring is possible with Jordan measurable pieces, which addresses questions by Laczkovich~\cite{LaczkovichTalk} and  M\'{a}th\'{e}~\cite{Mathe18icm}*{Question~6.2}. An advantage of a Jordan measurable circle squaring is that an arbitrarily large portion of it can be described in an error-free way with finitely many bits of information. Namely, for every $\e>0$, if $n$ is large enough, then at most $\e n^2$ boxes of the regular $n\times n$ grid on the equidecomposed unit square can intersect more than one piece and thus $O(n^2)$ bits are enough to describe our equidecompostion up to a set of measure at most~$\e$. (Furthermore, the dimension estimate of Theorem~\ref{th:circleSquare} shows that $\e$, as a function of $n\to\infty$, can be taken to be~$n^{-0.013+o(1)}$.)


We will obtain Theorem~\ref{th:circleSquare} as a special case of the following general result. For a set $A\subseteq \I R^k$, let $\C T_{A}:=\{A+\vvec{t}:\vvec{t}\in\I R^k\}$ consist of all its translations. For a family $\mathcal{F}$ of sets, let $\bSigma(\mathcal{F})$ be the collection of all countable unions of sets in $\mathcal{F}$. Also, recall that $\bSigma_1^0$ stands for the collection of open sets in $\I R^k$.

\begin{theorem}
\label{th:main}
If $k\ge1$ and $A,B\subseteq \mathbbm{R}^k$ are bounded sets such that $\lambda(A)=\lambda(B)>0$, $\boxdim(\partial A)<  k$ and $\boxdim(\partial B)<k$, then $A$ and $B$ are equidecomposable by translations so that all the following statements hold simultaneously:
\begin{enumerate}[(a)]
\item \label{it:a} for some $\zeta>0$ that depends on $k$, $\boxdim(\partial A)$ and $\boxdim(\partial B)$ only, the topological boundary of each piece has upper Minkowski dimension at most $k-\zeta$,
\item \label{it:b} each piece belongs to $\bB(\bSigma(\bB (\bSigma_1^0\cup \C T_A\cup \C T_B)))$,
\item \label{it:c} if
 \begin{equation}\label{eq:it:c}
  \lambda\left(\left\{\vvec{t}\in\mathbbm{R}^k: (A+\vvec{t})\cap B\neq \emptyset \text{ and } \lambda\left((A+\vvec{t})\cap B\right)=0\right\}\right) = 0
  \end{equation}
 (that is, the set of vectors $\vvec{t}\in\I R^k$ such that $(A+\vvec{t})\cap B$ is non-empty and Lebesgue-null has measure 0),  then each piece has positive Lebesgue measure.
\end{enumerate}
\end{theorem}

In fact, an explicit expression for a possible parameter $\zeta>0$ in Part~\ref{it:a} as a function of  $k$, $\boxdim(\partial A)$ and $\boxdim(\partial B)$ is given in~\eqref{eq:zeta}. For circle squaring, the bound in~\eqref{eq:zeta} states that any $\zeta<1/73$ suffices. Interestingly, by Part~\ref{it:a}, each non-null piece of the equidecomposition individually satisfies the hypothesis of the theorem and is therefore equidecomposable to a ball or a cube of the same measure. Moreover, each non-null piece automatically has non-empty interior. (Note that a Jordan measurable subset of $\I R^k$ is null if and only if it has empty interior.)

If the sets $A$ and $B$ in Theorem~\ref{th:main} are Borel, then, by Part~\ref{it:b}, all pieces of the equidecomposition can be taken to be Borel with a good control over their Borel complexity (that improves upon the analogous result in~\cite{MarksUnger17} by two hierarchy levels). For example, if each of $A$ and $B$ is  simultaneously  a $G_\delta$ set and an $F_\sigma$ set, that is, belongs to $\bDelta_{2}^0:=\{X\in \bSigma_2^0\mid \I R^k\setminus X\in \bSigma_2^0\}$  (in particular, if each of $A$ and $B$ is open or closed), then every piece in the equidecomposition is a Boolean combination of $F_\sigma$ sets, which follows from
$$
 \bB(\bSigma(\bB (\bDelta_2^0)))=\bB(\bSigma (\bDelta_2^0))\subseteq \bB(\bSigma(\bSigma_2^0))=\bB(\bSigma_2^0),
$$
 where  we use the easy facts that $\bDelta_{2}^0$ is closed under Boolean combinations and $\bSigma_2^0$ is closed under countable unions.

Regarding Part~\ref{it:c} of Theorem~\ref{th:main}, note that it is impossible to guarantee that all pieces in Theorem~\ref{th:main} have positive measure in general. For example, if $A$ contains a point $\vvec{x}$ such that the distance from $\vvec{x}$ to $A\setminus\{\vvec{x}\}$ is greater than the diameter of $B$, then every equidecomposition of $A$ to $B$ must include the  set $\{\vvec{x}\}$ as a single piece. However, the extra assumption of Part~\ref{it:c} applies in many natural cases: for example, it holds for a ball and a cube in any dimension (with arbitrary subsets of their boundaries removed), or for any open sets $A$ and $B$. 

It seems difficult to weaken the assumptions of Theorem~\ref{th:main} in some substantial way. Laczkovich~\cite{Laczkovich93}*{Corollary~3.5} showed that, for every $k\ge 1$, there is a bounded subset $A\subseteq \I R^k$ which is a countable union of pairwise disjoint convergent cubes but which is not equidecomposable to a single cube by translations. In particular, even in Theorem~\ref{th:L}, one cannot replace the assumption that $\boxdim(\partial A),\boxdim(\partial B)<k$ by requiring that, for example, the Hausdorff dimension of the boundaries is at most $k-1$. Another family that refutes various extensions of Theorem~\ref{th:L} (and thus of Theorem~\ref{th:main}) comes from the work of Laczkovich~\cite{Laczkovich03}*{Theorem~3} who showed that for every $k\ge 2$ there are continuum many \emph{Jordan domains} (i.e.\ homeomorphic images of the closed ball), each of volume 1 with everywhere differentiable boundary, so that no two are equidecomposable using any amenable subgroup of isometries. (In particular, this applies to the group $\I R^k$ of translations for any $k$ and the full group of isometries of $\I R^2$, which are amenable.)

On the other hand, if $k\ge 3$ and one allows all orientation-preserving isometries of $\I R^k$, then the obvious necessary conditions for a set to be equidecomposable to a cube using Lebesgue (resp.\ Baire) measurable pieces turn out to be sufficient, see Grabowski, M\'ath\'e and Pikhurko~\cite{GrabowskiMathePikhurko22}*{Corollary~1.10}. However, nothing like this is known for Borel and Jordan measurable equidecompositions. 

\subsection{Historical Background}

Tarski's Circle Squaring Problem has its roots in the theory of paradoxical decompositions. The most famous result in this area is undoubtedly the theorem of Banach and Tarski~\cite{BanachTarski24} that, if $k\ge3$ and $A,B\subseteq \mathbbm{R}^k$ are bounded and have non-empty interior, then $A$ and $B$ are equidecomposable. As a special case, one obtains the striking Banach--Tarski Paradox that a solid ball in $\mathbbm{R}^3$ admits an equidecomposition to two disjoint copies of itself. 

The assumption $k\ge3$ is necessary in the Banach--Tarski Theorem, as Banach~\cite{Banach23} proved that the Lebesgue measure on $\mathbbm{R}$ or $\mathbbm{R}^2$ can be extended to a finitely additive measure on all subsets which is invariant under isometries.
Thus, in $\mathbbm{R}^2$, there cannot exist equidecomposable sets of different Lebesgue measure. The theory of amenable groups, pioneered by von Neumann~\cite{Neumann29}, originated as an attempt to obtain a deeper group-theoretic understanding of the change in behaviour between dimensions two and three; the key difference turns out to be that the group of isometries of $\mathbbm{R}^k$ is amenable for $k\in\{1,2\}$ but not for $k\ge 3$. Generally, if $A$ and $B$ are Lebesgue measurable subsets of $\mathbbm{R}^k$ which are equidecomposable using an amenable subgroup of isometries, then they must have the same measure. Note that the group of translations of $\mathbbm{R}^k$ is amenable for all $k\ge1$, and so the condition that $\lambda(A)=\lambda(B)$ is necessary in Theorem~\ref{th:L} in general.

Decades before Laczkovich squared the circle, Dubins, Hirsch and Karush~\cite{DubinsHirschKarush63} proved that a disk is not \emph{scissor-congruent} to a square, meaning that there does not exist a squaring of the circle with pieces that are interior-disjoint Jordan domains, even if boundaries are ignored. This is in strong contrast to the case of polygons: the classical Wallace--Bolyai--Gerwein Theorem states that it is possible to dissect a polygon into finitely many pieces using straight lines and, ignoring boundaries, reassemble them to form any polygon of the same area; see e.g.,~\cite{TomkowiczWagon:btp}*{pp.~34--35}. Another related result, by Gardner~\cite{Gardner85}, asserts that there is no solution to the Circle Squaring Problem using isometries from a locally discrete subgroup of isometries.

For more background, we refer the reader to the monograph of Tomkowicz and Wa\-gon~\cite{TomkowiczWagon:btp}, whose Chapter 9 is dedicated to circle squaring.

\subsection{Some Ideas Behind the Proof of Theorem~\ref{th:main}}
\label{sec:ideas}

Let us give a very high-level outline of the proof of Theorem~\ref{th:main}; all formal definitions will appear later in the paper. A more detailed sketch of the partial result that all pieces can be made Jordan measurable is given in Section~\ref{sec:JSec}.

Like in the previous work, we assume that $A$ and $B$ are subsets of the torus $\I T^k:=\I R^k/\I Z^k$ and do all translations modulo $1$. We pick a suitable (somewhat large) integer $d$ and vectors $\vvec{x}_1,\ldots,\vvec{x}_d\in\I T^k$ satisfying certain conditions (that are satisfied with positive probability by random vectors). Let $G_d$ be the graph on $\I T^k$ where we connect $\vvec{u}$ to $\vvec{u}+\sum_{i=1}^d n_i \vvec{x}_i$ for each non-zero $(n_1,\ldots,n_d)\in\{-1,0,1\}^d$. Assuming that $\vvec{x}_1,\ldots,\vvec{x}_d$ do not satisfy any linear dependencies with rational coefficients, each component of $G_d$ is a $(3^d-1)$-regular graph on a copy of~$\I Z^d$. 

Our aim is to ``construct'' a bijection from $A$ to $B$ such that, for some constant $r$, each element of $A$ is moved by the bijection by distance at most $r$ within the graph $G_d$. Such a bijection naturally gives an equidecomposition between $A$ and $B$ that uses at most $(2r+1)^d$ pieces. As was observed by Marks and Unger~\cite{MarksUnger17}, the problem of finding such a bijection can be reduced to finding a  uniformly bounded integer-valued flow within the graph $G_d$, where the demand is 1 on $A$, $-1$ on $B$ and 0 elsewhere (see Lemma~\ref{outline:lem:matching} here). 

As one of the first steps of their proof, Marks and Unger~\cite{MarksUnger17} constructed a real-valued (i.e. not necessarily integer-valued) flow $f_\infty$ which satisfies this demand; see Lemma~\ref{outline:lem:realFlows}. The flow $f_\infty$ is defined to be a pointwise limit of a sequence of flows $f_m$ that are locally constructed from $A$ and $B$. Since the collection of subsets of $\I T^k$ with boundary of upper Minkowski dimension at most $k-\zeta$ is not a $\sigma$-algebra, we should not use the values of $f_\infty$ if we want to produce pieces with this structure. Instead, we work with the locally defined approximations $f_m$.

For flow rounding (that is, making all flow values integral), we construct Jordan measurable subsets $J_1,J_2,\ldots$ of $\I T^k$ such that their union $\bigcup_{i=1}^\infty J_i$ is co-null in $\I T^k$ and  $(J_i)_{i=1}^\infty$ is a \emph{toast sequence} (see Definition~\ref{def:scaff}), roughly meaning that each $J_i$ induces only finite (in fact, uniformly bounded) components in $G_d$ and the graph boundaries of all components arising this way are well separated from each other. In fact, each set $J_i$ is a finite union of \emph{strips}, i.e., sets of the form $[a,b)\times[0,1)^{k-1}$; in particular, it is Borel and its boundary is $(k-1)$-dimensional. The idea of using toast sequences to construct satisfying assignments was previously applied to many problems in descriptive combinatorics (with the exact definition of ``toast sequence'' often being problem-specific). For a systematic treatment of this idea for general actions of $\I Z^d$, we refer the reader to Greb\'\i k and Rozho\v n~\cite{GrebikRozhon23a}.

We can view the toast sequence $(J_1,J_2,\ldots)$ as a process where, at time $i$, vertices of the set $J_i$ arrive and our algorithm has to decide the value of the final integer flow $f$ on every edge with at least one vertex in this set. We are not allowed to look into the future nor modify any already defined values of the flow~$f$. We prove that, if all things are set up carefully, then this is indeed possible to do and, in fact, there are some constants $m_i$ and $R_i$ such that the value of $f$ on any edge $xy\in E(G_d)$ intersecting $J_i$ can be computed only from the current picture in the $R_i$-neighbourhood of $\{x,y\}$ in $G_d$ and the values of the approximation $f_{m_i}$ of $f_\infty$ there. 
Here, a key challenge is that, when we round the flow on $J_i$, we have only incomplete information (namely, the flow $f_{m_i}$ which meets the demands only within some small error). The idea that allows us to overcome this difficulty is that, if the cumulative error of $f_{m_i}$ on each component of $J_i$ is small, then whenever our algorithm encounters some inconsistency, it can round it to the nearest integer and produce values that are in fact perfectly compatible with all past and future choices of the algorithm.

The proof coming from the above arguments, with a careful choice of how the size of the components of $J_i$ can grow with $i$, produces a partial equidecomposition between $A$ and $B$ so that the topological boundaries of the pieces as well as the unmatched part of $A$ and of $B$ have upper Minkowski dimension less than $k$. Thus if we can extend this equidecomposition to all of $A$ and $B$ (even by using the Axiom of Choice), then we can achieve Part~\ref{it:a} of Theorem~\ref{th:main}. 

However, these ideas do not seem to be enough to yield pieces which satisfy Parts~\ref{it:a} and~\ref{it:b} of the theorem simultaneously. If $A$ and $B$ are Borel, then all pieces in the partial equidecomposition are Borel, but we did not see a way to match the remaining parts of $A$ and $B$ in a Borel way by using the results of Marks and Unger~\cite{MarksUnger17} as a ``black box''. 
So, in order to achieve Part~\ref{it:b} of Theorem~\ref{th:main},  we essentially run the proof from~\cite{MarksUnger17} on the complement of $\bigcup_{i=1}^\infty J_i$, making sure that, when we define the sets $J_i$ in the first place, we leave enough  remaining ``wiggle'' space. Running the Borel proof on the remaining set 
is probably the most technical part of this paper. This extra work is worth the effort though, since our modification of both the construction and analysis of Marks and Unger~\cite{MarksUnger17} also allows to reduce the Borel complexity of the obtained pieces. 

For Part~\ref{it:c}, we choose  $\vvec{x}_1,\ldots,\vvec{x}_d\in\I T^k$
so that, additionally, for every vector $\vvec{t}$ which is an integer linear combination of them, the intersection $(A+\vvec{t})\cap B$ is null only if it is empty.
Recall that our equidecomposition translates each
$\vvec{u}\in A$ by a vector of the form $\sum_{i=1}^dn_i\vvec{x}_i$ where $n_1,n_2,\ldots,n_d$ are integers between $-r$ and $r$, for some~$r$. Denote  by $\vvec{t}_1,\ldots,\vvec{t}_N$ the vectors $\vvec{t}$ of this form for which $(A+\vvec{t})\cap B$ is non-empty (and thus non-null). 
To achieve Part~\ref{it:c}, for each $1\leq i\leq N$, we pre-select a small non-null subset $A_{\vvec{t}_i}$ of $(B-\vvec{t}_i)\cap A$ to be a part of the piece of the equidecomposition that is translated according to $\vvec{t}_i$. 
This is done in such a way that the sets $A_{\vvec{t}_1}+\vvec{t}_1,\ldots,A_{\vvec{t}_N}+\vvec{t}_N$ are pairwise disjoint and our proof of Parts~\ref{it:a} and~\ref{it:b} still applies to the remaining sets $A\setminus \bigcup_{i=1}^NA_{\vvec{t}_i}$ and $B\setminus \bigcup_{i=1}^N(A_{\vvec{t}_i}+\vvec{t}_i)$, using the same translation vectors $\vvec{t}_1,\ldots,\vvec{t}_N$. Thus, all pieces are non-null by construction.

\subsection{Organisation}
\label{subsec:organisation}

As Theorem~\ref{th:circleSquare} is a direct consequence of Theorem~\ref{th:main} and the bound on $\zeta$ in~\eqref{eq:zeta}, all our focus will be on proving Theorem~\ref{th:main}.

In Section~\ref{sec:prelim}, we present some notation, introduce a few of the tools used in the paper
and prove some auxiliary results. In particular, we present the reduction of Marks and Unger~\cite{MarksUnger17} that it is enough to construct a bounded integer-valued flow $f$ from $A$ to $B$ in $G_d$ with appropriate properties (Lemma~\ref{outline:lem:matching}). This lemma will also take care of Part~\ref{it:c} of the theorem. 

Given these preliminaries, we can give a more detailed outline of the proof of Theorem~\ref{th:main}  in Section~\ref{sec:JSec}. There, we will mostly concentrate  on the (less technical) special case of making the pieces just Jordan measurable.
We will conclude that section by discussing some of the additional ideas which allow us to control the dimension of the boundaries and the Borel complexity of the pieces. 

Then we turn our attention to proving Theorem~\ref{th:main} in earnest. In Section~\ref{sec:realFlows}, we describe the procedure from~\cite{MarksUnger17} for constructing a real-valued flow from $A$ to $B$ in the graph $G_d$. In Section~\ref{sec:covering}, we construct the toast sequences that are needed in the proof.  In Section~\ref{sec:integerFlow} we show how to transform a converging sequence of real-valued flows into a bounded integer-valued flow using a toast sequence. In Section~\ref{sec:pieces}, we define our final flow $f$ (that produces the required equidecomposition via Lemma~\ref{outline:lem:matching}) and
analyse the upper Minkowski dimension of the boundaries of the pieces (in Section~\ref{sec:Dim}) as well as the Borel complexity of the pieces (in Section~\ref{sec:BorelComplexity}).

Even though many of our auxiliary results can be strengthened or generalised in various ways, we try to present just the simplest versions that suffice for proving Theorem~\ref{th:main}, since its proof is already very intricate and delicate. In fact, we sometimes present a separate (sketch) proof of a special case before giving the general proof, when we believe that this improves readability.


\section{Preliminaries}
\label{sec:prelim}

We denote by $\I R$ and $\I Z$ the sets of reals and integers, respectively. For $x\in\I R$, we define its \emph{nearest integer} $[x]$ to be equal to $\lfloor x\rfloor$ if $x-\lfloor x\rfloor <1/2$ and $\lceil x\rceil$ otherwise. The set of all values assumed by a function $f$ is denoted by $\range(f)$. When we write e.g.,\ $(a_i,b_i)_{i=1}^\infty$, we mean the sequence $(a_1,b_1,a_2,b_2,\ldots)$. The indicator function of a set $X$ is denoted by~$\ind_X$. By $\log$ we mean the natural logarithm.

\subsection{Discrepancy Bounds for the Torus}
\label{subsec:disc}


The \emph{$k$-torus} $\mathbbm{T}^k$ is the quotient group of $(\mathbbm{R}^k,+)$ modulo $\mathbbm{Z}^k$.
When we write the sum of $\vvec{u},\vvec{v}\in\I T^k$, we take it modulo 1 and thus $\vvec{u}+\vvec{v}$ is an element of $\I T^k$. In particular, the translation of $Y\subseteq \mathbb{T}^k$ by a vector $\vvec{t}\in\I T^k$,
 $$
  Y+\vvec{t}:=\left\{\vvec{y}+\vvec{t}: \vvec{y}\in Y\right\}\subseteq \I T^k,
  $$
  is  done modulo 1. Similarly,  scalar multiples of vectors in $\I T^k$ are also taken modulo 1.
  
For notational convenience, we may occasionally identify $\mathbb{T}^k$ with~$[0,1)^k$. While this identification is not a topological homeomorphism, it does preserve the set families that are of interest to us, namely, sets that are Borel, are Jordan measurable, have boundary of upper Minkowski dimension at most $k-\zeta$, etc.
  The restriction of the Lebesgue measure on~$\mathbbm{R}^k$ to the torus will be denoted by the same symbol~$\lambda$. We have $\lambda(\I T^k)=1$.

Given a finite set $F\subseteq \mathbbm{T}^k$ and a Lebesgue measurable set $X\subseteq \mathbbm{T}^k$, the \emph{discrepancy of $F$ relative to $X$} is defined to be
\begin{equation}\label{eq:discrepDef}D(F,X):= \big|\,|F\cap X| - |F|\cdot\lambda(X)\,\big|.\end{equation}
In other words, $D(F,X)$ is the deviation between $|F\cap X|$ and the expected size of this intersection if $F$ were a uniformly random subset of $\mathbbm{T}^k$ of cardinality $|F|$.

 A key tool used in this paper, as well as in~\cites{GrabowskiMathePikhurko17,MarksUnger17}, is the following discrepancy lemma of Laczkovich~\cite{Laczkovich92b}, whose proof sketch can be found in
 Appendix~\ref{app:discrep}.
 
\begin{lemma}[Laczkovich~{\cite{Laczkovich92b}*{Proof of Theorem~3}}; see also~{\cite{GrabowskiMathePikhurko17}*{Lemma~2.4}}]
\label{lem:discrep}
Let $X$ be a measurable subset of $\mathbbm{T}^k$ such that $k-1\le \boxdim(\partial X)<k$, let $d$ be a positive integer such that 
\[d>\frac{k}{k-\boxdim(\partial X)}\]
and let $\varepsilon\in\mathbbm{R}$ be such that
\[0 <\varepsilon<\frac{d(k-\boxdim(\partial X))-k}{k}.\]
If $\vvec{x}_1,\ldots,\vvec{x}_d$ are chosen uniformly at random from $\mathbbm{T}^k$ and independently of one another then, with probability $1$, there exists $c>0$ such that,  for every $\vvec{u}\in\mathbbm{T}^k$ and every integer $n\ge0$, we have
\[D\left(N_{n}^+[\vvec{u}],X\right) \leq c\cdot (n+1)^{d-1-\varepsilon},\]
 where we define 
 \beq{eq:Nn+}
 N_n^+[\vvec{u}]:=\left\{\vvec{u}+ \sum_{i=1}^dn_i\vvec{x}_i: (n_1,\ldots,n_d)\in\{0,\ldots,n\}^d\right\},
 \eeq
 calling it a \emph{discrete $(n+1)$-cube} (or just a \emph{discrete cube}).
\end{lemma} 

 Essentially, this lemma says that, if a set $X\subseteq\mathbbm{T}^k$ has ``small'' topological boundary, then, for $d$ sufficiently large, the discrepancy of any large discrete cube given by typical vectors $\vvec{x}_1,\ldots,\vvec{x}_d\in\mathbb{T}^k$ with respect to $X$ is significantly smaller than the number of points on the combinatorial boundary of the cube.

\subsection{Graph-Theoretic Definitions}

A (simple undirected) \emph{graph} is a pair $G=(V,E)$ where the elements of $V$ are called \emph{vertices} and $E$ is a collection of unordered pairs $\{u,v\}$ of vertices called \emph{edges}. For brevity, an edge $\{u,v\}\in E$ is written $uv$ or, equivalently, $vu$.  The vertex set and edge set of a graph $G$ are denoted by $V(G)$ and $E(G)$, respectively. 

A vertex $u$ is said to be \emph{adjacent to} (or \emph{a neighbour of}) a vertex $v$ if $uv\in E(G)$. Given a set $S\subseteq V(G)$, the \emph{subgraph of $G$ induced by $S$}, denoted as $G\induced S$, is the graph with vertex set $S$ and edge set $\{uv\in E(G): u,v\in S\}$.

Given a graph $G$ and $u,v\in V(G)$, we let $\dist_G(u,v)$ denote the \emph{graph distance from $u$ to $v$ in $G$}, i.e., the fewest number of edges in a path from $u$ to $v$ in $G$. If no such path exists, then $\dist_G(u,v):=\infty$. Also, note that $\dist_G(u,u)=0$ for every vertex~$u$. For sets $S,T\subseteq V(G)$, we let 
$$
 \dist_{G}(S,T):=\min\{\dist_G(u,v): u\in S\text{ and }v\in T\}
 $$ and, for $w\in V(G)$, we write $\dist_G(w,T)$ to mean $\dist_G(\{w\},T)$. 
Given $u\in V(G)$, the \emph{connected component of $G$ containing $u$} is the set 
$$[u]_G:=\{v\in V(G): \dist_G(u,v)<\infty\}.
$$ 
We say that $G$ is \emph{connected} if $[u]_G=V(G)$ for every (equivalently, some) vertex $u\in V(G)$. 

Given a graph $G$ and $u\in V(G)$, let 
$$N_G[u]:=\{v\in V(G): \dist_G(u,v)\leq 1\}
$$ denote the \emph{(closed) neighbourhood of $u$ in $G$}. Also, for a set $S\subseteq V(G)$, let $N_G[S]:=\bigcup_{u\in S}N_G[u]$ be the \emph{(closed) neighbourhood of $S$ in $G$}. The \emph{degree} of a vertex $u\in V(G)$ is defined to be 
$$\deg_G(u):=\left|N_G[u]\setminus \{u\}\right|=\left|N_G[u]\right|-1.
$$ 
We say that $G$ is \emph{locally finite} if $\deg_G(u)$ is finite for every $u\in V(G)$. 
For $d\in \mathbbm{N}$, we say that $G$ is \emph{$d$-regular} if $\deg_G(u)=d$ for all $u\in V(G)$.

The \emph{edge boundary} of $S\subseteq V(G)$ in $G$ is the set 
\begin{equation}
\label{eq:DefPartialE}
 \partial_ES:=\{uv\in E(G): u\in S\text{ and }v\in V(G)\setminus S\}.
 \end{equation}


\subsection{Network Flows}

A \emph{flow} in a graph $G$ is a function $f:V(G)\times V(G)\to \mathbbm{R}$ such that
\[f(u,v)=-f(v,u)\text{ for all }u,v\in V(G), \text{ and}\]
\[f(u,v)=0\text{ if }uv\notin E(G).\]
The quantity $f(u,v)$ is called the \emph{flow from $u$ to $v$ under $f$}. Given a finite set $S\subseteq V(G)$, the \emph{flow out of $S$ under $f$} is defined to be
\[\fout{f}(S):= \sum_{\substack{u\in S\\ v\in V(G)\setminus S}}f(u,v).\]
Note that, in this paper, we will deal with locally finite graphs only, and so the flow out of a finite set will always be well-defined. For a vertex $u\in V(G)$, the \emph{flow out of $u$ under $f$} is $\fout{f}(u):=\fout{f}(\{u\})$. The following is an easy consequence of the definition of a flow.
\begin{obs}
\label{obs:flowOutOfS}
Given a locally finite graph $G$, a flow $f$ in $G$ and a finite set $S\subseteq V(G)$, it holds that
$\fout{f}(S) = \sum_{u\in S}\fout{f}(u)$.\qed
\end{obs}
When it is clear from the context, we may view $\fout{f}$ just as a function $V(G)\to\I R$; for example, 
when we write $\|\fout{f} - g\|_\infty$ for some $g: V(G)\to\I R$, we mean the supremum of $|\fout{f}(u) - g(u)|$ over $u\in V(G)$.

 Given a function $\chi:V(G)\to \mathbbm{R}$, which we call a \emph{demand function}, a \emph{$\chi$-flow in $G$} is a flow $f$ such that $\fout{f}(u)=\chi(u)$ for every $u\in V(G)$. For $S,T\subseteq V(G)$, a \emph{flow from $S$ to $T$ in $G$} is a $\left(\ind_S-\ind_T\right)$-flow in $G$.

It is well-known that, if $\chi$ is an integer-valued demand function for a finite graph $G$, then a real-valued $\chi$-flow in $G$ can be converted into an integer-valued $\chi$-flow in $G$ by changing the flow along each edge by less than~$1$. This is known as the Integral Flow Theorem. It seems to have been first observed by Dantzig and Fulkerson~\cite{DantzigFulkerson56} using ideas of Dantzig~\cite{Dantzig51}; see~\cite{Schrijver03co1}*{Corollary~10.3a} and the discussion in~\cite{Schrijver03co1}*{p.~64}. The following version can be derived from the finite case via a standard compactness (i.e., Axiom of Choice) argument; for a proof see e.g.,~\cite{MarksUnger17}*{Corollary~5.2}.

\begin{theorem}[Integral Flow Theorem]
\label{th:IFT}
Let $G$ be a locally finite graph and let $\chi:V(G)\to \mathbbm{Z}$ be an integer-valued demand function. Then, for every $\chi$-flow $g$ in $G$, there exists an integer-valued $\chi$-flow $f$ in $G$ with $|f(u,v)-g(u,v)|<1$ for all $u,v\in V(G)$.\qed
\end{theorem}


 \subsection{The Setting}
 \label{subsec:setting}

At this point, we are ready to make some key definitions and assumptions that will apply throughout the rest of the paper. 

Let $k\ge1$ be integer, and let $A$ and $B$ be sets that satisfy the assumptions of Theorem~\ref{th:main}. Thus $A,B\subseteq \mathbbm{R}^k$ are bounded sets that have the same positive Lebesgue measure (i.e., $\lambda(A)=\lambda(B)>0$) and boundary of the upper Minkowski dimension less than $k$ (i.e., $\boxdim(\partial A)< k$ and $\boxdim(\partial B)< k$). 

By scaling $A$ and $B$ by the same factor, we
can assume that each of them has diameter less than $1/2$ in the $\ell_\infty$-norm. Furthermore, by translating them, we can assume that $A$ and $B$ are disjoint subsets of $[0,1)^k$.  If $A$ and $B$ are equidecomposable using translations inside $\mathbbm{T}^k$, then they are also equidecomposable using translations in $\mathbbm{R}^k$ with exactly the same pieces (by the diameter assumption). From now on, we always assume that we are working in the setting of the $k$-torus.

In order to satisfy Lemma~\ref{lem:discrep} and to optimise the bound on the upper Minkowski dimension of the boundaries of the final pieces, we make the following assignments. First, fix $\epsilon$ so that
 \begin{equation}
 \label{eq:epsilonGap}
 0<\epsilon< k-\max\left\{\boxdim(\partial A),\boxdim(\partial B)\right\},
 \end{equation} 
 where we should choose $\epsilon$ close to the upper bound for optimality. Then
  we set
\begin{equation}
\label{eq:dBound}
d:= \left\lfloor k/\epsilon\right\rfloor+1,
\end{equation}
that is, $d$ is the smallest integer greater than $k/\epsilon$. We also define
\begin{equation}
\label{eq:epsBound}
\varepsilon:= (d\epsilon -k)/k.
\end{equation}

Note that, if the boundaries of $A$ and $B$ have upper Minkowski dimension $k-1$, as is the case for most simple ``geometric'' sets like balls and cubes, then we can take $d=k+1$ and let $\epsilon$ and $\varepsilon$ be close to $1$ and $1/k$, respectively.


For each integer $i\ge0$, define
\begin{equation}\label{eq:ri'specific}r_{i}':=100^{2^{i+1} - 1}\end{equation}
and, for each integer $i\ge 1$, define
\begin{equation}\label{eq:rispecific}r_i:=100^{2^{i+1} - 2}.\end{equation}

Next, we recursively define $q_0'<q_1<q_1'<q_2<\ldots$ and $t_1<t_1'<t_2<t_2'<\ldots$ by setting $q_0':=0$ and, for each $i\ge1$,
\begin{eqnarray}
t_i&:=&2r_i+4q'_{i-1}+4,\nonumber\\
q_i&:=&t_i+2q'_{i-1}+4,\nonumber\\
t'_i&:=&4r'_i/5 + 2q_i,\nonumber\\
q'_i&:=&t'_i+2q_i+4.\label{eq:qi'}
\end{eqnarray}

We fix, for the rest of the paper, some vectors $\vvec{x}_1,\ldots,\vvec{x}_d\in \mathbbm{T}^k$ that satisfy all of the following properties:
 \begin{enumerate}
 \stepcounter{equation}\item\label{eq:XsDiscr} 
  the conclusion of Lemma~\ref{lem:discrep}, with respect to the above $d$ and $\varepsilon$, holds for some $c>0$ (which is also fixed 
 throughout the paper) for both $X:=A$ and $X:=B$;
 \stepcounter{equation}\item\label{eq:XsLinInd} 
 the projections of $\vvec{x}_{1},\ldots,\vvec{x}_{d}$ on the first coordinate are linearly independent over the rationals;
 \stepcounter{equation}
 \item\label{eq:XsAvoids}  if the extra assumption of Part~\ref{it:c} of Theorem~\ref{th:main} (that is,~\eqref{eq:it:c}) holds, then no integer combination of $\vvec{x}_1,\ldots,\vvec{x}_d$ belongs to the ``bad'' set 
 $$
 \{\vvec{t}\in\I T^k\mid (A+\vvec{t})\cap B\mbox{ is non-empty and $\lambda$-null}\};
 $$
 \stepcounter{equation}\item\label{eq:XQi} 
 the conclusion of Lemma~\ref{lem:stripLem}
 (to be stated in Section~\ref{sec:Dim}) holds.
 \end{enumerate}
 This is possible since, if $\vvec{x}_1,\ldots,\vvec{x}_d$ are chosen uniformly at random from $\mathbbm{T}^k$ independently of one another, then each of Properties~\eqref{eq:XsDiscr}--\eqref{eq:XsAvoids} holds with probability $1$, while Lemma~\ref{lem:stripLem} is satisfied with positive probability. For example, Property \eqref{eq:XsAvoids} holds almost surely
  since, for each non-zero $(n_1,\ldots,n_d)\in\I Z^d$, the vector $\sum_{i=1}^d n_i\vvec{x}_i$ is a uniform element of $\I T^k$ while the ``bad'' set  has measure 0 by~\eqref{eq:it:c};
 also, the zero vector is ``good''
  by our assumption that $A\cap B=\emptyset$.

\begin{remark}\label{rm:FastRi}
If the reader is interested only in a version of Theorem~\ref{th:main} where the conclusion of Part~\ref{it:a} is weakened to requiring that the pieces are just Jordan measurable, then we could have chosen an arbitrary sufficiently fast growing sequence $r_0'\ll r_1\ll r_1'\ll r_2\ll\ldots$ and ignored Property~\eqref{eq:XQi}, which would not be needed. This is why the (somewhat technical) statement of Lemma~\ref{lem:stripLem} is postponed to Section~\ref{sec:Dim}, where the dimension of the boundaries of the obtained pieces is analysed.
\end{remark}

We denote elements of $\mathbbm{Z}^d$ as vectors accented by an arrow, e.g.,\ $\vec{n}$, to distinguish them from vectors in $\mathbbm{T}^k$ which are typeset in boldface. For $\vec{n}\in\mathbbm{Z}^d$ and $1\leq i\leq d$, let $n_i$ denote the $i$-th coordinate of $\vec{n}$. 
Consider the action $a:\mathbbm{Z}^d\actson \mathbbm{T}^k$ defined by
\[
\vec{n}\cdot_a \vvec{u} := \vvec{u} + \sum_{i=1}^dn_i\vvec{x}_i,\quad\mbox{for  $\vec{n}\in \mathbbm{Z}^d$ and $\vvec{u}\in\mathbbm{T}^k$.}
\]
 This action is continuous and preserves the measure~$\lambda$. 

Next, we recall the definition of the graph $G_d$ from Subsection~\ref{sec:ideas}. Namely, the vertex set of $G_d$ is $\mathbbm{T}^k$
and the neighbours of $\vvec{u}\in\mathbbm{T}^k$ are exactly the vectors $\vvec{u}+\sum_{i=1}^dn_i\vvec{x}_i$
for non-zero $\vec{n}\in \{-1,0,1\}^d$. Equivalently, $G_d$ is the $(3^d-1)$-regular Schreier graph associated to the action $a$ with respect to the symmetric set $\{\vec{n}\in\mathbbm{Z}^d: \|\vec{n}\|_\infty=1\}$.   
The edge set of $G_d$, when viewed as a symmetric binary relation on the (standard Borel) space $\mathbb{T}^k$, is a closed and thus Borel set. This means that $G_d$ is a \emph{Borel graph}, a fundamental object of study in \emph{descriptive graph combinatorics} (see e.g.,\ the survey of this field by Kechris and Marks~\cite{KechrisMarks20survey}).
Note that the action $a$ is free
by our choice of the vectors $\vvec{x}_i$, namely by~\eqref{eq:XsLinInd}. Therefore, the subgraph of $G_d$ induced by any component of $G_d$ is nothing else than a copy of $\mathbbm{Z}^d$ in which two elements are adjacent if they are at $\ell_\infty$-distance~$1$. 

We pause to remind the reader of the role of the graph $G_d$ in the equidecomposition. 

\begin{remark}
\label{rem:GdRole}
The translation vectors used in the constructed equidecomposition will be of the form $\sum_{i=1}^dn_i\vvec{x}_i$ where $\vec{n}\in \mathbbm{Z}^d$ and $\|\vec{n}\|_\infty$ is bounded by a large constant $r$ depending only on $A$ and $B$ (which we do not attempt to estimate). Then an equidecomposition is equivalent to a bijection from $A$ to $B$ where each element of $A$ is matched to an element of $B$ at distance at most $r$ from it in $G_d$. Given such a matching, each piece of the  equidecomposition is indexed by an integer vector in $\{-r,\ldots,r\}^d$,
where the piece corresponding to $\vec{n}$ is the set of all $\vvec{u}\in A$ such that $\vvec{u}$ is matched to $\vvec{u}+\sum_{i=1}^dn_i\vvec{x}_i \in B$.  
\end{remark}

\subsection{Some Auxiliary Structures in $G_d$}\label{se:aux}

Here we present some useful building blocks and definitions.
Note that, once $d$ and $\vvec{x}_1,\ldots,\vvec{x}_d$ are fixed,
the definitions and constructions in this section do not depend on the sets $A$ and~$B$.

At certain points during the construction, we have to make arbitrary ``tie-breaking'' choices in a consistent manner. In doing so, it is useful to endow $\I T^k$ with a partial order in which each component of $G_d$ is totally ordered. We will use the \emph{lexicographic ordering} where, for distinct $\vvec{u},\vvec{v}\in \mathbbm{T}^k$, we write $\vvec{u}\lex\vvec{v}$ to mean $[\vvec{u}]_{G_d}=[\vvec{v}]_{G_d}$ and the first non-zero entry of $\vec{n}$ is positive, where $\vec{n}$ is the (unique) element of $\I Z^d$ with
\[\vvec{v}-\vvec{u} = \sum_{i=1}^dn_i\vvec{x}_i.\]
This naturally extends to a partial order on $\left(\mathbbm{T}^k\right)^t$ for any $t\ge1$ where $(\vvec{u}_1,\ldots,\vvec{u}_t)\lex (\vvec{v}_1,\ldots,\vvec{v}_t)$ if there exists $1\leq i\leq t$ such that $\vvec{u}_i\lex \vvec{v}_i$ and $\vvec{u}_j=\vvec{v}_j$ for all $1\leq j\leq i-1$. 

Given $\vvec{u}\in\mathbbm{T}^k$, we write $N_{G_d}[\vvec{u}]$ simply as $N[\vvec{u}]$. For $n\ge1$, define the \emph{$n$-neighbourhood} of $\vvec{u}$ to be
\[N_n[\vvec{u}]:=\left\{\vvec{v}\in \mathbbm{T}^k: \dist_{G_d}(\vvec{u},\vvec{v})\leq n\right\}.\]
 Note that $N_n[\vvec{u}]=N_{2n}^+[\vvec{v}]$, where $\vvec{v}:=\vvec{u}-n(\vvec{x}_1+\ldots+\vvec{x}_d)$ and $N_r^+[\vvec{v}]$ is defined as in~\eqref{eq:Nn+}. 
 Thus, $N_n[\vvec{u}]$ (resp.\ $N_n^+[\vvec{u}]$) is the set of vertices which can be reached from $\vvec{u}$ by taking at most $n$ steps in any (resp.\ ``completely non-negative'') directions in~$G_d$.
Given a set $S\subseteq\mathbbm{T}^k$ and an integer $n\ge1$, let 
$$
N_n[S]:=\bigcup_{\vvec{u}\in S}N_n[\vvec{u}]\quad\mbox{and}\quad N_n^+[S]:=\bigcup_{\vvec{u}\in S}N_n^+[\vvec{u}].
$$ 

We say that a set $Y\subseteq\mathbb{T}^k$ \emph{$r$-locally depends} on (or is an \emph{$r$-local function} of) $X_1,\ldots,X_m\subseteq \mathbb{T}^k$ if the inclusion of $\vvec{u}\in\mathbb{T}^k$ in $Y$ depends only on the intersection of $X_1,\ldots,X_m$ with~$N_r[\vvec{u}]$, by which we mean that it depends only on the sequence 
$$
\left(\left\{\vec{n}\in\I Z^d\mid \|\vec{n}\|_\infty\le r,\, \vvec{u}+n_1\vvec{x_1}+\ldots+n_d\vvec{x_d}\in X_j\right\}
\right)_{j=1}^m.
$$
 This is equivalent to $Y$ being some Boolean combination of the sets of form $X_j+n_1\vvec{x}_1+\ldots+n_d\vvec{x}_d$ for $1\le j\le m$ and $\vec{n}\in\I Z^d$ with $\|\vec{n}\|_\infty\le r$. When the specific value of $r$ is unimportant or clear from context, we simply say that $Y$ \emph{locally depends} on (or is a \emph{local function} of) $X_1,\ldots,X_m$. Here is a trivial but very useful observation.

\begin{obs}\label{obs:LocalRule} If $Y$ locally depends on some finite sequence of sets from an algebra $\C A$ on $\mathbb{T}^k$ which is \emph{invariant} under the action $a$ (that is, $A\pm\vvec{x}_i\in \C A$ for every $A\in\C A$ and $i\in \{1,\ldots,d\}$), then $Y$ also belongs to~$\C A$.\qed\end{obs}
 
Note that the families of subsets of $\I T^k$ that are Jordan measurable or have boundary of upper Minkowski dimension  at most some given constant are both examples of $a$-invariant algebras.

If $f$ is a real-valued flow in $G_d$, then for $\vec{\gamma}\in\mathbbm{Z}^d$ with $\|\vec{\gamma}\|_\infty=1$ and $\ell\in \mathbbm{R}$, we define
\begin{equation}\label{eq:Zdef}
 Z^{f}_{\vec{\gamma},\ell}:=\left\{\vvec{v}\in\mathbbm{T}^k: f\left(\vvec{v},\vec{\gamma}\cdot_a\vvec{v}\right) = \ell\right\}.
\end{equation}

\begin{obs}
\label{obs:Zpartition}
Given a flow $f$ in $G_d$ and $\vec{\gamma}\in \mathbbm{Z}^d$ with $\|\vec{\gamma}\|_\infty=1$, the sets $Z^{f}_{\vec{\gamma},\ell}$ for $\ell\in\I R$ with $|\ell|\leq \|f\|_\infty$ partition $\mathbbm{T}^k$. If the flow $f$ assumes finitely many possible values, then only finitely many of these sets are non-empty. \qed
\end{obs}

A flow $f$ with finite range will be identified with the finite sequence of sets $Z^{f}_{\vec{\gamma},\ell}$ for all possible choices of $\vec{\gamma}$ and $\ell$. Thus we will say that a flow $f$ is an \emph{$r$-local function} of $X_1,\ldots,X_m$ to mean that each set $Z^f_{\vec{\gamma},\ell}$ is an $r$-local function of $X_1,\ldots,X_m$. In the other direction,
$r$-local dependence of a set $Y$ on a flow $f$ means $r$-local dependence on the sequence of the sets $Z^f_{\vec{\gamma},\ell}$ (that is, the membership condition for $\vvec{u}\in\I T^k$ to be an element of $Y$ can be determined from the flow values on the edges intersecting~$N_r[\vvec{u}]$).

\begin{defn}
\label{def:strip}
Define a \emph{strip} to be a subset of $\mathbbm{T}^k$ of the form $[a,b)\times [0,1)^{k-1}$ for some $0\leq a<b\leq 1$. The \emph{width} of the strip $[a,b)\times [0,1)^{k-1}$ is defined to be $b-a$. 
\end{defn}

Note that any translation of a strip by a vector can be written as a union of at most two strips and that any Boolean combination of strips can be written as a union of finitely many disjoint strips. For these reasons, strips are particularly convenient to work with.

\begin{defn}
\label{def:discrete}
A set $X\subseteq \mathbbm{T}^k$ is said to be \emph{$r$-discrete (in $G_d$)} if $\dist_{G_d}(\vvec{x},\vvec{y}) > r$ for any distinct $\vvec{x},\vvec{y}\in X$ and \emph{maximally $r$-discrete (in $G_d)$} if it is maximal under set inclusion with respect to this property. 
\end{defn}

Of course, the above properties are not affected when we replace $X$ by any translate. The following fact is a consequence of~\eqref{eq:XsLinInd}, one of our assumptions on~$\vvec{x}_1,\ldots,\vvec{x}_d$.

\begin{obs}
\label{outline:obs:diamDiscrete}
For any $r\ge1$ there exists $\delta>0$ such that every set $X\subseteq \mathbbm{T}^k$ whose projection onto the first coordinate has diameter at most $\delta$ is $r$-discrete. \qed
\end{obs}

We will also need the following lemma (whose main proof idea goes back to Kechris, Solecki and Todorcevic~\cite{KechrisSoleckiTodorcevic99}*{Proposition~4.2}). 

\begin{lemma}
\label{outline:lem:simpleDiscrete}
For $r\ge0$, let $C_1,\ldots,C_m$ be subsets of $\mathbbm{T}^k$ such that every $C_i$ is $r$-discrete in $G_d$ and $\bigcup_{i=1}^mC_i = \mathbbm{T}^k$. Then there exists a maximally $r$-discrete set $X$ which is an $r(m-1)$-local function of $C_1,\ldots,C_m$.
\end{lemma}

\begin{proof}
Let $C_1':=C_1$ and, for $2\leq i\leq m$, define
\[
C_i':=C_i \setminus N_r\left[\bigcup_{j=1}^{i-1}C_j'\right].
\]
Let $X:=C_1'\cup\ldots \cup C_m'$. By construction, and because the sets $C_1,\ldots,C_m$ are $r$-discrete in $G_d$, the set $X$ is also $r$-discrete in $G_d$. Since the sets $C_1,\ldots,C_m$ cover $\mathbbm{T}^k$, it follows that $X$ is maximally $r$-discrete in $G_d$. 

Since $X=C_1'\cup\ldots \cup C_m'$, it is enough to show that each $C_i'$ is an $r(i-1)$-local function of $C_1,\ldots,C_i$. We use induction on $i=1,\ldots,m$. This is true for $i=1$ since the set $C_1'=C_1$ is a $0$-local function of~$C_1$. For each $2\leq i\leq m$, the set $C_i'$ is a  Boolean combination of $C_i$ and the sets $C_1',\ldots,C_{i-1}'$ translated by vectors of the form $\sum_{i=1}^d n_i\vvec{x}_i$ where $\|\vec{n}\|_\infty\leq r$.  So, by induction, $C_i'$ is an $r(i-1)$-local function of $C_1,\ldots,C_i$, as desired.
\end{proof}

The following corollary is easily derived from Observation~\ref{outline:obs:diamDiscrete} and Lemma~\ref{outline:lem:simpleDiscrete}.

\begin{corollary}
\label{outline:cor:stripDiscrete}
For any $r\ge0$ there exists a set $X\subseteq \mathbbm{T}^k$ which is maximally $r$-discrete in $G_d$ and can be expressed as a union of finitely many disjoint strips in $\mathbbm{T}^k$.\qed 
\end{corollary}

\begin{defn}
For $S\subseteq \mathbbm{T}^k$, let $\comp(S)$ be the collection of all components of $G_d\induced S$. 
\end{defn}

The following definition describes the types of structures that we will use to round a sequence of real-valued flows to an integer-valued one. We use the word ``toast'' which seems to be the standard term for structures of this type now.

\begin{defn}
\label{def:scaff}
We say that a sequence $(D_1,D_2,\ldots)$ of subsets of $\mathbbm{T}^k$ is a \emph{toast sequence (in $G_d$)} if the following three conditions are satisfied for all $i\ge1$:
\begin{enumerate}
\stepcounter{equation}
\item\label{scaffBdd} the elements of $\comp(D_i)$ are finite and have uniformly bounded cardinality,
\stepcounter{equation}
\item\label{scaffDist} any two distinct elements of $\comp(D_i)$ are at distance at least $3$ in $G_d$, and
\stepcounter{equation}
\item\label{scaffUt} for $1\leq j<i$, every $S\in \comp(D_j)$ satisfies either $N_2[S]\subseteq D_i$ or $\dist_{G_d}(S,D_i)\ge 3$.
\end{enumerate}
\end{defn}


Let us make a brief contextual remark. A concept which is ubiquitous in  descriptive graph combinatorics is the notion of hyperfiniteness. Namely, a Borel graph $G$ is said to be \emph{hyperfinite} if its edge set 
can be written as an increasing union of edge sets of Borel graphs with finite components; see, e.g.,\ 
\cite{Gao09idst}*{Section~7.2} or
\cite{KechrisMiller:toe}*{Section~II.6}.
In this context, a Borel toast sequence $(D_1,D_2,\ldots)$
gives a specific type of a hyperfiniteness certificate $(D_1,D_1\cup D_2,\ldots)$  (or, more precisely, the corresponding sequence of the induced edge sets) for the graph $G_d\induced \bigcup_{i=1}^\infty D_i$; see, e.g.,~\cites{GaoJacksonKrohneSeward22,ConleyJacksonMarksSewardTuckerdrob23}. 

\begin{remark}\label{rem:holes} Note that that our definition of a toast sequence allows \emph{holes} in $D_i$, that is, finite components of $G_d\induced (\I T^k\setminus D_i)$. (In fact, it even allows $D_i$ to have two holes at distance 2 in~$G_d$, even though this can be shown not occur in the toasts constructed in Section~\ref{sec:covering}.)
While it should be possible to get rid of all holes by modifying our constructions in Section~\ref{sec:ToastConstructions} (as in the approach taken by Marks and Unger~\cite{MarksUnger17}), we found it easier instead to write our proofs so that they apply to toast sequences with holes.\end{remark}

\subsection{The Setting (Continuation)}
\label{sec:Setting2}
\label{sec:ToastAssumptions}

Here we define two sequences of sets, with these definitions applying throughout the paper. Namely, for each $i\ge1$, we fix two sets $X_i$ and $Y_i$, each of which is a union of finitely many disjoint strips, so that $X_i$ is maximally $r_i$-discrete and $Y_i$ is maximally $r'_i$-discrete. This is possible by Corollary~\ref{outline:cor:stripDiscrete}.

For Part~\ref{it:b} of Theorem~\ref{th:main}, we will need the additional assumption that 
\begin{equation}\label{eq:disjointBoundaries}
 \dist_{G_d}(\partial Y_i,\partial Y_j)=\infty,\quad \mbox{for all $i\neq j$},
\end{equation}
 that is, each component of $G_d$ can intersect the topological boundary of $Y_i$ for at most one value of $i$.
Since each $Y_i$ is a finite union of strips, the projection on the first coordinate of the union of all components of $G_d$ that intersect the topological boundary $\partial Y_i$ is a countable set. Thus, if $Y_1,Y_2,\ldots$ are translated by uniformly random vectors in $\mathbbm{T}^k$ which are independent of one another and of $\vvec{x}_1,\ldots,\vvec{x}_d$, then, almost surely, the translated countable sets inside $\I T^1$ are  pairwise disjoint and~\eqref{eq:disjointBoundaries} holds.

To achieve Part~\ref{it:a} of Theorem~\ref{th:main} with good quantitative bounds on the dimension of the boundaries, we additionally require by our assumption~\eqref{eq:XQi} that the sets $X_i$ satisfy Lemma~\ref{lm:Xi}, which basically requires that $X_i$ is an $r_i^{d+1+o(1)}$-local function of a single strip.
This extra assumption directly gives an upper bound on the number of strips that make up $X_i$; also, it will allow us to analyse the boundary of sets which are local functions depending on some sets~$X_i$.

\subsection{Equidecompositions from Integer-Valued Flows}
\label{subsec:final}

We will need the following lemma, whose main idea is inspired by the proof sketch in~\cite{MarksUnger17}*{Remark~6.2}. In the context of Theorem~\ref{th:main}, this lemma 
shows that it is enough to find a bounded integer-valued flow $f$ from $A$ to $B$ with the desired regularity properties.

\begin{lemma}
\label{outline:lem:matching}
Let $f$ be a bounded integer-valued flow from $A$ to $B$ in $G_d$. Then there is an integer $R$ and an equidecomposition 
between $A$ and $B$ such that each piece is an $R$-local function of $f$, $A$, $B$ and finitely many strips. Moreover, if~\eqref{eq:it:c} (that is, the extra assumption of Part~\ref{it:c} of Theorem~\ref{th:main}) holds, then we can additionally require that each piece contains a Lebesgue measurable subset of positive measure.
\end{lemma}

\begin{proof}
Using Corollary~\ref{outline:cor:stripDiscrete}, let $W_r$, for each integer $r\geq1$,  be a maximally $r$-discrete subset which is a finite union of strips.
For each $\vvec{v}\in \mathbbm{T}^k$, let $\eta_r(\vvec{v})$ be the vertex $\vvec{u}\in W_r$  such that $\dist_{G_d}(\vvec{v},\vvec{u})$ is minimised and, among all vertices of $W_r$ at the minimum distance from $\vvec{v}$, the vertex $\vvec{u}$ comes earliest under $\lex$. For each $r\ge1$ and $\vvec{u}\in W_r$, let
\[V_r(\vvec{u}):=\{\vvec{v}\in\mathbbm{T}^k: \eta_r(\vvec{v})=\vvec{u}\}.\]
The sets $V_r(\vvec{u})$ for $\vvec{u}\in W_r$ clearly partition $\mathbbm{T}^k$. They can be thought of as ``Voronoi cells'' generated by $W_r$ with respect to the graph distance in $G_d$, where ties are broken using $\lex$. Since $W_r$ is a maximal $r$-discrete set, every element of $V_r(\vvec{u})$ is at distance at most $r$ from $\vvec{u}\in W_r$ and thus the diameter of every Voronoi cell is at most~$2r$.

Since $W_r$ is $r$-discrete, we have
\[N_{\lfloor r/2\rfloor}[\vvec{u}]\subseteq V_r(\vvec{u})\]
for every $\vvec{u}\in W_r$. Combining this with Lemma~\ref{lem:discrep}, we see that
\[
\min_{\vvec{u}\in W_r}\min\left\{\,|V_r(\vvec{u})\cap A|,|V_r(\vvec{u})\cap B|\,\right\} = \Omega(r^d)
\]
as $r\to \infty$. It is not hard to argue (see, e.g., the proof of Lemma~\ref{lem:Jmeas}) that $\left|\partial_EV_r(\vvec{u})\right| = O(r^{d-1})$ as $r\to\infty$ where the implicit constant depends on $d$ only. (Recall that $\partial_E$ denotes the edge boundary of a set, as defined in~\eqref{eq:DefPartialE}.) Thus, if $r$ is sufficiently large with respect to $\|f\|_\infty$, then 
\begin{equation}\label{eq:VoronoiLots}
\min\left\{\,|V_r(\vvec{u})\cap A|,|V_r(\vvec{u})\cap B|\,\right\} \ge \sum_{\vvec{v}\vvec{w}\in \partial_EV_r(\vvec{u})}|f(\vvec{v},\vvec{w})|\end{equation}
for every $\vvec{u}\in W_r$. We fix $r$ large enough so that \eqref{eq:VoronoiLots} holds for every $\vvec{u}\in W_r$; to achieve the ``moreover'' part of the lemma, we will need the slightly stronger inequality \eqref{eq:enoughAB} stated later.

For each pair $\vvec{u},\vvec{u}'\in W_r$, define 
\[F(\vvec{u},\vvec{u}'):=\sum_{\substack{\vvec{v}\in V_r(\vvec{u})\\ \vvec{w}\in V_r(\vvec{u}')}} f(\vvec{v},\vvec{w}),\]
 that is, $F(\vvec{u},\vvec{u}')$ is the total flow sent by $f$ from the Voronoi cell of $\vvec{u}$ to that of~$\vvec{u}'$. Define $A(\vvec{u},\vvec{u}'):=\emptyset$ for every pair of distinct $\vvec{u},\vvec{u}'\in W_r$ with $F(\vvec{u},\vvec{u}')=0$.
Given $\vvec{u}\in W_r$, there are finitely many $\vvec{u}'\in W_r\setminus\{\vvec{u}\}$ for which $F(\vvec{u},\vvec{u}')\neq 0$. For every such $\vvec{u}'$, one by one in order prescribed by $\lex$, we define $A(\vvec{u},\vvec{u}')$ to be the set of those $\max\{0,F(\vvec{u},\vvec{u}')\}$  elements of $V_r(\vvec{u})\cap A$ that have not already been assigned to $A(\vvec{u},\vvec{u}'')$ for some $\vvec{u}''\lex \vvec{u}'$ and, subject to that, are minimal under $\lex$. Note that, by \eqref{eq:VoronoiLots}, this is always possible. Similarly, we define $B(\vvec{u},\vvec{u}')$ to be the set of those $\max\left\{0,-F(\vvec{u},\vvec{u}')\right\}$ $\lex$-minimal elements of $V_r(\vvec{u})\cap B$ that have not already been assigned to $B(\vvec{u},\vvec{u}'')$ for some $\vvec{u}''\lex \vvec{u}'$.  Finally, for each $\vvec{u}\in W_r$, define $A(\vvec{u},\vvec{u})$ to be the set of those vertices in $V_r(\vvec{u})\cap A$ that have not been assigned to $A(\vvec{u},\vvec{u}')$ for any $\vvec{u}'\in W_r\setminus\{\vvec{u}\}$ and define $B(\vvec{u},\vvec{u})$ similarly. We have by construction that
$$
|A(\vvec{u},\vvec{u}')|=|B(\vvec{u}',\vvec{u})|,\quad \mbox{for all distinct $\vvec{u},\vvec{u}'\in W_r$,}
$$
and this holds also for $\vvec{u}=\vvec{u'}$ since $f$ is a flow from $A$ to $B$.
 The final equidecomposition assigns,  for all $\vvec{u},\vvec{u}'\in W_r$, the vertices of $A(\vvec{u},\vvec{u}')$ to those of $B(\vvec{u}',\vvec{u})$ in the order prescribed by $\lex$; see Figure~\ref{fig:matching} for an illustration.

\begin{figure}[htbp]
\begin{center}
\includegraphics[width=0.95\textwidth]{VoronoiNew.1}
\end{center}
\caption{An illustration of the construction of the equidecomposition from an integer-valued flow $f$ in Lemma~\ref{outline:lem:matching} in the simplified setting $d=2$. White nodes are elements of a maximally 19-discrete set and the boundaries of the induced Voronoi cells are in bold. Elements of $A$ and $B$ correspond to black round and square nodes, respectively. The total flow from the central Voronoi cell to the bottom-left cell is $2$; hence the first two elements of $A$ in the lexicographic order in the central Voronoi cell are associated to that cell. After distributing elements of $A$ and $B$ to the neighbouring cells, the two remaining (lexicographically latest) elements of $A$ and of $B$ in the central cell are mapped to one another.}
\label{fig:matching}
\end{figure}

Since each cell has diameter at most~$2r$, this yields an equidecomposition by translations in which the translation vectors are of the form $\sum_{i=1}^dn_i\vvec{x}_i$ with $\|\vec{n}\|_\infty\le 4r+1$. Thus the total number of pieces is finite. Furthermore, the vector by which we translate any $\vvec{v}\in V_r(\vvec{u})\cap A$ depends only on the situation inside $V_r(\vvec{u})$, its adjacent Voronoi cells and their adjacent Voronoi cells, and the values of $f$ at their boundary edges. All these cells are contained entirely inside~$N_{6r+2}[\vvec{v}]$, and so the statement holds with $R:=6r+2$.

The ``moreover'' part of the lemma is only needed for Part~\ref{it:c} of Theorem~\ref{th:main}. Its proof may be skipped by the reader interested only in the other parts of Theorem~\ref{th:main}. In brief, the new idea is to arrange that the piece  corresponding to each translation vector $\vvec{t}\in\I T^k$ that is used in the equidecomposition is pre-assigned a small non-null subset $A_{\vvec{t}}\subseteq A$. Of course, this is possible only if $(A+\vvec{t})\cap B$ has positive measure, so we have to exclude all vectors violating this.

Let us provide the details. For each integer $r\ge 1$, let $W_r$, $\eta_r$ and $V_r$ be as in the proof of the first part. Now, fix an integer $r$ such that, for every $\vvec{u}\in W_r$, it holds that 
\begin{equation}\label{eq:enoughAB}\min\left\{\,\left|V_r(\vvec{u})\cap A\right|,\left|V_r(\vvec{u})\cap B\right|\,\right\} \ge \sum_{\vvec{v}\vvec{w}\in \partial_EV_r(\vvec{u})}\left(|f(\vvec{v},\vvec{w})|+1\right),\end{equation}
 that is, we require a slightly stronger bound than the one in~\eqref{eq:VoronoiLots}.

Let $T$ be the set of all vectors of the form $\vvec{t}=\sum_{i=1}^dn_i\vvec{x}_i$ such that $\vec{n}\in \mathbbm{Z}^d$ with $\|\vec{n}\|_\infty \leq 10r$ and $(A+\vvec{t})\cap B\neq\emptyset$.  Since \eqref{eq:it:c} holds, we have by~\eqref{eq:XsAvoids} that, for every $\vvec{t}\in T$, the set $(A+\vvec{t})\cap B$ has positive measure. (Recall that the sets $A$ and $B$ satisfying the assumptions of Theorem~\ref{th:main} are necessarily Lebesgue measurable.)


\begin{claim}
\label{cl:preSelect}
There exists a sequence $(A_{\vvec{t}})_{\vvec{t}\in T}$ of pairwise disjoint measurable non-null subsets of $A$ such that the union $\bigcup_{\vvec{t}\in T}A_{\vvec{t}}$ is $(100r)$-discrete in~$G_d$ and, for all $\vvec{t}\in T$, $A_{\vvec{t}}+\vvec{t}\subseteq B$.
\end{claim}

\begin{proof}
We use a simple greedy argument. By Observation~\ref{outline:obs:diamDiscrete}, we can choose $\gamma$ sufficiently small so that every strip of width at most $\gamma$ is $(100r)$-discrete. Assume further that $\gamma$ is chosen small enough so that
\[\gamma(200r+1)^d|T| < \lambda\left((A+\vvec{t})\cap B\right)=\lambda\left(A\cap (B-\vvec{t})\right)\]
for every $\vvec{t}\in T$. Label the elements of $T$ by $\vvec{t}_1,\ldots,\vvec{t}_{|T|}$ in an arbitrary fashion.  

Start by taking $S_{\vvec{t}_1}$ to be a strip of width $\gamma$ such that the intersection 
\[
A_{\vvec{t}_1}:=A\cap (B-\vvec{t}_1)\cap S_{\vvec{t}_1}
\] 
is non-null. Clearly, the measure of $A_{\vvec{t}_1}$ is at most the measure of $S_{\vvec{t}_1}$, which is~$\gamma$.  

Now, let $2\leq i\leq |T|$ and assume that each of the sets $A_{\vvec{t}_1},\ldots,A_{\vvec{t}_{i-1}}$ has measure at most $\gamma$. Then, by our choice of $\gamma$, the set
\[
\left(A\cap (B-\vvec{t}_i)\right)\setminus \left(\bigcup_{j=1}^{i-1}N_{100r}\left[A_{\vvec{t}_j}\right]\right)
\]
has positive measure. Take any strip $S_{\vvec{t}_i}$ of width $\gamma$ that has non-null intersection with the above set and let $A_{\vvec{t}_i}$ be this intersection.

The collection $\{A_{\vvec{t}_i}:1\leq i\leq |T|\}$ has all of the desired properties, simply by construction, finishing the proof of the claim. 
\end{proof}

Let $(A_{\vvec{t}})_{\vvec{t}\in T}$ be the sequence returned by the claim. Observe that, for any distinct $\vvec{t},\vvec{t}'\in T$, the sets $A_{\vvec{t}}+\vvec{t}$ and $A_{\vvec{t}'}+\vvec{t}'$ are disjoint; otherwise, by the definition of $T$, the sets $A_{\vvec{t}}$ and $A_{\vvec{t}'}$ would be at distance at most $20r$ in $G_d$, contradicting the fact that $\bigcup_{\vvec{t}\in T}A_{\vvec{t}}$ is $(100r)$-discrete in $G_d$. Thus, the sets $A_{\vvec{t}}$ for $\vvec{t}\in T$ form an equidecomposition of the subset
\[A^*:=\bigcup_{\vvec{t}\in T}A_{\vvec{t}}\]
of $A$ to the subset 
\[B^*:=\bigcup_{\vvec{t}\in T}\left(A_{\vvec{t}}+\vvec{t}\right)\]
of $B$ using translations in $T$.

Now, for each $\vvec{t}\in T$ and each vertex $\vvec{v}\in A_{\vvec{t}}$, take a shortest path from $\vvec{v}$ to $\vvec{v}+\vvec{t}$ in $G_d$, chosen to be $\lex$-minimal among all such paths, and let $f_{\vvec{v}}$ be the (unique) $\{0,\pm1\}$-valued flow in $G_d$ from $\{\vvec{v}\}$ to $\{\vvec{v}+\vvec{t}\}$ supported on the pairs which form edges of this path. Since the set $\bigcup_{\vvec{t}\in T}A_{\vvec{t}}$ is $(100r)$-discrete, any two such paths are edge disjoint; in fact they are  at least $80r$ apart in~$G_d$.
Thus the flow
\[f^*:=\sum_{\vvec{v}\in \bigcup_{\vvec{t}\in T}A_{\vvec{t}}}f_{\vvec{v}}\]
satisfies $\|f^*\|_\infty\leq 1$. Also, letting $f':= f-f^*$, we see that \eqref{eq:enoughAB} implies that, for every $\vvec{u}\in W_r$,
\[
\min\left\{\,|V_r(\vvec{u})\cap A|,|V_r(\vvec{u})\cap B|\,\right\}\ge \sum_{\vvec{v}\vvec{w}\in \partial_EV_r(\vvec{u})}|f'(\vvec{v},\vvec{w})|,
\]
 since at most one path in the support of $f^*$ can intersect $\partial_EV_r(\vvec{u})$.
Clearly, $f'$ is a flow from $A\setminus A^*$ to $B\setminus B^*$ in $G_d$. By the proof of the first part of the lemma with $f'$ in the place of $f$, we see that there is an equidecomposition from $A\setminus A^*$ to $B\setminus B^*$ in which each $\vvec{v}\in A\setminus A^*$ is associated to a point in $B\setminus B^*$ which is either in the same Voronoi cell as $\vvec{v}$ or one which neighbours it. In particular, all of the translation vectors used in the equidecomposition are contained in $T$. Moreover, by making the $\lex$-smallest choices inside each Voronoi cell, we can assume
that each piece is an $O(r)$-local function of $f,W_r,A,B,A_{\vvec{t}_1},\ldots,A_{\vvec{t}_{|T|}}$. In turn, each set $A_{\vvec{t}_{i}}$ is a $O(r)$-local function of $A$, $B$ and the strips $S_{\vvec{t}_1},\ldots,S_{\vvec{t}_{|T|}}$. 

Taking, for each $\vvec{t}\in T$, the union of $A_{\vvec{t}}$ and the (possibly empty) piece of the equidecomposition from $A\setminus A^*$ to $B\setminus B^*$ corresponding to $\vvec{t}$ yields an equidecomposition of $A$ to $B$ using translation vectors in $T$ such that every piece is non-null, as required.
\end{proof}

\section{Proof Outline}  
\label{sec:JSec}

With some preliminaries covered by Section~\ref{sec:prelim}, we can now 
provide a more detailed outline of the proof of Theorem~\ref{th:main} than the one given in Section~\ref{sec:ideas}.
In order to illustrate how the various arguments fit together, we will first sketch the construction of an equidecomposition with Jordan measurable pieces (which are not necessarily Borel) and then outline how to build upon it to get the full result.
Our approach mostly follows the general recipe established by Marks and Unger~\cite{MarksUnger17}, some aspects of which can be traced back to the ideas of Laczkovich~\cites{Laczkovich90,Laczkovich92,Laczkovich92b} and Grabow\-s\-ki, M\'ath\'{e} and Pikhurko~\cite{GrabowskiMathePikhurko17}. A major challenge here is that the family of Jordan measurable subsets of $\mathbb{T}^k$ is not a $\sigma$-algebra (although it is an algebra).
Therefore, there are a few key differences in our implementation and analysis of this strategy which are necessary to achieve Jordan measurable pieces. 

\subsection{Real-Valued Flows}
\label{subsec:realFlows}

Following Marks and Unger~\cite{MarksUnger17}*{Section~4}, we construct a sequence $f_1,f_2,\ldots$ of bounded real-valued flows in $G_d$ which converge uniformly to a bounded real-valued flow $f_\infty$ from $A$ to $B$. 
The following lemma 
summarises the key properties of $f_1,f_2,\ldots$ that we will need.
Recall that $c$ and $\varepsilon$ are fixed quantities which were defined in Section~\ref{subsec:setting}.

\begin{lemma}[Marks and Unger~\cite{MarksUnger17}]
\label{outline:lem:realFlows}
There exist flows $f_1,f_2,\ldots$ in $G_d$ such that for all $m\ge1$ the following statements hold with $f_0:=0$ being the flow which is identically zero:
\begin{enumerate}
\stepcounter{equation}
\item\label{eq:fmRational} $2^{2dm} f_m$ is integer-valued,
\stepcounter{equation}
\item\label{eq:AtoB}  $\|\fout{f_m} - \ind_A+\ind_B\|_\infty\leq \frac{2c}{2^{m(1+\varepsilon)}}$, 
\stepcounter{equation}
\item\label{eq:fmChange} $\|f_m-f_{m-1}\|_\infty\leq\frac{2c}{2^{d+\varepsilon(m-1)}}$, and
\stepcounter{equation}
\item\label{eq:local} 
$f_m$ is a $(2^m-1)$-local function of $A$ and $B$.
\end{enumerate}
\end{lemma}

The proof of the lemma follows that of a similar result in~\cite{MarksUnger17}*{Section~4}; we include the proof in Section~\ref{sec:realFlows} for completeness. The rough idea is as follows. Suppose that we have fixed a partition of $[\vvec{u}]_{G_d}\cong \I Z^d$ into discrete $2^m$-cubes. Within each cube $Q$, we define the flow $f_m$ so that it cancels as much as possible between the positive demand $\ind_A$ and the negative demand $-\ind_B$, and spreads the rest uniformly over~$Q$. A bit more formally, the restriction of $f_m$ to $Q$ is an $(\ind_A-\ind_B+\xi(Q))$-flow, where $\xi(Q):=(|A\cap Q|-|B\cap Q|)/|Q|$. Also, $f_m$ is zero on all edges between distinct cubes of the partition.  Assuming that the fixed dyadic partitions are aligned for different values of $m$, we can construct such $f_m$ incrementally from~$f_{m-1}$. As long as we do this increment in a ``reasonable'' way, the discrepancy bound of Lemma~\ref{lem:discrep} gives a good upper bound on~$\|f_m-f_{m-1}\|_\infty$. Unfortunately, one cannot take a perfect dyadic decomposition of each component of $G_d$ in a constructive way for a generic choice of $\vvec{x}_1,\ldots,\vvec{x}_d$. However, the convenience of working with flows (versus graph matchings) is that we can always take their convex combinations. So, in order to make the construction of each $f_m$ local, one can simply take the average over all $2^{dm}$ possible partitions of each component (a copy of $\I Z^d$) into a grid of discrete $2^m$-cubes. 

Each flow $f_m$ assumes finitely many values by~\eqref{eq:fmRational} and is a local function of the Jordan measurable sets $A$ and $B$.
Since the right hand side of \eqref{eq:fmChange} is summable, the sequence $f_1,f_2,\ldots$ converges uniformly to a bounded flow in $G_d$, which we denote by~$f_\infty$. By \eqref{eq:AtoB}, the limit $f_\infty$ is a flow 
from $A$ to~$B$.

With Lemma~\ref{outline:lem:realFlows} in hand, the goal of the next two steps is to transform the sequence $f_1,f_2,\ldots$ into a bounded integer-valued flow $f$ from $A$ to $B$ in $G_d$. Notice that, if we were indifferent about the structure of the pieces of the final equidecomposition, then we could just apply the Integral Flow Theorem to $f_\infty$ to obtain such a flow $f$ and then feed it into Lemma~\ref{outline:lem:matching}. In particular, this approach is sufficient for proving Theorem~\ref{th:L}. However, this would use the Axiom of Choice and yield virtually no structural guarantees on the pieces of the equidecomposition. The main purpose of the next two steps, therefore, is to obtain an integer-valued flow from $A$ to $B$ in a more careful manner which allows us to analyse the obtained pieces.

\subsection{Toast Sequences}

Recall that the notion of a toast sequence was defined in Definition~\ref{def:scaff}. Informally, if $(J_i)_{i=1}^\infty$ is a toast sequence, then one can view $J_i$ as a collection of bounded and well separated  connected subgraphs of $G_d$ that arrive at time $i$ so that every component $S$ of vertices that arrived in an earlier stage is either entirely inside $J_i$ or entirely outside $J_i$, including some ``padding''.

The following lemma provides a simple construction of a toast sequence which covers $\mathbbm{T}^k$ up to a null set. (This will be sufficient for our construction of an equidecomposition with Jordan measurable pieces.)

\begin{lemma}
\label{outline:lem:JordanScaff} 
There exists a toast sequence $(J_1,J_2,\ldots)$ in $G_d$ such that $\lambda\left(\bigcup_{i=1}^\infty J_i\right)=1$ and, for each $i\ge 1$, the set $J_i$ is a union of finitely many disjoint strips.
\end{lemma}

\begin{proof}[Proof Sketch] The construction is slightly more complicated than is necessary, in order for the sets $J_i$ to be defined in the same way as they will be in the proof of Theorem~\ref{th:main}.

Recall that $r_0'<r_1<r_1'<r_2<r_2'<\ldots$ is a rapidly increasing sequence of integers and $X_i\subseteq \I T^k$ is  a maximally $r_i$-discrete
set which is a union of finitely many strips. 
Define $I_i$ to be the set of all $\vvec{v}\in \mathbbm{T}^k$ for which there exists $\vvec{u}\in X_i$ such that
\[\dist_{G_d}(\vvec{v},\vvec{u}')\ge\dist_{G_d}(\vvec{v},\vvec{u})+5 r'_{i-1}, \quad\mbox{for every $\vvec{u}'\in X_i\setminus \{\vvec{u}\}$.}\]
 Informally speaking, the components of $I_i$ are the Voronoi cells of $X_i\subseteq V(G_d)$, as defined in the proof of Lemma~\ref{outline:lem:matching}, except we retract somewhat from their graph boundaries. It is not hard to show  that the diameter of each component of $I_i$ is at most $2r_i$ and every two components are at distance at least $5r_{i-1}'$ (see Lemma~\ref{lem:Ii}).

We use induction on $i$ to define $J_i$. Fix $i\ge 1$ and assume that the sets $J_j$, for all $j<i$, have already been defined. We obtain $J_i$ as the result of the following procedure. Initialise $J_i:=I_i$. 
Then, while there exist $j<i$ and a component $S$ of $J_j$ such that $\dist_{G_d}(S,J_i)\leq q_{j-1}'+4$
and $N_{q_{j-1}'+4}[S]\not\subseteq J_i$, we add  to $J_i$ all elements of $N_{q_{j-1}'+4}[S]$. (Recall that $q_i'$ was defined by~\eqref{eq:qi'}, being slightly larger than~$4r_i'/5$.)

Since the sequence $r_0'<r_1<r_1'<\ldots$ increases sufficiently rapidly, this procedure eventually terminates and in fact stays within the $q'_{i-1}$-neighbourhood of $I_i$ (Lemma~\ref{lem:Ji}). Thus, when we construct $J_i$ from $I_i$, the original components of $I_i$ do not merge and stay well separated, being at distance at least~$5r_{i-1}'-2q'_{i-1}$ from each other.
(This stronger separation property is not needed here, but will be useful in the general proof.)
The other properties of a toast sequence from Definition~\ref{def:scaff} hold because, whenever some earlier component $S'$ (that is, $S'\in \comp(J_j)$ with $j<i$)
comes too close to a currently defined component $S$ of $J_i$, we add $S'$ with some padding to $S$ (with the new enlarged set $S$ still being connected).

We trivially have $\lambda(J_i)\ge \lambda(I_i)$. On the other hand, the measure of $I_i$ can be shown to be least $1-O(r_{i-1}'/r_i)$ (see Lemma~\ref{lem:Jmeas}), which approaches 1 as $i\to\infty$.

Whether or not a vertex $\vvec{u}\in\mathbbm{T}^k$ is contained in $J_i$ can be determined by the structure of the sets $X_1,\ldots,X_i$ at bounded distance (say, $4r_i$) from $\vvec{u}$ in $G_d$ (see Lemma~\ref{lm:JiLocality}); thus, $J_i$ is a local function of $X_1,\ldots,X_i$. So, it can be written as a union of finitely many disjoint strips by Observation~\ref{obs:LocalRule} (since such sets form an $a$-invariant algebra), finishing our proof sketch. 
\end{proof}

\subsection{Integer-Valued Flow}
\label{subsec:intFlows}

The last remaining step is to use the flows $(f_m)_{m=1}^\infty$ from Lemma~\ref{outline:lem:realFlows} and
the toast sequence $(J_i)_{i=1}^\infty$ from Lemma~\ref{outline:lem:JordanScaff} to construct an integer-valued bounded flow $f$ from $A$ to $B$ that will be used as input to Lemma~\ref{outline:lem:matching} to equidecompose $A$ and $B$. Unfortunately, the construction of $f$ is somewhat involved, even in the context of Jordan measurable pieces.

Although this is not strictly necessary, we keep the perspective
that  a new set of vertices $J_i$  arrives at time $i$ and we have to fix for good the value of $f$ on every new edge intersecting $J_i$, being compatible with all previous choices. It is enough to concentrate on constructing the final values of $f$ on the edges in $\partial_E J_i$, that is, on all new boundary edges. Indeed, since every component $S$ of $J_i$ is finite, there are only finitely many possible extensions of $f$ to a uniformly bounded integer flow inside $S$ and, if at least one exists, then the lexicographically smallest extension is a local function of the boundary flow values, $J_1,\ldots,J_i,A$ and $B$ whose radius is at most the diameter of~$S$. A bit later, we will address the problem of certifying the existence of such an extension.

So, take any $S\in\comp(J_i)$, i.e.,\ a component of $G_d\induced J_i$ which must be finite by the definition of a toast sequence. Initially let $f$ be $f_{m_i}$ on $\partial_E S$, where $m_i$ is sufficiently large integer depending on the maximum diameter of the components of $J_i$.
We repeat the following for every ``connected'' part $P\subseteq \partial_E S$ of the edge boundary of~$S$. Let $S'$ consist of the vertices ``enclosed'' by~$P$; thus, $S'$ is either a hole of $S$ or the set $S$ with all its holes filled.
By adjusting the value of $f$ on one edge of~$P$, we make the total flow out of $S'$ to be $|A\cap S'|-|B\cap S'|$. 
Since $\fout{f_{m_i}}$ is a very good approximation to $\ind_A-\ind_B$,
we need to adjust the flow by less than~$1/2$.
Using Lemma~\ref{lem:triEvenDeg}, we find a sequence $(\vvec{u}_1\vvec{v}_1,\ldots,\vvec{u}_{t-1}\vvec{v}_{t-1})$ in $P$ so that each edge in $P$ appears at least once but at most constant number of times while $\vvec{u}_j\vvec{v}_j$ and $\vvec{u}_{j+1}\vvec{v}_{j+1}$ are in a triangle in $G_d$ for each $1\le j\le t-2$. 
Make $f$ to be integer on $\vvec{u}_1\vvec{v}_1$ by adding a $0$-flow (i.e., a flow $\phi$ such that $\fout{\phi}:V(G_d)\to\I R$ is the zero function) on the unique triangle containing both $\vvec{u}_1\vvec{v}_1$ and $\vvec{u}_2\vvec{v}_2$. (Note 
that the third edge of this triangle is not in $\partial_E S$
and thus we do not care about the flow through it being integral at this stage.) By adding a $0$-flow along the triangle containing both
$\vvec{u}_2\vvec{v}_2$ and $\vvec{u}_3\vvec{v}_3$,
we make the flow value on $\vvec{u}_2\vvec{v}_2$ integer, and so on. We repeat this procedure $t-2$ times in total, ensuring that the final flow $f$ is integer-valued on all edges of $P$ except perhaps $\vvec{u}_{t-1}\vvec{v}_{t-1}$. Note that this part of the process does not change the flow out of any vertex under~$f$. So $\fout{f}(S')$ keeps its (integer) value and thus the final flow has to be integer also on $\vvec{u}_{t-1}\vvec{v}_{t-1}$.
Clearly, the definition of $f$ can be made local by fixing a consistent rule for choosing the adjustment, the edge sequence and the values of the added $0$-flows.


If we had used $f_\infty$ as the initial values of $f$ on $P$, then the initial adjustment step would not be necessary, and the obtained rounding algorithm would be essentially the same as used by Marks and Unger~\cite{MarksUnger17}. Unfortunately, if we wish to have Jordan measurable pieces, then we should (and, in fact, we do) avoid using the values of $f_\infty$ (which is the pointwise limit of $f_1,f_2,\ldots$) when defining $f$ on $\bigcup_{i=1}^\infty \partial_E J_i$.
However, we can freely use $f_\infty$ to define, in a way that parallels the construction of $f$ from $f_{m_i}$, a real-valued flow which certifies via the
Integral Flow Theorem (Theorem~\ref{th:IFT}) that the constructed integer flow values on $\bigcup_{i=1}^\infty \partial_E J_i$, can be extended to a bounded integer $(\ind_A-\ind_B)$-flow on all edges of $G_d$
(see Lemma~\ref{lm:RoundingGD} for details). As mentioned above, given that such an extension exists, we can choose the lexicographically minimal one on each finite component of the remaining graph; this includes all edges inside~$\bigcup_{i=1}^\infty J_i$. Here, crucially, the choice of the extension is independent of~$f_\infty$.

Now, during the construction, each of the (finitely many) sets $Z^{f}_{\vec{\gamma},\ell}\cap \left(\bigcup_{j=1}^i J_j\right)$ that describe the flow $f$ on edges incident to the vertices which have arrived by time $i$ (where $Z^{f}_{\vec{\gamma},\ell}$ is as in~\eqref{eq:Zdef}), only grows in time, since the values of $f$ fixed at time $j$ are never overwritten later. These sets are Jordan measurable (as local functions of $f_{m_1},\ldots,f_{m_i}$ and $J_1,\dots,J_i$) and their union over all values of $i$ and $\ell$ covers the co-null subset $\bigcup_{i=1}^\infty J_i$ of~$\I T^k$. Thus, for every $\ell$, the set $Z^{f}_{\vec{\gamma},\ell}\cap \left(\bigcup_{j=1}^\infty J_j\right)$ is Jordan measurable by the following lemma (applied with $Z:=\I T^k$).

\begin{lemma}
\label{lem:JordanLimit}
Let $Z_1,\ldots,Z_N$ be pairwise disjoint subsets of a Jordan measurable set $Z\subseteq \mathbbm{T}^k$.
If, for every real $\e>0$, there are Jordan measurable subsets $Z_j'\subseteq Z_j$  such that $\lambda(Z\setminus \bigcup_{j=1}^N Z_j')<\e$, then all of the sets $Z_1,\ldots,Z_N$ are Jordan measurable. 
\end{lemma}

\begin{proof} By adding $\mathbbm{T}^k\setminus Z$ as an extra set and increasing $N$ by 1, we can assume that $Z=\mathbbm{T}^k$. For each integer $i\ge 1$, let $Z_{1,i}\subseteq Z_1,\ldots,Z_{N,i}\subseteq Z_N$ be Jordan measurable sets such that $\lambda(Z\setminus \bigcup_{j=1}^N Z_{j,i})<1/i$ and, for $1\le j\le N$, we  let $U_{j,i}$ be the interior of $Z_{j,i}$. Also, define $U_j:=\bigcup_{i=1}^\infty U_{j,i}$.  Since $Z_{j,i}$ is Jordan measurable, we have that $U_{j,i}$ has the same Lebesgue measure as $Z_{j,i}$. Thus, the sets $U_1,\ldots,U_N$ are pairwise disjoint, open and cover $\mathbbm{T}^k$ up to a null set. For any $1\leq j\leq N$, since $U_j$ is open and contained in $Z_j$, the boundary of $Z_j$ is disjoint from~$U_j$. Also, for any $j'\neq j$, since the set $U_{j'}$ is open and disjoint from $Z_{j}$, we see that the boundary of $Z_{j}$ is also disjoint from~$U_{j'}$. Therefore, for each $1\leq j\leq N$, the boundary of $Z_{j}$ is contained in the complement of $\bigcup_{i=1}^NU_i$ and therefore has measure zero. Thus the set $Z_j$ is Jordan measurable.
\end{proof}

It also follows from the above lemma  that the remaining set $\I T^k\setminus \bigcup_{i=1}^\infty J_i$ is Jordan measurable and, being also a null set, has null closure. Thus, without destroying Jordan measurability, we can use the Axiom of Choice to define a suitable integer-valued flow on all edges inside this set, using $f_{\infty}$ in the same way as above to certify the existence via the Integral Flow Theorem.


Finally, we apply Lemma~\ref{outline:lem:matching} to $f$ to get an equidecomposition between $A$ and $B$ as an $R$-local function of $f$ (that is, the sets  $Z^{f}_{\vec{\gamma},\ell}$) and finitely many strips; thus, the pieces are Jordan measurable.

\begin{remark} We could have slightly restructured the proof so that whenever the partial integer-valued flow $f$ is defined in the whole $R$-neighbourhood of some $\vvec{u}\in A$, where $R$ is the constant returned by Lemma~\ref{outline:lem:matching}, then we assign $\vvec{u}$ to the appropriate part of the final equidecomposition. Then the obtained Jordan measurable pieces will be incrementally growing as we process $J_1,J_2,\ldots$ one by one, exhausting the set $A$ up to measure zero. Thus, we could have applied Lemma~\ref{lem:JordanLimit} directly to the pieces (with $Z:=A$).
\end{remark}

\subsection{Making the Pieces Borel}
\label{se:Borel}

Let us now discuss how the approach to obtaining a Jordan measurable equidecomposition described above can be built upon to obtain the stronger conditions of Theorem~\ref{th:main}. 

The main issue to address is the use of the Axiom of Choice in the final step, which must be avoided in order to yield Part~\ref{it:b} of the theorem. If $A$ and $B$ are Borel, then a natural idea for obtaining simultaneously Jordan and Borel pieces is to try to use the existence of a Jordan measurable equidecomposition together with the equidecomposition of Marks and Unger~\cite{MarksUnger17}; however, we did not see a way of doing this directly.  Instead, we follow essentially 
the same proof as outlined above (with all locally defined structures being Borel if the sets $A$ and $B$ are) except that we use a version of Marks and Unger's proof in~\cite{MarksUnger17} to round the flow values inside $\mathbbm{T}^k\setminus \bigcup_{i=1}^\infty J_i$ in a way that preserves Borel structure. The Jordan measurability of the final pieces is preserved since, in this extra stage, we do not change any flow value on any edge intersecting~$\bigcup_{i=1}^\infty J_i$. 

Our extra steps can be summarised as follows.
With sets $J_i$ defined as above and fixed for good (where the role of the parameter $r_{i-1}'$ is to provide enough ``wiggle space'' between the components of $J_i$), we interleave the sequence $(J_1,J_2,\ldots)$ with sets $K_1,K_2,\ldots$ to create an augmented toast sequence $(J_i,K_i)_{i=1}^\infty$. (Recall that $(J_i,K_i)_{i=1}^\infty$ is a shorthand for $(J_1,K_1,J_2,K_2,\ldots)$.) Simultaneously, we will construct sets $L_1,L_2,\ldots$ so that $(K_i,L_i)_{i=1}^\infty$ is also a toast sequence. The purpose of the sets $L_1,L_2,\ldots$ will be to cover all of the points of $\I T^k$ that are ``left behind'' by the construction of $J_1,J_2,\ldots$ in a Borel way, while the purpose of $K_i$ is to provide some extra ``cushioning'' around $J_i$ to protect the flow values on edges intersecting $J_i$ from being influenced by~$L_i$. The sequence $(K_1,K_2,\ldots)$ can be seen as a mediator between the competing goals of $J_1,J_2,\ldots$ (to keep the boundaries of the final pieces small) and $L_1,L_2,\ldots$ (to cover all of $\I T^k$ in a Borel way). 

We build $K_i$ by initialising $K_i:=N_2[J_i]$ and then iteratively adding (with an appropriate padding) all components of $J_j$, $K_j$ or $L_j$ with $j<i$ that come too close. As before, this process when started at any $S\in\comp(J_i)$ goes only a small distance away from $S$ and the obtained component $S'$ of $K_i$ satisfies $S'\subseteq N_{q_{i-1}'+2}[S]$ (Lemma~\ref{lem:KiLi}) while, trivially, $S'\supseteq N_2[S]$. Thus every component $S$ of $J_i$ is ``mimicked'' by a component $S'$ of $K_i$.
When we construct $L_i$, we start with $L_i:=N_{2r_i'/5}[Y_i]$ (where $Y_i$ is maximally $r_i'$-discrete) and keep adding (with some padding) components of $K_j$ for $j\le i$ and $L_j$ for $j<i$ that come too close. Note that each component $S$ of $J_j$ with $j\le i$ is ``protected'' in this process by the  component $S'$ of $K_j$ that mimics it: if $L_i$ comes too close to $S$ then $S'\supseteq S$ with some extra padding is added to $L_i$.  What we have achieved is that the toast sequence $(K_i,L_i)_{i=1}^\infty$ covers all vertices in $\bigcup_{i=1}^\infty L_i\supseteq \bigcup_{i=1}^\infty N_{2r_i'/5}[Y_i]$. Using a compactness argument, we will show that, by applying this construction with a modified version of $Y_i$, we can ensure that all vertices of $\I T^k$ are covered; see the discussion following Lemma~\ref{lem:Licover}.

Unfortunately, $(J_i,K_i,L_i)_{i=1}^\infty$ is not, in general, a toast sequence. Also, we do not want any $S\in\comp(K_j)\cup\comp(L_j)$  to ``interfere'' with what happens inside $J_i$  for some $i>j$. 
Our solution is, essentially, to remove all such conflicting components $S$, obtaining sets $\Tilde{K}_j$ and $\Tilde{L}_j$. The new sequence $(J_i,\Tilde{K}_i,\Tilde{L}_i)_{i=1}^\infty$ is then a toast sequence that does not break up any finite components coming from 
the toast sequence 
$(J_i)_{i=1}^\infty$; also, 
it still covers every vertex of $\bigcup_{i=1}^\infty L_i$ (see Lemma~\ref{lm:Tilde}). The proofs of these claims rely on the fact that the components of $K_j$ ``mimic'' the components of~$J_j$.

We define the desired integer-valued Borel flow $f$ by using the same construction as in Section~\ref{subsec:intFlows}, except that we replace the toast sequence $(J_i)_{i=1}^\infty$ by $(J_i,\Tilde{K}_i,\Tilde{L}_i)_{i=1}^\infty$.

\subsection{Reducing/Analysing the Complexity of the Pieces}
\label{se:Analysis}

The idea to use $f_{m_i}$ instead of $f_\infty$ in the construction of the integer-valued flow $f$ is, essentially by itself, enough to save one level of the Borel hierarchy when compared to the proof of Marks and Unger~\cite{MarksUnger17}.
There are a couple of places where we change their construction or its analysis in order to drop down by another level of Borel complexity. (In particular, we have to use a different toast sequence $(J_i,\Hat{K}_i,\Hat{L}_i)_{i=1}^\infty$, some modification of
$(J_i,\Tilde{K}_i,\Tilde{L}_i)_{i=1}^\infty$, for the final rounding.) These changes are incorporated in the main construction, while the Borel complexity analysis is postponed until Section~\ref{sec:BorelComplexity} for the clarity of presentation.

Our approach to bounding the upper Minkowski dimension (defined in~\eqref{eq:DefBoxDim}) of the boundaries of the pieces is, for given $\delta>0$, to choose some index $i=i(\delta)$ and run our algorithm to process $J_1,\ldots,J_i$. Although we do not know the final equidecomposition at this moment, we may nonetheless determine for some elements of $A$ to which final piece they will belong. This defines the current partial pieces $A_1,\ldots,A_N$ (that give an equidecomposition of some subset of $A$ to a subset of $B$). Now, we estimate the number of boxes of a regular grid with side length $\delta$ that intersect the currently unassigned part of $A$ or the boundary of any current piece $A_j$. (Thus, our dimension estimates depend only on $J_1,J_2,\ldots$ and not on how  $K_1,L_1,K_2,L_2,\ldots$ are built around them.) Recall by Lemma~\ref{outline:lem:JordanScaff} that each set $J_j$ is a finite union of strips, whose number we can estimate if the sets $X_i$ are carefully chosen (namely, as in Lemma~\ref{lm:Xi}).  Also,
we can estimate a radius $R$ so that the current pieces $A_1,\ldots,A_N$ can be $R$-locally determined from $A$, $B$ and $J_1,\ldots,J_i$. Thus each  $A_j$ is a Boolean combination of $(2R+1)^d$ translates of each of these sets, and the number of boxes that intersect its boundary $\partial A_j$ can be bounded above by the number of boxes intersecting the boundary of at least one of these translates. Furthermore, every other box that may potentially intersect the boundary of a final piece has to be a subset of $A':=A\setminus(\bigcup_{j=1}^N A_j)$ and their number can be upper bounded by $\lambda(A')/\delta^k$ by the trivial volume argument. Thus we have to control our parameters carefully to get a good balance between the minimal distance between the components of $J_i$ (in order to control the measure of the leftover part $A'$ of $A$) and
their maximum diameter (as our local rule processing $J_i$ has to use radius at least as large as this diameter).

\section{Proof of Lemma~\ref{outline:lem:realFlows}}
\label{sec:realFlows}

Our goal in this section is to present, for the sake of completeness, the (somewhat rephrased) construction of Marks and Unger~\cite{MarksUnger17} of a sequence $f_1,f_2,\ldots$ of real-valued flows in $G_d$ which converge uniformly to a bounded flow $f_\infty$ from $A$ to $B$ in~$G_d$. For the reader's convenience, let us repeat the  statement of the result that we will use.

\begin{manuallemma}{\ref{outline:lem:realFlows}}[Marks and Unger~\cite{MarksUnger17}]
There exist flows $f_1,f_2,\ldots$ in $G_d$ such that for all $m\ge1$ the following statements hold with $f_0:=0$ being the flow which is identically zero:
\begin{enumerate}
\item[{\rm \eqref{eq:fmRational}}] $2^{2dm} f_m$ is integer-valued,
\item[{\rm \eqref{eq:AtoB}}]   $\|\fout{f_m} - \ind_A+\ind_B\|_\infty\leq \frac{2c}{2^{m(1+\varepsilon)}}$, 
\item[{\rm \eqref{eq:fmChange}}]  $\|f_m-f_{m-1}\|_\infty\leq\frac{2c}{2^{d+\varepsilon(m-1)}}$, and
\item[{\rm \eqref{eq:local}}]
$f_m$ is a $(2^m-1)$-local function of $A$ and $B$.
\end{enumerate}
\end{manuallemma}

\begin{proof}[Proof of Lemma~\ref{outline:lem:realFlows}] First, we need to introduce a few definitions (that apply only in this section).
Given a discrete cube $Q$ in $G_d$ (as defined in \eqref{eq:Nn+}), define
\[\xi(Q) := \frac{|Q\cap A| - |Q\cap B|}{|Q|}.\]
In other words, $\xi(Q)$ measures the difference between the number of vertices of $A$ and $B$ in $Q$ normalised by the number of vertices in $Q$. In particular, if $Q$ is a discrete $1$-cube, i.e.,\ $Q=\{\vvec{u}\}$ for some $\vvec{u}\in \I T^k$, then $\xi(Q) = \I 1_A(\vvec{u}) - \I 1_B(\vvec{u})$. 

By our assumption~\eqref{eq:XsDiscr}, it holds for every discrete $2^m$-cube $Q$ that
\begin{equation}\label{eq:Q}
|\xi(Q)|\le \frac{2c(2^m)^{d-1-\varepsilon}}{2^{dm}}=\frac{2c}{2^{m(1+\varepsilon)}}.
\end{equation}

Given $\vvec{u}\in\I T^k$,  $\vec{\gamma}\in\{-1,0,1\}^d$ and $n\in \I Z$, let $\phi_{\vvec{u},n\vec{\gamma}}$ be the unique $\{-1,0,1\}$-valued $\left(\I 1_{\{\vvec{u}\}}- \I 1_{\{(n\vec{\gamma})\cdot \vvec{u}\}}\right)$-flow in $G_d$ supported on the edges of the path
\[\vvec{u},\,\vec{\gamma}\cdot \vvec{u},\,(2\vec{\gamma})\cdot \vvec{u},\,\ldots,\,(n\vec{\gamma})\cdot \vvec{u}.\]
In other words, $\phi_{\vvec{u},n\vec{\gamma}}$ sends a unit of flow  from~$\vvec{u}$ to $(n\vec{\gamma})\cdot \vvec{u}$ along the straight-line path (whose direction is~$\vec{\gamma}$ and graph length is~$|n|$).
In particular, if $n\vec{\gamma}$ is the zero vector, then the flow $\phi_{\vvec{u},n\vec{\gamma}}$ is identically zero. Also, for an integer $m\ge 0$ and a discrete $2^{m+1}$-cube $Q$ containing $\vvec{u}\in\I T^k$, let
 $$
 \phi_{\vvec{u},Q}:=\frac1{2^d} \sum_{\substack{\vec{\gamma}\in\{-1,0,1\}^d\\ (2^{m}\vec{\gamma})\cdot \vvec{u}\in Q}} \phi_{\vvec{u},2^m\vec{\gamma}}.
 $$
 Note that the above sum contains exactly $2^d$ terms. Informally speaking, $\phi_{\vvec{u},Q}$ spreads one unit of demand from $\vvec{u}$ uniformly among $Q\cap ((2^m\I Z^d)\cdot_a \vvec{u})$, the set of 
 all $2^d$ points of $Q$
 that, when we view $Q$ as a subset of $\I Z^d$, are congruent to $\vvec{u}$ modulo~$2^m$.
For every discrete $2^{m+1}$-cube $Q$, let $\mathcal{P}(Q)$ be the unique partition of $Q$ into $2^d$ discrete $2^{m}$-cubes. Given  $\vvec{u}\in \I T^k$, let $\mathcal{Q}_n(\vvec{u})$ denote the set of all discrete $n$-cubes that contain~$\vvec{u}$. See Figure~\ref{fig:FlowF} for an illustration of some of these definitions.

\begin{figure}[htbp]
\begin{center}
\includegraphics[width=0.96\textwidth]{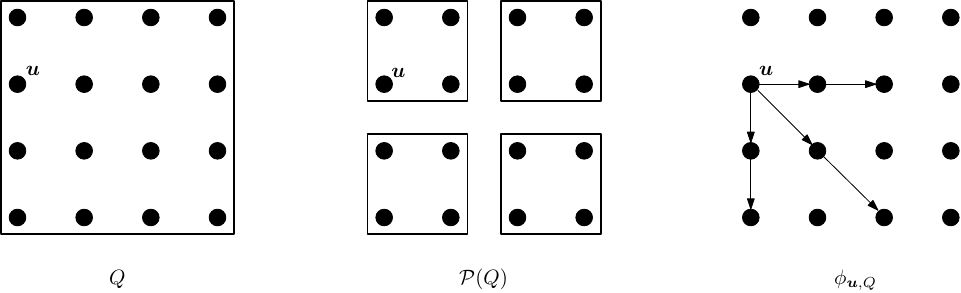}
\end{center}
\caption{An example, for $d=2$, of a 4-cube $Q$, its dyadic partition $\mathcal{P}(Q)$ into $2$-subcubes and the flow $\phi_{\vvec{u},Q}$ for a vertex $\vvec{u}\in Q$ (where each arrow represents a flow of value $1/4$ in that direction).}
\label{fig:FlowF}
\end{figure}

We are now ready to define the flows $f_m$. We use induction on $m$ and 
start by defining $f_0:=0$ to be the identically zero flow. For $m\geq1$, we define $f_m:=f_{m-1}+\theta_m$, where
\begin{equation}\label{eq:ThetaM}
\theta_m:= \frac{1}{2^{dm}} \sum_{\vvec{u}\in \I T^k}   \sum_{C\in \mathcal{Q}_{2^{m-1}}(\vvec{u})} \xi(C) \sum_{\substack{\mbox{\scriptsize discrete $2^m$-cube $Q$}\\ \mathcal{P}(Q)\ni C}} \phi_{\vvec{u},Q}.
\end{equation}

Let us provide a probabilistic interpretation that motivates the definition in~\eqref{eq:ThetaM} and relates it to the proof outline of Section~\ref{subsec:realFlows}. Suppose that we wish to compute $\theta_m$ and $f_m$ inside the component $[\vvec{v}]_{G_d}$ of some vertex $\vvec{v}\in\I T^k$. Take a random partition $\C G_m$ of the orbit $[\vvec{v}]_{G_d}$, which is a copy of $\I Z^d$, into a \emph{$2^m$-grid} (that is, a regular grid of discrete
$2^m$-cubes), where all $2^{dm}$ choices of $\C G_m$ are equally likely. By defining 
 $$
\C G_{i}:=\bigcup\{\mathcal{P}(Q)\mid Q\in\C G_{i+1}\}
 $$ 
inductively for $i=m-1,\ldots,0$, we get the dyadic refinements of $\C G_m$ down to the (deterministic) partition $\C G_0$ into discrete $1$-cubes (i.e., singletons). Note that we work inside one component only since, for generic $\vvec{x}_1,\ldots,\vvec{x}_d$, there cannot exist a measurable choice of a $2^m$-grid inside each orbit by ergodicity considerations. 

Starting with  $f_{\C G_0}:=0$ being the zero flow, obtain inductively for $i=1,2,\ldots,m$, the flow $f_{\C G_i}$ from $f_{\C G_{i-1}}$ by adding 
\begin{equation}\label{eq:ThetaIGm}
\theta_{\C G_i}:= \sum_{Q\in \C G_i} \sum_{\vvec{u}\in Q} \left( \I 1_A(\vvec{u}) - \I 1_B(\vvec{u}) -\fout{f_{\C G_{i-1}}}(\vvec{u})\right)\phi_{\vvec{u},Q}.
\end{equation}

 Note that $f_{\C G_i}$ and $\theta_{\C G_i}$ do not depend on $m$ (for $m\ge i$). Inside each $2^i$-cube $Q\in\C G_i$, the increment flow $\theta_{\C G_i}$ spreads the current demand error $\I 1_A - \I 1_B-\fout{f_{\C G_{i-1}}}$ uniformly inside each congruence class modulo $2^{i-1}$. 
Of course, if the error $\I 1_A - \I 1_B-\fout{f_{\C G_{i-1}}}$ is constant on each cube in  $\mathcal{P}(Q)$ for some $Q\in\C G_i$ then $\theta_{\C G_i}$ spreads this error evenly inside~$Q$. Thus an easy induction on $i=1,2,\ldots,m$ shows that $\I 1_A - \I 1_B-\fout{f_{\C G_i}}$ is constant on every cube $Q\in \C G_i$. Moreover, since $f_{\C G_i}$ is zero on the edge boundary of $Q\in \C G_i$ (and thus sends no flow out of $Q$), this constant is~$\xi(Q)$, that is,
 \begin{equation}\label{eq:FoutFiGi}
 \I 1_A(\vvec{u}) - \I 1_B(\vvec{u})-\fout{f_{\C G_i}}(\vvec{u}) =\xi(Q),\quad \mbox{for all
 $Q\in\C G_i$ and $\vvec{u}\in Q$}.
 \end{equation}
 Thus we have
\begin{equation}\label{eq:ThetaIGm2}
\theta_{\C G_i}:= \sum_{Q\in \C G_i} \sum_{C\in\mathcal{P}(Q)}\xi(C) \sum_{\vvec{u}\in C} \phi_{\vvec{u},Q},\quad\mbox{for every $1\le i\le m$}
\end{equation}
  (and for $i=m$ we arrive at the definition in~\eqref{eq:ThetaM}).

It follows from~\eqref{eq:ThetaM} and~\eqref{eq:ThetaIGm2} that, inside the component $[\vvec{v}]_{G_d}$, the flow $\theta_i$  for each $1\le i\le m$ is the expectation of $\theta_{\C G_i}$ over a uniformly random $2^m$-grid $\C G_m$, since $\C G_i=\mathcal{P}^{m-i}(\C G_m)$ is a uniformly random $2^i$-grid.
Note that, by the linearity of expectation, $f_m$ inside $[\vvec{v}]_{G_d}$ is the expectation of~$f_{\C G_m}$.

Now we are ready to verify all four conclusions, \eqref{eq:fmRational}--\eqref{eq:local}, of Lemma~\ref{outline:lem:realFlows}.

Conclusion~\eqref{eq:fmRational} states that $2^{2dm} f_m$ is integer-valued. Indeed, the only non-integer factors in~\eqref{eq:ThetaM} in addition to $2^{-dm}$ are $2^{-d(m-1)}$ (from the definition of $\xi(C)$) and $2^{-d}$ (from the definition of~$\phi_{\vvec{u},Q}$), proving~\eqref{eq:fmRational}.

Next, let estimate how close $\fout{f}_m$ is from the desired demand $\ind_A-\ind_B$. 
By taking the expectation over $\C G_m$ of~\eqref{eq:FoutFiGi} for $i:=m$, we get that
 \begin{equation}\label{eq:foutm-1}  \I 1_A(\vvec{u}) - \I 1_B(\vvec{u})-\fout{f_{m}}(\vvec{u}) = \frac1{2^{dm}}\sum_{Q\in \mathcal{Q}_{2^{m}}(\vvec{u})}\xi(Q),
 \end{equation}
 for all $m\ge 0$ and $\vvec{u}\in [\vvec{v}]_{G_d}$. As $\vvec{v}\in\I T^k$ was arbitrary, \eqref{eq:foutm-1} holds for every $\vvec{u}\in \I T^k$.
Now, \eqref{eq:foutm-1} and~\eqref{eq:Q}
imply that, for every $\vvec{u}\in\I T^k$,
\[\left|\fout{f_m}(\vvec{u}) - \I 1_A(\vvec{u}) +\I 1_B(\vvec{u}) \right| = \left|\sum_{Q\in \mathcal{Q}_{2^m}(\vvec{u})}\frac{\xi(Q)}{2^{dm}}\right| \leq
2^{dm}\cdot \frac{2c}{2^{m(1+\varepsilon)}}\cdot\frac1{2^{dm}}= \frac{2c}{2^{m(1+\varepsilon)}},\]
which is exactly the bound of Conclusion~\eqref{eq:AtoB}.

For \eqref{eq:fmChange}, we need to compute the maximum flow along any given edge under $\theta_m$. Let us re-write the right-hand side of~\eqref{eq:ThetaM} using only straight-line paths:
\begin{equation}\label{eq:ThetaM2}
\theta_m= \frac{1}{2^{dm}} \sum_{\vvec{u}\in\I T^k}  \sum_{C\in \mathcal{Q}_{2^{m-1}}(\vvec{u})}\xi(C)
\sum_{\vec{\gamma}\in\{-1,0,1\}^d\atop \vec{\gamma}\not=(0,\dots,0)}
2^{|\{i\mid \gamma_i=0\}|}\, \frac{\phi_{\vvec{u},2^{m-1}\vec{\gamma}}}{2^d}.
\end{equation}
(Note that, for a discrete $2^{m-1}$-cube $C$, the quantity $2^{|\{i\mid \gamma_i=0\}|}$ is exactly the number of discrete $2^m$-cubes $Q$ such that $\mathcal{P}(Q)$ contains both $Q$ and $(2^{m-1}\vec{\gamma})\cdot_a Q$.) For any given edge $\vvec{v}\vvec{w}$, there is a unique $\vec{\gamma}\in\{-1,0,1\}^d$ such that $\vvec{w}=\vec{\gamma}\cdot \vvec{v}$. The number of vertices $\vvec{u}$ such that either $\phi_{\vvec{u},2^{m-1}\vec{\gamma}}$ or $\phi_{\vvec{u},-2^{m-1}\vec{\gamma}}$ are non-zero on the pair $(\vvec{v},\vvec{w})$ is precisely $2\cdot 2^{m-1} = 2^m$. Since $\vec{\gamma}$ is non-zero, we have $2^{|\{i\mid \gamma_i=0\}|}\le 2^{d-1}$. So, by 
\eqref{eq:Q} and~\eqref{eq:ThetaM2}, we have
\[|\theta_m(\vvec{v},\vvec{w})|\leq 2^{-dm}\cdot 2^m\cdot 2^{(m-1)d}\cdot
\frac{2c}{2^{(m-1)(1+\varepsilon)}}\cdot\,\frac{2^{d-1}}{2^d} = \frac{2c}{2^{d+\varepsilon(m-1)}}.\]
 Thus indeed, as stated by Conclusion~\eqref{eq:fmChange}, we have $\|f_m-f_{m-1}\|_\infty\leq \frac{2c}{2^{d+\varepsilon(m-1)}}$.
 
Finally, let us prove \eqref{eq:local} which states that $f_m$ is a $(2^m-1)$-local function of $A$ and~$B$. Since this is trivially true for the zero flow $f_0$, it is enough to prove that, for every $m\ge 1$, the flow $\theta_m$ is a $(2^m-1)$-local function of $A$ and $B$. 
Since every discrete $2^m$-cube containing a vertex~$\vvec{v}$ lies inside the $(2^m-1)$-neighbourhood of~$\vvec{v}$,
it follows from~\eqref{eq:ThetaM} that the values of the flow $\theta_{m}$ on the edges incident to $\vvec{v}$ depend only on intersections of $A$ and $B$ with $N_{2^m-1}[\vvec{v}]$. Thus $f_m=\theta_1+\ldots+\theta_m$ is indeed a $(2^m-1)$-local function of $A$ and $B$, as required.

This completes the proof of Lemma~\ref{outline:lem:realFlows}. \end{proof}

Note that the sequence $f_1,f_2,\ldots$ is clearly convergent as the expression on the right side of the inequality \eqref{eq:fmChange} is summable as a function of $m$. We can therefore define $f_\infty$ to be the pointwise limit of $f_1,f_2,\ldots$\,. The following corollary is immediate by applying the triangle inequality and summing the bound in \eqref{eq:fmChange}.

\begin{obs}
\label{obs:sumAll}
For any $m\geq0$,
\[
\|f_\infty-f_m\|_\infty\le \sum_{i=m}^{\infty}\|f_{i+1}-f_i\|_\infty 
\le \sum_{i=m}^{\infty} \frac{2c}{2^{d+\varepsilon i}}
= \frac{c\,2^{1+\varepsilon}}{2^{d+\varepsilon m}(2^\varepsilon-1)}.\qed\]
\end{obs}

\section{Toast Sequences}
\label{sec:covering}

Our next goal is to formally describe the construction of the toast sequence $(J_1,J_2,\ldots)$ from the proof of Lemma~\ref{outline:lem:JordanScaff} as well as some other toast sequences built around it.

\subsection{Some Preliminaries}
\label{sec:ToastPrelim}

Let us say that a set $A$ \emph{cuts} a set $B$ if both $B\setminus A$ and $B\cap A$ are non-empty.
The following definition is key to describing our constructions.

\begin{defn}
\label{def:closure}
For $i\ge1$, let $D=(D_1,D_2,\ldots,D_i)$ be a sequence of subsets of $\mathbbm{T}^k$ and let $b\ge0$. Define $\mathcal{C}_b(D)\subseteq\mathbbm{T}^k $ to be the set constructed as follows:
\begin{enumerate}
\stepcounter{equation}
\item initialise $\mathcal{C}_b(D):=D_i$,
\stepcounter{equation}
\item\label{it:add} while there exist $1\leq j\leq i-1$ and $S\in\comp(D_{j})$ such that $\mathcal{C}_b(D)$ cuts $N_b[S]$,
add all vertices of $N_b[S]$ to $\mathcal{C}_b(D)$.
\end{enumerate}
\end{defn}

It may help the reader to have  the following 
informal description of the step in~\eqref{it:add} in mind: the current set $\mathcal{C}_b(D)$ iteratively swallows each set $N_b[S]$ that it intersects (unless the whole set $N_b[S]$ is already inside it). Note that the order in which we perform the operations in~\eqref{it:add} does not affect the final set~$\mathcal{C}_b(D)$.
Intuitively, Definition~\ref{def:closure} is designed so that the boundaries of the components of the final set $\mathcal{C}_b(D)$ are well separated from $D_1\cup \dots\cup D_{i-1}$.  Unfortunately, neither some separation between the components of the new set $\mathcal{C}_b(D)$  nor their boundedness holds automatically (as two distinct components of $D_i$ may grow too close to each other, or even merge); this will require proofs based on various extra properties of the input sets~$D_j$.

The next lemma will be used to demonstrate that if the diameter of each component of $G_d\induced D_j$ is not too large and the distance between them is not too small, then the obtained set $\mathcal{C}_b(D)$ can be added without violating Properties~\eqref{scaffBdd} and \eqref{scaffDist} from the definition of a toast sequence. Note that the lemma holds without assuming any bound on the distance between distinct components of $G_d\induced D_i$.


\newcommand{\NewQ}{b}
\newcommand{\NewT}{a}

\begin{lemma}
\label{lem:chain}
Let $D=(D_1,D_2,\ldots,D_i)$ be a sequence of subsets of $\mathbbm{T}^k$, let $b\ge0$, and let $(\NewT_1,\NewT_2,\ldots,\NewT_{i-1})$ and $(\NewQ_0,\NewQ_1,\ldots,\NewQ_{i-1})$ be sequences of non-negative integers such that
\begin{equation}\label{eq:chain}
 \NewQ_j\ge \NewT_{j}+2\NewQ_{j-1}+2b,\quad\mbox{for all $1\leq j\leq i-1$}.
 \end{equation}
If, for every $1\leq j\leq i-1$, every component of $G_d\induced D_j$ has diameter at most $\NewT_j$ in $G_d$ and the distance in $G_d$ between any two components of $G_d\induced D_j$ is greater than $\NewQ_{j-1}+2b$, then $\mathcal{C}_b(D)\subseteq N_{\NewQ_{i-1}}[D_i]$. 
\end{lemma}

\begin{proof}
Take any vertex $\vvec{w}\in \mathcal{C}_b(D)\setminus D_i$.
By definition of $\mathcal{C}_b(D)$, there exists a sequence $S_1,S_2,\ldots,S_n$ of distinct subsets of $\mathbbm{T}^k$ such that
\begin{enumerate}
\stepcounter{equation}
\item for each $1\leq \ell\leq n$, there exists $1\leq j\leq i-1$ such that $S_\ell$ is a component of $G_d\induced D_j$,
\stepcounter{equation}
\item $\dist_{G_d}(S_1,D_i)\leq b$,
\stepcounter{equation}
\item $\dist_{G_d}(S_\ell,S_{\ell+1})\leq 2b$ for $1\leq \ell < n$, and
\stepcounter{equation}
\item $\dist_{G_d}(S_n,\vvec{w})\leq b$. 
\end{enumerate}
Given such a vertex $\vvec{w}$ and sets $S_1,\ldots,S_n$, let $\vvec{u}\in D_i$ be such that $\dist_{G_d}(\vvec{u},S_1)\leq b$. Our aim is to prove, by induction on $i$, that, for every such $\vvec{w}$ and $\vvec{u}$, we have $\dist_{G_d}(\vvec{u},\vvec{w})\leq \NewQ_{i-1}$. In the base case $i=1$, the statement is true vacuously as $\mathcal{C}_b(D)=D_1$ (and so no such $\vvec{w}$ can exist). 

So, suppose that $i\ge2$. If none of the sets $S_1,\ldots,S_n$ is a component of $G_d\induced D_{i-1}$, then the sequence $S_1,\ldots,S_n$ actually certifies that $\vvec{w}\in \mathcal{C}_b(D_1,D_2,\ldots,D_{i-2},\vvec{u})$ and we have that $\dist_{G_d}(\vvec{u},\vvec{w})\leq \NewQ_{i-2}\le \NewQ_{i-1}$ by induction on $i$. 

Next, suppose that there is a unique index $\ell$ such that $S_\ell$ is a component of $G_d\induced D_{i-1}$. Let $\vvec{y}$ be a vertex of $S_\ell$ at distance at most $2b$ from $S_{\ell-1}$ (or, at distance at most $b$ from $\vvec{u}$ in the case $\ell=1$) and let $\vvec{z}$ be a vertex of $S_\ell$ at distance at most $2b$ from $S_{\ell+1}$ (or at distance at most $b$ from $\vvec{w}$ in the case $\ell=n$). The sequences $S_{\ell-1},S_{\ell-2},\ldots S_1$ and $S_{\ell+1},S_{\ell+2},\ldots, S_n$ certify respectively that $\vvec{u}\in \mathcal{C}_b(D_1,D_2,\ldots, D_{i-2},N_b[\vvec{y}])$ and $\vvec{w}\in\mathcal{C}_b(D_1,D_2,\ldots, D_{i-2},N_b[\vvec{z}])$. So, by the inductive hypothesis and~\eqref{eq:chain},
\begin{eqnarray*}
\dist_{G_d}(\vvec{u},\vvec{w})&\leq& \dist_{G_d}(\vvec{u},\vvec{y})+\dist_{G_d}(\vvec{y},\vvec{z})+\dist_{G_d}(\vvec{z},\vvec{w})\\
&\leq& \left(\dist_{G_d}(\vvec{u},N_b[\vvec{y}]) +b\right)
+ \NewT_{i-1} + \left(\dist_{G_d}(N_b[\vvec{z}],\vvec{w})+b\right) \\
&\leq& \left(\NewQ_{i-2} +b\right)+ \NewT_{i-1} + \left(\NewQ_{i-2} +b\right)\ \leq\ \NewQ_{i-1},
\end{eqnarray*}
as desired.

Finally, we consider the remaining case that there are two indices $1\leq \ell < \ell'\leq n$ such that $S_{\ell}$ and $S_{\ell'}$ are distinct components of $G_d\induced D_{i-1}$. Choose these indices so that $\ell'-\ell$ is minimised. Let $\vvec{y}$ be a vertex of $N_b[S_{\ell}]$ which is at distance at most $b$ from $S_{\ell+1}$ and let $\vvec{z}$ be a vertex of $N_b[S_{\ell'}]$ at distance at most $b$ from $S_{\ell'-1}$. The sequence $S_{\ell+1},S_{\ell+2},\ldots,S_{\ell'-1}$ certifies that $\vvec{z}$ is in $\mathcal{C}_b(D_1,\ldots,D_{i-2},\{\vvec{y}\})$. By the inductive hypothesis, 
\[\dist_{G_d}(S_{\ell},S_{\ell'})\leq \dist_{G_d}(\vvec{y},\vvec{z})+2b\leq \NewQ_{i-2}+2b,\]
which contradicts the hypotheses of the lemma. This completes the proof.
\end{proof}

\subsection{Constructions of Toast Sequences}
\label{sec:ToastConstructions}

We now construct some toast sequences that we use in the proof of Theorem~\ref{th:main}. We remark that, as constructed, they will not immediately cover all of~$\mathbbm{T}^k$. However, at the end of the section, we will give a compactness argument of Boykin and Jackson~\cite{BoykinJackson07} (see also Marks and Unger~\cite{MarksUnger17}*{Lemma~A.2}) which allows us to produce one toast sequence that covers all of the vertices (see Lemma~\ref{lm:FinalToast}).

We will use the global parameters $r_i,r_i',q_i,q_i',t_i,t_i'$ that were defined in Section~\ref{subsec:setting}. 
For the reader's convenience, we repeat these definitions here. Namely, we set
\begin{align}
\tag{\ref{eq:ri'specific}}
r_{i}'&:=100^{2^{i+1} - 1},\quad i\ge 0,\\
r_i&:= 100^{2^{i+1} - 2},\quad i\ge 1.
\tag{\ref{eq:rispecific}}
\end{align}
Thus $r_0'=100$ while $r_i=(r_{i-1}')^2$ and $r_i'=100r_i$ for each $i\ge 1$. Then we defined $q_0':=0$ and, for $i\ge 1$,
 \begin{align}
t_i&:=2r_i+4q'_{i-1}+4,\nonumber\\
q_i&:=t_i+2q'_{i-1}+4,\nonumber\\
t'_i&:=4r'_i/5 + 2q_i,\nonumber\\
q'_i&:=t'_i+2q_i+4.\tag{\ref{eq:qi'}}
\end{align}
 Note that, within additive $O(\sqrt{r_i})$ as $i\to\infty$, the last four values are $2r_i$, $2r_i$, $84r_i$ and $88r_i$ respectively. 
 It is tedious, but not hard, to verify that, with the specific definitions above (or for the sufficiently fast-growing sequences as in Remark~\ref{rm:FastRi}, with $r_i'$ divisible by $5$), the following inequalities hold for all $i\ge 1$:
\begin{align}
\label{eq:qtMonotone}
r_i'-2&>q_i'>t_i'>q_i>t_i>r_i>5r'_{i-1},\\
\label{eq:r'q}
r_i'&\ge 15q_i+25,\\
\label{eq:qjAnotherIneq}
q_{i+1}'&\ge 2r_{i+1}+6q_{i}+4.
\end{align}
These are some key inequalities that will be used when establishing various properties of the constructed sets in Section~\ref{sec:ToastProperties}.

Recall that, for each $i\ge1$,   $X_i$ is maximally $r_i$-discrete and $Y_i$ is maximally $r'_i$-discrete such that both these sets are the unions of finitely many disjoint strips.
 For each $i\ge1$, define 
 \[
 I_i:=\bigcup_{\vvec{u}\in X_i} I_{i}^{\vvec{u}},
 \]
  where $I_{i}^{\vvec{u}}$ denotes the set of all $\vvec{v}\in \mathbbm{T}^k$ such that
 \[\dist_{G_d}(\vvec{v},\vvec{u}')\ge\dist_{G_d}(\vvec{v},\vvec{u})+5 r'_{i-1},\quad\mbox{for every $\vvec{u}'\in X_i\setminus \{\vvec{u}\}$.}
 \]
  The set $I_i^{\vvec{u}}$ can be viewed as the ``partial Voronoi cell'' of~$\vvec{u}\in X_i$.
  
 Now, inductively for $i=1,2,\ldots$, define
 \begin{eqnarray}
 J_i&:=&\mathcal{C}_2(N_{q'_0+2}[J_1],\ldots,N_{q'_{i-2}+2}[J_{i-1}],I_i),\label{eq:Ji}\\
 K_i&:=&\mathcal{C}_2(K_1,L_1,\ldots,K_{i-1},L_{i-1},N_2[J_i])\label{eq:Ki},\\
L_i&:=&\mathcal{C}_2(K_1,L_1,\ldots,K_{i-1},L_{i-1},K_i,N_{ 2r'_i/5}[Y_i]).\label{eq:Li}
 \end{eqnarray}

We will show that  $(J_1,K_1,J_2,K_2,\ldots)$ and $(K_1,L_1,K_2,L_2,\ldots)$ are toast sequences. In particular, this implies that $(J_1,J_2,\ldots)$ is a toast sequence.
 Along the way, we will also prove some key properties of these sequences of sets which will be applied later in the paper to prove Theorem~\ref{th:main}.

\subsection{Properties of the Constructed Sequences}
\label{sec:ToastProperties}
We begin with the following statement, which simply follows by construction.

\begin{lemma}
\label{lem:JKLtrivial}
For every $i\ge1$, we have $I_i\subseteq J_i$, $N_2[J_i]\subseteq K_i$ and $N_{2r'_i/5}[Y_i]\subseteq L_i$.\qed
\end{lemma}

Next, we prove a lemma regarding the structure of $I_i$. 

\begin{lemma}
\label{lem:Ii}
The sets $I_i^{\vvec{u}}$ for $\vvec{u}\in X_i$ are exactly the components of $G_d\induced I_i$ and the distance in $G_d$ between any two of these sets is at least $5r'_{i-1}$. Also, for each $\vvec{u}\in X_i$, the diameter of $G_d\induced I_i^{\vvec{u}}$ is at most $2r_i$ and $I_i^{\vvec{u}}\cap X_i=\{\vvec{u}\}$. Furthermore, $I_i$ is a $2r_i$-local function of $X_i$.
\end{lemma}

\begin{proof}
Take any $\vvec{u}\in X_i$ and $\vvec{v}\in I_i^{\vvec{u}}$. It holds by $r_i\ge 5r'_{i-1}$ that $\vvec{u}\in I_i^{\vvec{u}}$. Every vertex on a shortest path from $\vvec{v}$ to $\vvec{u}$ in $G_d$ is contained in $I_i^{\vvec{u}}$ so this set is connected. Since $X_i$ is maximally $r_i$-discrete and $\vvec{u}$ is the element of $X_i$ at minimum distance from $\vvec{v}$, we have that $\dist_{G_d}(\vvec{u},\vvec{v})\leq r_i$. Thus, the diameter of $G_d\induced I_i^{\vvec{u}}$ is at most $2r_i$.

Now, for distinct $\vvec{u},\vvec{w}\in X_i$ in the same component of $G_d$, let $\vvec{v}_0\vvec{v}_1\ldots \vvec{v}_n$ be the shortest path in $G_d$ from $I_i^{\vvec{u}}$ to $I_i^{\vvec{w}}$. Then we have by definition that
\[\dist_{G_d}\left(\vvec{v}_0,\vvec{u}\right)+2n\ge \dist_{G_d}(\vvec{v}_n,\vvec{u})+n\ge\dist_{G_d}(\vvec{v}_n,\vvec{w})+5r'_{i-1}+n\]
\[\ge \dist_{G_d}(\vvec{v}_0,\vvec{w})+5 r'_{i-1}\ge \dist_{G_d}\left(\vvec{v}_0,\vvec{u}\right)+10r'_{i-1}\]
and so $n\ge 5r'_{i-1}$. Therefore, the sets $I_i^{\vvec{u}}$, for $\vvec{u}\in X_i$, are the components of $G_d\induced I_i$ and the distance in $G_d$ between any two such components is at least $5r'_{i-1}$. 

Finally, in order to decide if some $\vvec{v}\in\I T^k$ belongs to $I_i$, it is enough to compute the distance $d$ from $\vvec{v}$ to $X_i$ in $G_d$ and then check if the $(d+5r'_{i-1}-1)$-neighbourhood of $\vvec{v}$ contain a unique element of~$X_i$. Since $d\le r_i$, the set $I_i$ is a local function of $X_i$ of radius $r_i+5r'_{i-1}\le 2r_i$, finishing the proof of the lemma.
\end{proof}

The following lemma holds because $Y_i$ is maximally $r'_i$-discrete; we omit its easy proof.

\begin{lemma}
\label{lem:2/5}
Every component of $G_d\induced N_{2r'_i/5}[Y_i]$ has diameter at most $4r_i'/5$ and contains exactly one vertex of~$Y_i$. Also, the distance in $G_d$ between any two such components is at least~$r_i'/5$. \qed
\end{lemma}

The proofs of the next two lemmas are the main applications of Lemma~\ref{lem:chain}.

\begin{lemma}
\label{lem:Ji}
For $i\ge1$, we have that $J_i\subseteq N_{q'_{i-1}}[I_i]$, the distance in $G_d$ between any two distinct components  of $G_d\induced J_i$ is least $5r_{i-1}'-2q_{i-1}'$, and every component  of $G_d\induced J_i$ has diameter at most $2r_i+2q'_{i-1}$ and contains exactly one vertex of~$X_i$.
\end{lemma}

\begin{proof}
We proceed by induction on $i$. The base case $i=1$ follows from Lemma~\ref{lem:Ii} as $J_1=I_1$ and $q'_0=0$.  Let $i\ge2$. We would like to apply Lemma~\ref{lem:chain} with $b:=2$ to the sequences 
\begin{equation}
\label{eq:seqs}
 \mbox{$(N_{q'_0+2}[J_1],\ldots,N_{q'_{i-2}+2}[J_{i-1}],I_i)$, $(t'_1,\ldots,t'_{i-1})$ and $(q'_0,\ldots,q'_{i-1})$.}
 \end{equation} 

Let us check the assumptions of Lemma~\ref{lem:chain}. By induction, for every $1\leq j\leq i-1$, every component of $G_d\induced N_{q'_{j-1}+2}[J_j]$ has diameter at most $(2r_{j}+2q'_{j-1})+2(q'_{j-1}+2)= t_j<t'_j$
and the distance in $G_d$ between any two such components is at least $(5r'_{j-1}-2q'_{j-1})-2(q'_{j-1}+2)$, which is strictly larger than $q'_{j-1}+4$ by \eqref{eq:qtMonotone}. Also, the distance between any two components of $G_d\induced I_i$ is at least $5r'_{i-1}>q'_{i-1}+4$ while the inequality in~\eqref{eq:chain} for $1\le j\le i-1$ (which states that $q'_j\ge t'_j+2q'_{j-1}+4$) holds by $q_{j}>q'_{j-1}$.
Thus Lemma~\ref{lem:chain} applies to the sequences in~\eqref{eq:seqs} and gives that $J_i\subseteq N_{q'_{i-1}}[I_i]$, proving the first stated property. 

This routinely implies all other claims about~$J_i$. Indeed, it follows from Lemma~\ref{lem:Ii} that the distance in $G_d$ between different components of $G_d\induced J_i$ is at least $5r'_{i-1}-2q'_{i-1}$ (note that $5r'_{i-1}- 2q'_{i-1}\ge 2$) and each component has diameter at most $2r_i+2q'_{i-1}$. Finally, each component $S\in \comp(J_i)$ is built from $I_i^{\vvec{u}}$ for some $\vvec{u}\in X_i$ and, by the separation bounds proved above, $\vvec{u}$ is the unique vertex in~$S\cap X_i$.
\end{proof}

\begin{lemma}
\label{lem:KiLi}
For each $i\ge1$, it holds that
\begin{enumerate}
\stepcounter{equation}
 \item\label{eq:KiStructure} 
 $K_i\subseteq N_{q'_{i-1}+2}[J_i]$,  the distance in $G_d$ between any two distinct components of $G_d\induced K_i$ is least $5r_{i-1}'-4q_{i-1}'-4$, and  every component of $G_d\induced K_i$ has diameter at most $t_i$ and contains exactly one vertex of~$X_i$.
\stepcounter{equation}
 \item \label{eq:LiStructure}
 $L_i\subseteq N_{2r'_i/5+q_{i}}[Y_i]$, the distance in $G_d$ between any two distinct components  of $G_d\induced L_i$ is least $r_i'/5-2q_i$, and 
  every component of  $G_d\induced L_i$ has diameter at most  $t'_i$ and contains exactly one vertex of~$Y_i$. 
 \end{enumerate}
\end{lemma}

\begin{proof}
We use induction on $i$. The base case $i=1$ holds for $K_i$ by Lemma~\ref{lem:Ji} since $K_1=N_{2}[J_1]$ and $q_0'=0$. Let $i\ge 2$. In the $i$-th step of induction, we prove the claims about $L_{i-1}$ and $K_i$. 
Analogously to the proof of Lemma~\ref{lem:Ji}, we will apply the inductive hypotheses and Lemma~\ref{lem:chain} with $b:=2$ to respectively 
\begin{itemize}
\item $(K_1,L_1,\ldots,K_{i-1},N_{2r'_{i-1}/5}[Y_{i-1}])$, $(t_1,t_1',\ldots,t'_{i-2},t_{i-1})$ and $(q_0',q_1,\ldots,q'_{i-2},q_{i-1})$,
\item $(K_1,L_1,\ldots,L_{i-1},N_2[J_i])$, $(t_1,t_1',\ldots,t_{i-1},t'_{i-1})$ and $(q_0',q_1,\ldots,q_{i-1},q'_{i-1})$.
\end{itemize}

Let us check the assumptions of Lemma~\ref{lem:chain} for the first triple of sequences.
By induction, 
for every $1\leq j\leq i-1$,
every component of $G_d\induced K_j$ has diameter at most $t_j$ and the distance in $G_d$ between any two such components is at least $5r'_{j-1}-4q'_{j-1}-4$, which is strictly more than $q'_{j-1}+4$ by \eqref{eq:qtMonotone}. Similarly, for each $1\leq j\leq i-2$,  every component of $G_d\induced L_{j}$ has diameter at most $t'_{j}$ and the distance in $G_d$ between any two such components is at least $r'_{j}/5-2q_{j}$, which is strictly larger than $q_{j}+4$ by~\eqref{eq:r'q}. Also, 
the inequalities in~\eqref{eq:chain} are routine to verify.

So,  Lemma~\ref{lem:chain} applies and gives that $L_{i-1}\subseteq N_{q_{i-1}}[N_{2r'_{i-1}/5}[Y_{i-1}]]$. This implies all other stated properties of~$L_{i-1}$. Indeed, by Lemma~\ref{lem:2/5}, the distance between any two components of $G_d\induced L_{i-1}$, is at least $r'_{i-1}/5-2q_{i-1}$,
and
each component has diameter at most  $4r'_{i-1}/5+2q_{i-1}= t'_{i-1}$ and contains exactly one vertex from $Y_i$, proving all claims about~$L_{i-1}$.

By $r'_{i-1}/5-2q_{i-1}> q_{i-1}+4$, we can apply Lemma~\ref{lem:chain} also to the second triple of sequences to derive that $K_i\subseteq N_{q'_{i-1}+2}[J_i]$. It follows from Lemma~\ref{lem:Ji} that the distance between any two components of $G_d\induced K_i$ is at least $(5r_{i-1}'-2q_{i-1}')-2(q'_{i-1}+2)=5r'_{i-1}-4q'_{i-1}-4$, while each component has diameter at most $2r_i+2q'_{i-1}+2(q'_{i-1}+2)= t_i$ and contains exactly one element of~$X_i$.
\end{proof}

We will need to refer a few times to the following, not entirely trivial result so we state it as a separate lemma.

\begin{lemma}\label{lm:KLKJ} Let $i>j\ge 1$.
 If $S'\in \comp(K_j)$ and $T\in \comp(J_i)$ are at distance at most $2$ in $G_d$, then $N_2[S']\subseteq T$.
\end{lemma}

\begin{proof} 
Observe that, by Lemma~\ref{lem:KiLi}, we have $S'\subseteq K_j\subseteq N_{q'_{j-1}+2}[J_j]$. Take a component $S$ of $G_d\induced J_j$ such that $S'':= N_{q'_{j-1}+2}[S]$ intersects~$S'$. Let us show that, in fact, $S'\subseteq S''$.
Suppose on the contrary that we have some $\vvec{u}\in S'\setminus S''$. Again by Lemma~\ref{lem:KiLi}, $\vvec{u}$ is within distance $q'_{j-1}+2$ to some component $X\in\comp(J_j)$ that has to be different from~$S$. But then the distance between $X$ and $S$ is, by the triangle inequality when we go via $S'$, at most $2(q'_{j-1}+2)+t_i<5r'_{j-1}-2q'_{j-1}$, contradicting Lemma~\ref{lem:Ji}.

By the definition of $J_i$, every component of $N_{q'_{j-1}+2}[J_j]$ that is within distance $2$ of $T\in \comp(J_i)$ is included with its 2-neighbourhood into~$T$. Since $S''\subseteq N_{q'_{j-1}+2}[J_j]$, the set $N_2[S'']$ is included into~$T$. Thus $T$ contains $N_2[S']\subseteq N_2[S'']$, as desired.\end{proof}

We are now in position to prove the following lemma.

\begin{lemma}
\label{lem:JKLscaff}
Both of $(J_1,K_1,J_2,K_2,\ldots)$ and $(K_1,L_1,K_2,L_2,\ldots)$ are toast sequences.
\end{lemma}

\begin{proof}
The fact that the components of the subgraphs of $G_d$ induced by each of $J_i$, $K_i$ and $L_i$ are of uniformly bounded cardinality and are pairwise separated by distance at least three follows from 
Lemmas~\ref{lem:Ji} and~\ref{lem:KiLi}.
So, Properties \eqref{scaffBdd} and \eqref{scaffDist} of Definition~\ref{def:scaff} hold for each of the two sequences. 

Let us check Property \eqref{scaffUt} for the sequence $(J_i,K_i)_{i=1}^\infty$. Note some asymmetry in \eqref{scaffUt}, which states that if two components from two different times are at distance at most 2 then the 2-neighbourhood of the \emph{earlier} one is contained in the \emph{later} one. There are some cases to consider. First, take any $S\in\comp(J_i)$. Suppose on the contrary to \eqref{scaffUt} that $S$ is at distance at most 2 from $X=J_j$ for $j>i$ or $X=K_j$ for $j\ge i$ but $N_2[S]$ is a not a subset of~$X$. Let $T\in\comp(X)$ satisfy $\dist_{G_d}(S,T)\le 2$. By construction, $T$ does not cut $N_2[S]$ as otherwise this set would be added into~$T$ (since $T$ was built after~$S$). Hence we have that $T\subseteq N_2[S]$. 
This is impossible for $j>i$ because the diameter of $S$ is at most $2r_i+2q_{i-1}'$ by Lemma~\ref{lem:Ji} while $T$ contains a ball of radius $\lfloor (r_j-5r_{i-1}')/2\rfloor$ centred at the unique point of $X_j\cap T$ by the definition of $I_j\subseteq J_j\subseteq K_j$. Also, the case $X=K_i$ leads to a contradiction. Indeed, when constructing $T\in\comp(K_i)$ we started with $N_2[T']$ for some $T'\in\comp(J_i)$, By Lemmas~\ref{lem:Ii}, \ref{lem:Ji} and~\ref{lem:KiLi}, we must have $T'=S$. Thus $K_i\supseteq T\supseteq N_2[T']$ contains $N_2[S]$, a contradiction. Second, take any $S\in\comp(K_i)$. No set $J_j$ with $j>i$ can violate \eqref{scaffUt} by Lemma~\ref{lm:KLKJ}. Also, any $T\in\comp(K_j)$ with $j>i$ and $\dist_{G_d}(S,T)\le 2$ contains $N_2[S]$ (again since $T\subseteq N_2[S]$ is impossible by the diameter argument). Thus~\eqref{scaffUt} holds for $(J_i,K_i)_{i=1}^\infty$. 

Likewise, Property~\eqref{scaffUt} can be verified for the sequence $(K_i,L_i)_{i=1}^\infty$, using additionally the fact that each component of $G_d\induced L_i$ contains a ball of radius $2r_i'/5$.

Thus each of $(J_i,K_i)_{i=1}^\infty$ and $(K_i,L_i)_{i=1}^\infty$ is a toast sequence.
\end{proof}

For each $i\ge 1$, define
 $\Tilde{K}_i$ from $K_i$ (resp.\ $\Tilde{L}_i$ from $L_i$) by removing the vertex sets of all components $S$ such that there is $j>i$ with $J_j$ (resp.\ $K_j$) being at distance at most $2$ from~$S$ in the graph~$G_d$. By Lemma~\ref{lm:KLKJ}, it holds that, in fact,
 \begin{equation}\label{eq:TildeKj}
 \Tilde{K}_i=K_i\setminus\left(\bigcup_{j=i+1}^\infty J_j\right),
 \end{equation}
  that is, we could have equivalently defined $\Tilde{K}_i$ from $K_i$ by removing all individual vertices that belong to some~$J_j$ with $j>i$. Likewise (with the claim following trivially from our definitions), we have
   \begin{equation}\label{eq:TildeLj}
   \Tilde{L}_i=L_i\setminus\left(\bigcup_{j=i+1}^\infty K_j\right).
   \end{equation}

The following lemma records some properties of these sets, in particular showing that
they together with $J_i$'s form a toast sequence without losing any single vertex of $\bigcup_{i=1}^\infty L_i$. 
Moreover, 
the structure of the toast sequence $(J_1,J_2,\ldots)$ is preserved in a very strong way.

\begin{lemma}\label{lm:Tilde} 
The following properties hold.

\begin{enumerate}
\stepcounter{equation}
 \item\label{it:TildeComponents} For every $i\ge 1$, we have $\comp(\Tilde{K}_i)\subseteq \comp(K_i)$ and
 $\comp(\Tilde{L}_i)\subseteq \comp(L_i)$.
 \stepcounter{equation}
\item\label{it:Tilde2} $\bigcup_{i=1}^\infty (J_i\cup \Tilde{K}_i\cup\Tilde{L}_i)\supseteq \bigcup_{i=1}^\infty L_i$. 
\stepcounter{equation}
 \item \label{it:Tilde4} For every integers $i,j\ge 1$, if $T\in\comp(J_j)$ and $S\in \comp(\Tilde{L}_i)$ are at distance at most $2$, then $i\ge j$ and  $N_2[T]\subseteq S$.
 \stepcounter{equation}
\item\label{it:Tilde3} The sequence $(J_1,\Tilde{K}_1,\Tilde{L}_1,J_2,\Tilde{K}_2,\Tilde{L}_2,\ldots)$ is a toast sequence. 
 \stepcounter{equation}
 \end{enumerate}
 \end{lemma}

 \begin{proof}
  Conclusion~\eqref{it:TildeComponents} is a trivial consequence of the definition of $\Tilde{K}_i$ and~$\Tilde{L}_i$.
  
  For~\eqref{it:Tilde2}, take any $\vvec{u}\in L_i\setminus \Tilde{L}_i$. We have to show that
  $\vvec{u}\in \bigcup_{j=1}^\infty (J_j\cup \Tilde{K}_j)$.
  By~\eqref{eq:TildeLj}, there is $j>i$ with $\vvec{u}\in K_j$. If $\vvec{u}\in \Tilde{K}_j$ then we are done; otherwise the vertex $\vvec{u}$ lies
  inside some $J_h$ with $h>j$  by~\eqref{eq:TildeKj}, proving~\eqref{it:Tilde2} in either case.
 
For~\eqref{it:Tilde4}, note that $i\ge j$ for otherwise the component $S$ of $L_i$ (recall that 
$\comp(\Tilde{L}_i)\subseteq \comp(L_i)$ by~\eqref{it:TildeComponents}) would be  removed when constructing  $\Tilde{L}_i$ from~$L_i$. 
By the definition of $K_j$, we have that $K_j\supseteq N_2[J_j]$. Thus $K_j$ contains every vertex of the connected set $N_2[T]$ which in turn intersects~$S$. Thus by the  construction of the set $L_i$, its component
$S$ must contain $N_2[T]$ as a subset, proving~\eqref{it:Tilde4}. 

For~\eqref{it:Tilde3}, note by~\eqref{it:TildeComponents} that 
no new components (and thus no new boundaries) 
are created when we construct $\Tilde{K}_j$  from $K_j$ and $\Tilde{L}_j$ from $L_j$. 
Since  $(J_1,K_1,J_2,K_2,\ldots)$ and $(K_1,L_1,K_2,L_2,\ldots)$ are toast sequences by Lemma~\ref{lem:JKLscaff},  the only remaining case that requires some checking is that no conflict to the definition of a toast sequence (namely, Property~\eqref{scaffUt} of Definition~\ref{def:scaff}) can come from $S\in \comp(\Tilde{L}_i)$ and $C\in \comp(J_j)$. This has already been taken care of by~\eqref{it:Tilde4}.
\end{proof}

\subsection{Some Structure of the Constructed Sets}

The next two lemmas will be useful in analysing the structure of the pieces in the final equidecomposition.


\begin{lemma}\label{lm:JiLocality} For each integer $i\ge 1$, the set $J_i$ is a $4r_i$-local function of $X_1,\dots,X_i$.\end{lemma}

\begin{proof} We use induction on $i$ with the base case $i=1$ following from $J_1=I_1$ being in fact an $(r_1+5r_0)$-local function of~$X_1$. Let $i\ge 2$. In order to decide if some $\vvec{v}\in\I T^k$ is included into $J_i$, it is enough to know all elements $\vvec{u}\in X_i$ at distance at most $r_i+q'_{i-1}$ from $\vvec{v}$ and the component of $G_d\induced J_i$ containing each such~$\vvec{u}$. By Lemma~\ref{lem:Ji}, the component $S\in\comp(J_i)$ containing $\vvec{u}\in X_i$ has diameter at most $2r_i+2q'_{i-1}$ and contains no other elements of $X_i$. So to build $S$ for given $\vvec{u}\in X_i$, we need to know only the sets $N_{q'_0+2}[J_1],\dots,N_{q'_{i-2}+2}[J_{i-1}]$ inside $N_{2r_i+2q'_{i-1}}[\vvec{u}]$. 
By induction, each of these sets is a $4r_{i-1}$-local function of $X_1,\dots,X_{i-1}$. The lemma follows since $(r_i+q'_{i-1})+(2r_i+2q'_{i-1})+4r_{i-1}\le 4r_i$.
\end{proof}

\begin{lemma}
\label{lem:JKLstrips}
Each of the sets $J_i,K_i$ and $L_i$ for $i\ge1$ can be expressed as a union of finitely many disjoint strips.
\end{lemma}
\begin{proof} Observe that each of the sets $J_i$, $K_i$ and $L_i$ is a local function of $X_1,Y_1,\ldots,X_i,Y_i$. Lemma~\ref{lm:JiLocality} proves this for $J_i$ and its argument can be easily adapted to $K_i$ and $L_j$ using the uniform diameter bounds of Lemma~\ref{lem:KiLi}.
Since every set $X_j$ and $Y_j$ is in turn a finite union of disjoint strips (by our choices in Section~\ref{sec:Setting2}) and such unions form an invariant algebra, the lemma holds by Observation~\ref{obs:LocalRule}.\end{proof}

Let us remark that the sets $\Tilde{K}_i$ and $\Tilde{L}_i$ need not satisfy Lemma~\ref{lem:JKLstrips}, that is, they are not in general finite unions of strips. 


\begin{lemma}
\label{lem:Jmeas}
We have
\[\lambda(I_{i}) = 1 - O\left(r'_{i-1}/r_i\right)\]
as $i\to\infty$, where the constant factor is bounded by a function of $d$ only. 
\end{lemma}

\begin{proof}
Note that $I_{i}$ is a union of finitely many disjoint strips and, therefore, it is measurable.  For any pair of distinct vertices $\vvec{u},\vvec{u}'\in X_i$, let $W_i(\vvec{u},\vvec{u}')$ be the set of all $\vvec{v}\in \mathbbm{T}^k$ such that
\[\max\left\{\dist_{G_d}(\vvec{u},\vvec{v}),\dist_{G_d}(\vvec{u}',\vvec{v})\right\}\leq r_i+5r'_{i-1}\text{ and}\]
\[\left|\dist_{G_d}(\vvec{u},\vvec{v})-\dist_{G_d}(\vvec{u}',\vvec{v})\right|\leq 5r'_{i-1}.\]
Define $W_i$ to be the union of $W_i(\vvec{u},\vvec{u}')$ over all pairs of distinct vertices $\vvec{u},\vvec{u}'\in X_i$. The set $W_i$ can be written as a Boolean combination of translates of $X_i$, and so it is measurable. By the definition of $I_i$ and the fact that $X_i$ is maximally $r_i$-discrete, we have that $\mathbbm{T}^k\setminus I_i\subseteq W_i$. So, it suffices to prove that $\lambda(W_i)= O(r'_{i-1}/r_i)$. 

First, for each $\vvec{u}\in X_i$, let us bound the number of $\vvec{u}'\in X_i$ such that $W_i(\vvec{u},\vvec{u}')$ is non-empty. Given any such $\vvec{u}'$, we have
\[\dist_{G_d}(\vvec{u},\vvec{u}')\leq 2r_i+10r'_{i-1}.\]
Therefore,
\[N_{\lfloor r_i/2\rfloor}[\vvec{u}']\subseteq N_{2r_i+\lfloor r_i/2\rfloor+10r'_{i-1}}[\vvec{u}].\]
Since the set $X_i$ is $r_i$-discrete, the sets $N_{\lfloor r_i/2\rfloor}[\vvec{u}']$ are disjoint for different $\vvec{u}'\in X_i$. So the number of such $\vvec{u}'$ is at most
\[\frac{\left(4r_i+2\lfloor r_i/2\rfloor+20r'_{i-1}+1\right)^d}{\left(2\lfloor r_i/2\rfloor+1\right)^d},\]
which can be upper  bounded above by a constant that depends on $d$ only.

Now, given a pair of distinct $\vvec{u},\vvec{u}'\in X_i$ in the same component of $G_d$, let $\vec{n}\in \mathbbm{Z}^d$ be such that $\vvec{u}'=\vec{n}\cdot_a\vvec{u}$. Note that $\|\vec{n}\|_\infty\ge r_i+1$ since $X_i$ is $r_i$-discrete. Then $W_i(\vvec{u},\vvec{u}')$ is the set of all $\vvec{v}\in \mathbbm{T}^k$ of the form $\vvec{v}=\vec{m}\cdot_a\vvec{u}$ where 
\begin{itemize}
\item $\|\vec{m}\|_\infty\leq r_i+5r'_{i-1}$,
\item $\|\vec{n}-\vec{m}\|_\infty\leq r_i+5r'_{i-1}$, and
\item $\big|\|\vec{m}\|_\infty-\|\vec{n}-\vec{m}\|_\infty\big|\leq 5r'_{i-1}$.
\end{itemize}
For fixed $d$, the number of vectors $\vec{m}$ of this type is $O\left(r_i^{d-1}r_{i-1}'\right)$. Indeed, there are $d^2$ ways to choose the indices $j_1,j_2\in\{1,\ldots,d\}$ such that $\left|m_{j_1}\right|$ and $\left|n_{j_2}-m_{j_2}\right|$ are maximum. In the case that $j_1=j_2$, it holds that $|m_{j_1}-(n_{j_1}-m_{j_1})|\leq 5r'_{i-1}$ so there are $O(r_{i-1}')$ choices for $m_{j_1}$ and $O(r_i)$ choices for $m_j$ for each $j\in\{1,\ldots,d\}\setminus \{j_1\}$. On the other hand, if $j_1\neq j_2$, then the number of choices of $m_{j_1}$ is $O(r_i)$ and, given this choice, the number of choices for $n_{j_2}-m_{j_2}$ (and thus for $m_{j_2}$) is $O(r_{i-1}')$ and, again, there are $O(r_i)$ choices for $m_j$ for each $j\in\{1,\ldots,d\}\setminus\{j_1,j_2\}$. 

Suppose now that we split $\mathbbm{T}^k$ into Voronoi cells with the points in $X_i$ as the centres, similarly to the proof of Lemma~\ref{outline:lem:matching}. Each Voronoi cell contains $\Omega(r_i^d)$ elements and, by the arguments above, $O(r_i^{d-1}r_{i-1}')$ of these points are in $W_i$. It follows that the measure of $W_i$ is $O(r_{i-1}'/r_i)$. 
\end{proof}

\subsection{Covering the Whole Torus}
\label{sec:ToastCovering}

Finally, we apply a compactness argument of Boykin and Jackson~\cite{BoykinJackson07} to get a toast sequence that covers the whole torus~$\I T^k$. For each $\vec{p}\in \{0,\ldots,5\}^d$ and $i\ge1$, define 
\[Y_i^{\vec{p}}:=Y_i-\lfloor r'_i/3\rfloor\sum_{j=1}^d(p_j-3)\vvec{x}_j.\]
Note that $Y_i^{\vec{p}}$ is simply a shifted version of $Y_i$ and so it inherits the property of being a maximally $r'_i$-discrete union of finitely many strips. Also, $Y_i^{(3,3,\ldots,3)}$ is simply $Y_i$. We define, for each $i\ge1$,
\begin{eqnarray}
 \label{eq:Kip}
K_i^{\vec{p}}&:=&\mathcal{C}_2(K_1^{\vec{p}},L_1^{\vec{p}},\ldots,K_{i-1}^{\vec{p}},L_{i-1}^{\vec{p}},N_2[J_i]),\text{ and}\\
\label{eq:Lip}
L_i^{\vec{p}}&:=&\mathcal{C}_2(K_1^{\vec{p}},L_1^{\vec{p}},\ldots,K_{i-1}^{\vec{p}},L_{i-1}^{\vec{p}},K_i^{\vec{p}},N_{2r'_i/5}[Y_i^{\vec{p}}]),
\end{eqnarray}
 as well as $\Tilde{K}_i^{\vec{p}}\subseteq {K}_i^{\vec{p}}$ and $\Tilde{L}_i^{\vec{p}}\subseteq {L}_i^{\vec{p}}$ by removing whole components that are at distance at most $2$ from respectively $J_j$ and $K_j^{\vec{p}}$ for some $j>i$. Note that these definitions are exactly the shifted versions of  the previous definitions, namely of~\eqref{eq:Ki}, \eqref{eq:Li} and so on, 
  except the same (unshifted) sets $J_i$ are used. (Thus, for example, it is not true in general that e.g.,\ $K_i^{\vec{p}}$ is just a shifted copy of~$K_i$.)
Therefore, all of the lemmas proved in Section~\ref{sec:ToastProperties} apply equally well to the sets 
$$
J_1,K_1^{\vec{p}},\Tilde{K}_1^{\vec{p}},L_1^{\vec{p}},\Tilde{L}_1^{\vec{p}},J_2,K_2^{\vec{p}},\Tilde{K}_2^{\vec{p}},L_2^{\vec{p}},\Tilde{L}_2^{\vec{p}},\ldots\,,
$$ 
as they apply to the sets $J_1,K_1,\Tilde{K}_1,L_1,\Tilde{L}_1,J_2,K_2,\Tilde{K}_2,L_2,\Tilde{L}_2,\ldots$\,.

The above definitions are motivated by the following result which implies that every vertex of $\I T^k$ belongs to $L_i^{\vec{p}}$ for at least one (in fact infinitely many) pairs $(i,\vec{p})$.


\begin{lemma}
\label{lem:Licover} Every element of $\I T^k$ is covered by infinitely many of the sets $N_{2r'_i/5}[Y_i^{\vec{p}}]$
for $\vec{p}\in\{0,\ldots,5\}^d$ and integer $i\ge 1$.
\end{lemma}

\begin{proof}
For each vertex $\vvec{v}\in\mathbbm{T}^k$, let 
\[R_i(\vvec{v}):=\{\vec{n}/r'_i: \vec{n}\cdot_a\vvec{v}\in Y_i\}\subseteq \mathbbm{R}^d.\]
Note that, since $Y_i$ is maximally $r'_i$-discrete, we have that $R_i(\vvec{v})$ contains a point in $[-1,1]^d$ for all $i\ge1$. Thus, by compactness of $[-1,1]^d$, the set $R^*(\vvec{v})$ of accumulation points of the sequence $R_1(\vvec{v}),R_2(\vvec{v}),\ldots$ in $[-1,1]^d$ is non-empty. 

If $\vvec{v}$ and $\vvec{v}'$ are in the same component of $G_d$, then $R_i(\vvec{v})$ is the same as $R_i(\vvec{v}')$ shifted by $\ell_\infty$-distance $\dist_{G_d}(\vvec{v},\vvec{v}')/r'_i$ in $\mathbbm{R}^d$. Since the sequence $r'_1,r'_2,\ldots$ is increasing, this distance tends to zero. Thus the sequences $(R_i(\vvec{v}))_{i=1}^\infty$ and $(R_i(\vvec{v}'))_{i=1}^\infty$ have the same set of accumulation points in $\I R^d$; in particular, we have $R^*(\vvec{v})=R^*(\vvec{v}')$. Thus the function $R^*$ is invariant under the action $a:\I Z^d\actson \I T^k$. 

Define
$$
T_{\vec{p}}':=\left\{\vvec{v}\in\mathbbm{T}^k:R^*(\vvec{v})\cap \prod_{s=1}^d\left[\frac{p_s-3}{3},\frac{p_s-2}{3}\right]\neq\emptyset\right\},\quad
\mbox{for $\vec{p}\in \{0,\ldots,5\}^d$}.
$$
 These sets cover the whole torus $\I T^k$, since $R^*(\vvec{v})$ is non-empty for every $\vvec{v}$. For each $\vec{p}\in\{0,\ldots,5\}^d$, we let 
 \beq{eq:Tp}
T_{\vec{p}}:=T_{\vec{p}}'\setminus \bigcup_{\substack{\vec{q}\in\{0,\ldots,5\}^d \\ \vec{q}\lex\vec{p}}} T_{\vec{q}'}
 \eeq
 consist of those $\vvec{v}$ for which $\vec{p}$ is the lexicographically smallest vector with $R^*(\vvec{v})$ having non-empty intersection with $\prod_{s=1}^d\left[\frac{p_s-3}{3},\frac{p_s-2}{3}\right]$.

Since the (set-valued) function $R^*$ is constant on each component of $G_d$, we have that each set $T_{\vec{p}}'$ (and thus each set $T_{\vec{p}}$) is a union of components of $G_d$, that is, is invariant under the action~$a$. Also, the sets $T_{\vec{p}}$, for $\vec{p}\in\{0,\ldots,5\}^d$, partition $\mathbbm{T}^k$. 

Now, for any $\vvec{v}\in T_{\vec{p}}$, we have  by the definition of $R^*(\vvec{v})$ that there are infinitely many indices $i$ for which there is a vertex $\vvec{u}\in Y_i$ of the form $\vvec{u}=\vec{n}\cdot_a\vvec{v}$ for some $\vec{n}$ such that 
\[r_i'\cdot \left(\frac{p_j-3}{3} - \frac{1}{100}\right)\leq n_j\leq r_i'\cdot \left(\frac{p_j-2}{3} + \frac{1}{100}\right),\quad\mbox{for all $1\leq j\leq d$.}\]
 Thus, for all such $i$ with $r_i'$ sufficiently large, it holds that
$$
\dist_{G_d}(\vvec{v},Y_i^{\vec{p}})\le \dist_{G_d}\left(\vvec{v},\vvec{u}-\lfloor r'_i/3\rfloor\sum_{j=1}^d(p_j-3)\vvec{x}_j\right)\le \frac{r_i'}3+\frac{3r_i'}{100}<\frac{2r_i'}5.
$$
We see that every element of $T_{\vec{p}}$ is covered  by $N_{2r'_i/5}[Y_i^{\vec{p}}]$ for infinitely many integers~$i$, proving the lemma.
\end{proof}

Finally, we are ready to construct the toast sequence that will be enough for proving all statements of Theorem~\ref{th:main}. Namely, for $i\ge 1$, define
\begin{eqnarray}
\label{eq:HatKi}
\Hat{K}_i&:=&\bigcup_{\vec{p}\in \{0,\ldots,5\}^d}\left(\Tilde{K}_i^{\vec{p}}\cap T_{\vec{p}}\right), \mbox{ and}\\
\label{eq:HatLi}
\Hat{L}_i&:=&\bigcup_{\vec{p}\in \{0,\ldots,5\}^d}\left(\Tilde{L}_i^{\vec{p}}\cap T_{\vec{p}}\right).
\end{eqnarray}
Since each set $T_{\vec{p}}$ consists of whole components of $G_d$ and the sets $T_{\vec{p}}$ form a partition of the torus $\I T^k$, Lemma~\ref{lm:Tilde} implies the following.

\begin{lemma}\label{lm:FinalToast}
$(J_1,\Hat{K}_1,\Hat{L}_1,J_1,\Hat{K}_2,\Hat{L}_2,\ldots)$ is a toast sequence that completely covers~$\I T^k$.
\qed\end{lemma}

\section{Using Toast Sequences to Round Flows}
\label{sec:integerFlow}

Here we show how one can use a toast sequence to obtain an integer-valued $(\ind_A-\ind_B)$-flow from a sequence of real-valued flows. The key new challenge is that the used real-valued flows $f_i$ need not meet the demand $\ind_A-\ind_B$ exactly (only when we pass to the limit as $i\to\infty$) while we are not allowed to access the whole sequence when rounding.

\subsection{Dealing with Finite Connected Sets}

Here, we present one of the main subroutines, which  perturbs 
a flow along the boundary of a finite set of vertices, attaining the desired properties on all edges except possibly one.
Many of the ideas appearing here are borrowed from~\cite{MarksUnger17}*{Section~5}. 

\begin{defn}
A \emph{triangle} in a graph $G$ is a set $\{u,v,w\}$ of three distinct vertices of $G$, every pair of which are adjacent in $G$.
\end{defn}

\begin{defn}
Given a graph $G$ and a set $F\subseteq E(G)$, let $\triangle_F$ be the graph with vertex set $F$ where two elements $uv$ and $yz$ of $F$ are adjacent in $\triangle_F$ if they are contained in a common triangle of $G$.
\end{defn}

\begin{defn}
An \emph{Eulerian circuit} in a finite graph $G$ is a sequence $(v_1,v_2,\ldots, v_{t})$ of vertices of $G$ such that $v_1=v_{t}$ and the sequence $(v_1v_2,v_2v_3,\ldots,v_{t-1}v_t)$ is an enumeration of the edge set of~$G$. 
\end{defn}

In the language of graph theory, an Eulerian circuit is a closed walk in a graph which traverses every edge exactly once. Perhaps the most classical result in graph theory is Euler's Theorem from 1736 which says that a finite graph $G$ has an Eulerian circuit if and only if it is connected and all of its vertices have even degree; see e.g.~\cite{Diestel17gt}*{Theorem~1.8.1}. The following lemma highlights a small technical advantage of defining the graph $G_d$ in terms of the set $\{\vec{\gamma}\in \mathbbm{Z}^d: \|\vec{\gamma}\|_\infty=1\}$ of generators of $\mathbbm{Z}^d$ as opposed to the standard basis (although one can also work in the latter graph, see e.g.~\cite{CieslaSabok22}*{Section 6}).
Recall that the edge-boundary $\partial_E S$ of $S\subseteq\I T^k$ consists of the
edges of $G_d$ with exactly one vertex in~$S$.

\begin{lemma}[Marks and Unger~{\cite{MarksUnger17}*{Proof of Lemma~5.6}}]
\label{lem:triEvenDeg}
If $S\subseteq \mathbbm{T}^k$, then every finite component of $\triangle_{\partial_ES}$ has an Eulerian circuit. 
\end{lemma}

\begin{proof}
By Euler's Theorem, it suffices to show that every vertex of $\triangle_{\partial_ES}$ has even degree. Let $\vvec{u}\vvec{v}\in \partial_E S$ with $\vvec{u}\in S$ and $\vvec{v}\notin S$. Note that, for any triangle $\{\vvec{u},\vvec{v},\vvec{w}\}$ in $G_d$ containing $\vvec{u}$ and $\vvec{v}$, exactly one of the edges $\vvec{u}\vvec{w}$ or $\vvec{v}\vvec{w}$ is in $\partial_ES$. Also, for any pair of edges of $G_d$ which are adjacent in $\triangle_{\partial_ES}$, there is a unique triangle which contains them.  So, it suffices to prove that $\vvec{u}\vvec{v}$ is contained in an even number of triangles of $G_d$.

Let $\vec{n}$ be the unique element of $\mathbbm{Z}^d$ such that $\|\vec{n}\|_\infty =1$ and
$\vvec{v}=\vvec{u}+\sum_{i=1}^dn_i\vvec{x}_i$.
Let $T_0:=\{i: 1\leq i\leq d\text{ and }n_i=0\}$ and $T_1:=\{1,2,\ldots,d\}\setminus T_0$. Note that the number of triangles containing $\vvec{u}\vvec{v}$ is exactly 
\[3^{|T_0|}\cdot 2^{|T_1|} -2.\]
 Indeed, the number of choices of $\vec{n}'=(n_1',\ldots,n_d')\in\mathbbm{Z}^d$ such that
$\vvec{w}:=\vvec{u}+\sum_{i=1}^dn_i'\vvec{x}_i$
 forms a triangle with $\vvec{u}$ and $\vvec{v}$ in $G_d$ or satisfies $\vvec{w}\in\{\vvec{u},\vvec{v}\}$ is exactly $3^{|T_0|}$ (each $n_i'$ for $i\in T_0$ can assume any value in $\{-1,0,1\}$) times $2^{|T_1|}$ (there are exactly two possible values for $n_i'$ for each $i\in T_1$). 
 
 Since $|T_1|\ge1$, the number of triangles is even, finishing the proof.
\end{proof}

Recall that by a \emph{hole} of $S\subseteq \mathbbm{T}^k$ we mean a finite component of $G_d\induced\left(\mathbbm{T}^k\setminus S\right)$. 
Marks and Unger~\cite{MarksUnger17}*{Proof of Lemma~5.6} use a result of Tim\'ar~\cite{Timar13} to show that, if $S$ is a finite subset of $\mathbbm{T}^k$ with no holes such that $G_d\induced S$ is connected, then $\triangle_{\partial_E S}$ is connected. Combining this with Lemma~\ref{lem:triEvenDeg} and Euler's Theorem, we get the following.

\begin{lemma}[Marks and Unger~{\cite{MarksUnger17}*{Lemma~5.6}}]
\label{lem:triEuler}
If $S$ is a finite subset of $\mathbbm{T}^k$ with no holes and $G_d\induced S$ is connected, then $\triangle_{\partial_E S}$ has an Eulerian circuit.\qed
\end{lemma}

The following definition is helpful for explaining the way in which we perturb flow values on triangles in $G_d$. 

\begin{defn}
Given a graph $G$ and an ordered triple $(u,v,w)\in V(G)^3$ such that $\{u,v,w\}$ is a triangle in $G$, define $\circlearrowleft_{u,v,w}$ to be the flow in $G$ such that, for $x,y\in V(G)$,
\[\circlearrowleft_{u,v,w}(x,y):=\left\{\begin{array}{rl} 1, & \text{if }(x,y)\in \{(u,v),(v,w),(w,u)\},\\
-1, & \text{if }(x,y)\in \{(v,u),(w,v),(u,w)\},\\
0, & \text{otherwise}.
\end{array}\right.\]
\end{defn}

Given a flow $\phi$ in $G_d$ and a finite subset $S$ of $\mathbbm{T}^k$ with no holes such that $G_d\induced S$ is connected, the operation of \emph{rounding $\phi$ along the boundary of $S$} is defined as follows. 
Using Lemma~\ref{lem:triEuler}, we let $\vvec{u}_1\vvec{v}_1,\ldots, \vvec{u}_t\vvec{v}_t$ be an Eulerian circuit in $\triangle_{\partial_ES}$ where, for each $1\leq s\leq t$, we have $\vvec{u}_s\in S$ and $\vvec{v}_s\notin S$. Moreover, among all Eulerian circuits, we choose the \emph{$\lex$-minimal} one (that is, one such that $(\vvec{u}_1,\vvec{v}_1,\ldots,\vvec{u}_t,\vvec{v}_t)$ is minimal under $\lex$). Note that $\vvec{u}_t=\vvec{u}_1$ and $\vvec{v}_t=\vvec{v}_1$. First, we do the \emph{adjustment} step, where we redefine $\phi(\vvec{u}_{t-1},\vvec{v}_{t-1})$ by adding $[\fout{\phi}(S)]-\fout{\phi}(S)$ to it and change $\phi(\vvec{v}_{t-1},\vvec{u}_{t-1})$ accordingly. The flow out of $S$ is now an integer. Then, for each $s=1,\ldots,t-2$, one by one, we let $\vvec{w}_s$ be the (unique) vertex of $\{\vvec{u}_{s+1},\vvec{v}_{s+1}\}$ such that $\{\vvec{u}_s,\vvec{v}_s,\vvec{w}_s\}$ is a triangle in $G_d$ and redefine $\phi$ to be
\begin{equation}\label{eq:PhiIncr}
 \phi:= \phi + \left(\left[ \phi(\vvec{u}_s,\vvec{v}_s)\right] - \phi(\vvec{u}_s,\vvec{v}_s)\right)\circlearrowleft_{\vvec{u}_s,\vvec{v}_s,\vvec{w}_s}.
 \end{equation}
 (See Figure~\ref{fig:rounding} for an illustration.)
Each of the steps in~\eqref{eq:PhiIncr} preserves the flow out of every vertex.   After all $t-2$ steps have been completed,  every edge $\vvec{u}_{i}\vvec{v}_{i}$ of $\partial_ES$,  except possibly $\vvec{u}_{t-1}\vvec{v}_{t-1}$, is assigned to an integer flow value by $\phi$ (because the last triangle update affecting $\phi(\vvec{u}_{i}\vvec{v}_{i})$ makes it integer). However, the total flow out of $S$  is an integer, and so $\phi(\vvec{u}_{t-1},\vvec{v}_{t-1})$ must be an integer as well. 

For future reference, we also define the flow $\FO{\phi}{S}$ (where $\phi$ in this notation will refer to the initial flow $\phi$ before any modifications took place), which is the sum of all $t-2$ increments in~\eqref{eq:PhiIncr}. Thus, $\FO{\phi}{S}$ is a 0-flow (i.e.,\ the flow out of every vertex is 0) and, with $\phi$ referring to its initial value, the flow $\phi+\FO{\phi}{S}$ assumes integer values on all edges of $\partial_ES$ except possibly $\vvec{u}_{t-1}\vvec{v}_{t-1}$.

More generally, given a flow $\phi$ in $G_d$ and a set $D\subseteq \mathbbm{T}^k$ such that every set in $\comp(D)$ is finite, the operation of \emph{rounding $\phi$ along the boundary of $D$} is defined as follows. First, for every $S\in \comp(D)$ and every hole $S'$ of $S$, we round $\phi$ along the boundary of~$S'$. Then, for each $S\in\comp(D)$, we round $\phi$ along the boundary of the union of $S$ and all of its holes. All of these operations are well-defined and the order in which we perform them does not affect the result. Note that, for every finite connected $S$, the edge sets $\partial_E
S'$, for $S'$ being a hole in $S$ or being $S$ with all its holes
filled, partition the edge boundary of $S$. Indeed, every pair in
$\partial_E S$ has exactly one point outside of $S$, lying in a
(unique) component $S'$ of $G_d\induced (\I T^k\setminus S)$. If $S'$ is not
a hole of $S$, then $S'$ is infinite while $S''$, the complement of $S'$ in the component of $G_d$ containing  $S$, is finite. Trivially, $S''$ cannot have any holes, so $S''$ is exactly $S$ with all its holes filled, as desired. 

We define 
$$\FO{\phi}{D}:=\sum_{S\in\comp(D)}\sum_{S'} \FO{\phi}{S'},$$
 where $S'$ in the inner sum is a hole of $S$, or $S$ with all its holes filled.

\begin{figure}[htbp]
\begin{center}
\includegraphics[scale=0.7]{rounding.1}
\hspace{1cm}
\includegraphics[scale=0.7]{rounding.2}
\hspace{1cm}
\includegraphics[scale=0.7]{rounding.3}
\end{center}
\caption{Two steps of the operation of rounding a flow $\phi$ along the boundary  of a set $S$. The vertices within the grey region are in $S$. The edge $\vvec{u}_s\vvec{v}_s$ currently being rounded is depicted by a bold black line and the other two edges of the triangle containing $\vvec{u}_s\vvec{v}_s$ and $\vvec{u}_{s+1}\vvec{v}_{s+1}$ are depicted by a bold grey line. Some edges are labelled with numbers which represent their current flow value in the direction indicated by the arrow.}
\label{fig:rounding}
\end{figure}

Next, we show that if, additionally, any two distinct components of $G_d\induced D$ are at distance at least three\footnote{This assumption, while convenient, is actually not necessary.} in $G_d$, then rounding along the boundary of such a set $D$ cannot displace the flow values by an arbitrary amount.

\begin{lemma}
\label{lem:boundary3d}
Let $\phi$ be a flow in $G_d$ and $D$ be a subset of $\mathbbm{T}^k$ such that each component of $G_d\induced D$ is finite and any two distinct components of $G_d\induced D$ are at distance at least three in $G_d$. Then, the sum of the absolute values of flow changes on each $\vvec{u}\vvec{v}\in E(G_d)$ when we round $\phi$ on the boundary of $D$ is at most $(3^d-2)/2$. 
(In particular, this also upper bounds the total change on each edge, that is,
$
\|\FO{\phi}{D}\|_\infty \le (3^d-2)/2$.)

\end{lemma}

\begin{proof} 
By construction, the flow from $\vvec{u}$ to $\vvec{v}$ changes only if there exists $S\in\comp(D)$ such that either
\begin{itemize}
\item $\vvec{u}\vvec{v}\in \partial_ES$,
\item $\vvec{u},\vvec{v}\in N[S]\setminus S$, or
\item $\vvec{u},\vvec{v}\in N[\mathbbm{T}^k\setminus S]\cap S$. 
\end{itemize}
Since any two components of $D$ are separated by a distance of at least three in $G_d$, we see that the choice of the component $S$ is unique.

In the first two cases, we let $T$ be the unique component of $G_d\induced \left(\mathbbm{T}^k\setminus S\right)$ which contains $\{\vvec{u},\vvec{v}\}\setminus S$. While rounding $\phi$ on the boundary of $T$ (in the case that $T$ is finite), or on the boundary of the union of $S$ and all of its holes (in the case that $T$ is infinite), the value of $\phi(\vvec{u},\vvec{v})$ is changed at most once for every triangle containing $\vvec{u}$ and $\vvec{v}$, where we view the last triangle of the Eulerian tour (the one that we have not used for any flow updates) as ``responsible'' for the initial adjustment of the flow on the second to last edge. The number of such triangles is at most the number of neighbours of $\vvec{u}$ different from $\vvec{v}$, which is at most $3^d-2$. Also, each time that the flow value of an edge changes, it is displaced by at most $1/2$ (since we round to a nearest integer). The flow from $\vvec{u}$ to $\vvec{v}$ is not changed at any other stage, giving the required.

In the third case, for each $\vvec{w}\notin S$ such that $\{\vvec{u},\vvec{v},\vvec{w}\}$ is a triangle in $G_d$, there is a unique component $T$ of $G_d\induced \left(\mathbbm{T}^k\setminus S\right)$ containing $\vvec{w}$. This triangle is used at most once when rounding $\phi$ along the boundary of $T$ (in the case that $T$ is finite) or the union of $S$ and all of its holes (in the case that $T$ is infinite). Thus, each triangle containing $\vvec{u}$ and $\vvec{v}$ contributes at most $1/2$ to the amount by which $\phi(\vvec{u},\vvec{v})$ changes; there are at most $3^d-2$ such triangles, and so the proof is complete.
\end{proof}

Next, given a flow $\phi$ in $G_d$ and a subset $S$ of $\mathbbm{T}^k$ such that $G_d\induced S$ is connected, the operation of \emph{completing $\phi$ within $S$} is defined as follows. Consider all integer-valued flows $\phi'$ in $G_d\induced S$ such that 
\begin{equation}\label{eq:Complete}   
\sum_{\vvec{v}\in S}\phi'(\vvec{u},\vvec{v}) + \sum_{\vvec{v}\in\mathbbm{T}^k\setminus S}\phi(\vvec{u},\vvec{v}) = \ind_A(\vvec{u})-\ind_B(\vvec{u}),\quad\mbox{for every $\vvec{u}\in S$,}
\end{equation}
 and, given this, $\|\phi'\|_\infty$ is as small as possible.
   (It will be the case that, whenever we apply this operation, at least one such $\phi'$ exists.)
 If $S$ is finite then we choose $\phi'$ so that the sequence $(\phi'(\vvec{u},\vvec{v}): (\vvec{u},\vvec{v})\in S^2)$ is lexicographically minimised, where the pairs in $S$ are viewed as being ordered according to $\lex$. If $S$ is infinite then we choose $\phi'$ arbitrarily, using Theorem~\ref{th:IFT} (i.e., using the Axiom of Choice). 
Change $\phi(\vvec{u},\vvec{v})$ to be equal to $\phi'(\vvec{u},\vvec{v})$ for all $\vvec{u},\vvec{v}\in S$. 

\subsection{Rounding a Sequence of Flows on a Toast Sequence}
\label{sec:FinalRounding}

The following key definition will allow us to round flows.

\begin{defn}\label{def:RoundingGD}
The \emph{rounding} of a sequence $(g_1,g_2,\ldots)$ of real-valued flows  in $G_d$ on a toast sequence $(D_1,D_2,\ldots)$  is the flow $f$ obtained via the following steps.
 \begin{enumerate}[(i)]
 \setcounter{enumi}{0}
 \item Initially, let $f$ be the identically zero flow except for each $i\ge 1$, $S\in\comp(D_i)$, and  $(\vvec{u},\vvec{v})\in\partial_ES$ we define $f(\vvec{u},\vvec{v}):=g_i(\vvec{u},\vvec{v})$ and $f(\vvec{v},\vvec{u}):=g_i(\vvec{v},\vvec{u})$ (that is, we copy the values of $g_i$ on $\partial_E D_i$). 
 \item For each $i\ge 1$, round $f$ along the boundary of~$D_i$.
 \item For each component $S$ of the graph $G':=(\I T^k,E(G_d)\setminus \bigcup_{i=1}^\infty \partial_ED_i)$, which is obtained from $G_d$ by removing the edge boundaries of all sets $D_i$, complete $f$ within $S$, as specified after~\eqref{eq:Complete}. 
 (If at least one completion step fails, then the whole procedure fails and $f$ is undefined.)
 \end{enumerate}
 \end{defn}
 
Observe that, since $(D_1,D_2,\ldots)$ is a toast sequence, the first two steps in the above definition are well-defined; also, the second step does not depend on the order in which we round the boundaries.

\begin{lemma}\label{lm:RoundingGD} In the notation of Definition~\ref{def:RoundingGD}, suppose that there are a constant $C$  and a sequence  $d_1\le d_2\le\ldots $ of integers such that  \begin{equation}\label{eq:RoundingGDLim}
 \lim_{i\to\infty} g_i(\vvec{u})=\ind_A(\vvec{u})-\ind_B(\vvec{u}),\quad\mbox{for each $\vvec{u}\in\I T^k$}, 
 \end{equation}
$\|g_1\|_\infty \leq C$ and, for each $i\ge 1$, we have that
\begin{equation}\label{eq:RoudingGD}
(3^d-1)\cdot (d_i+1)^d\cdot \sum_{j=i}^\infty\|g_{j+1}-g_{j}\|_\infty < 1/2,
 \end{equation}
 and every component $S$ of $G_d\induced D_i$ has diameter at most $d_i$.

 Then the following statements hold:
 \begin{enumerate}
 \stepcounter{equation}
\item\label{it:RoundingGDCompl} all completion steps succeed,
 \stepcounter{equation}
\item\label{it:RoundingGDABFlow} the obtained function $f$ is an integer-valued flow from $A$ to $B$,
 \stepcounter{equation}
\item\label{it:RoundingGDNorm} $\|f\|_\infty \le C+(3^d+2)/2$,
 \stepcounter{equation}
\item\label{it:RoundingGDLocal} the restriction of $f$ to the edges intersecting $D_i$ is a $(d_i+1)$-local function of $g_1,\ldots,g_i$ and $D_1,\ldots,D_i$.
 \end{enumerate}
 \end{lemma}
\begin{proof}
By combining the fact that $\|g_1\|_\infty\leq C$ with the bound in \eqref{eq:RoudingGD} in the case $i=1$, we see that the sequence $g_1,g_2,\ldots$ converges uniformly to a bounded function. Let $f_\infty$ denote the pointwise limit of this sequence. By~\eqref{eq:RoundingGDLim}, the function $f_\infty$ is an $(\ind_A-\ind_B)$-flow.

Let us show~\eqref{it:RoundingGDCompl}, that is, that all completion steps succeed.
For this, we introduce one more operation. Given two flows $\phi$ and $\psi$ in $G_d$ and a finite subset $S$ of $\mathbbm{T}^k$ with no holes such that $G_d\induced S$ is connected, the operation of \emph{equalising $\phi$ to $\psi$ along the boundary of $S$} is defined as follows. Let $(\vvec{u}_1\vvec{v}_1,\ldots,\vvec{u}_t\vvec{v}_t)$ with $\vvec{u}_1,\ldots,\vvec{u}_t\in S$ be the lexicographically minimal Eulerian circuit in $\triangle_{\partial S}$ (which exists by Lemma~\ref{lem:triEuler}).
For each $s=1,\ldots,t-2$, one by one, we let $\vvec{w}_s\in \{\vvec{u}_{s+1},\vvec{v}_{s+1}\}$ be such that $\{\vvec{u}_s,\vvec{v}_s,\vvec{w}_s\}$ is a triangle in $G_d$ and redefine $\phi$ to be
 \begin{equation}\label{eq:FlowEq}
  \phi:=\phi + \left(\psi(\vvec{u}_s,\vvec{v}_s) - \phi(\vvec{u}_s,\vvec{v}_s)\right)\circlearrowleft_{\vvec{u}_s,\vvec{v}_s,\vvec{w}_s}.
  \end{equation}
  Note that this operation makes the new value of $\phi$ on $(\vvec{u}_s,\vvec{v}_s)$ to be equal to $\psi(\vvec{u}_s,\vvec{v}_s)$.
  Also, we define $\FE{\phi,\psi}{S}$ (where $\phi$ stands for its initial value) as the difference between the final and initial flows $\phi$ during this process, that is, $\FE{\phi,\psi}{S}$ is the sum of all $t-2$ increments in~\eqref{eq:FlowEq}. Thus $\FE{\phi,\psi}{S}$ is a $0$-flow such that $\phi+\FE{\phi,\psi}{S}$ is equal to $\psi$ on every pair in $\partial S$ except possibly the pair $\vvec{u}_{t-1}\vvec{v}_{t-1}$.
 More generally, for a set $D\subseteq\mathbbm{T}^k$ such that every component of $G_d\induced D$ is finite, \emph{equalising $\phi$ to $\psi$ along the boundary of $D$} is defined as follows. First, for each component $S$ of $G_d\induced D$ and each hole $S'$ of $S$, we equalise $\phi$ to $\psi$ along the boundary of $S'$. Then, for every component $S$ of $G_d\induced D$, we equalise $\phi$ to $\psi$ along the boundary of the union of $S$ and all of its holes. Also, define $$\FE{\phi,\psi}{D}:=\sum_{S\in\comp(D)}\sum_{S'} \FE{\phi,\psi}{S'},$$ 
 where $S'$ in the inner sum is a hole of $S$, or $S$ with all holes filled.

 The proof of the following claim follows similar lines as that of Lemma~\ref{lem:boundary3d}, so we give only a very brief proof sketch.
 
 \begin{claim}
 \label{cl:boundaryEqualise}
 Let $\phi$ and $\psi$ be flows in $G_d$, $M$ be a non-negative integer and $D$ be a subset of $\mathbb{T}^k$ such that $|\partial_ES|\leq M$ for every component $S$ of $G_d\induced D$ and any two components of $G_d\induced D$ are at distance at least three in $G_d$. Then $\|\FE{\phi,\psi}{D}\|_\infty\le M\|\phi-\psi\|_\infty$.
 \end{claim}
 
 \begin{proof}[Sketch of Proof] Note that, before the $s$-th step as in~\eqref{eq:FlowEq}, the current value of $\phi$ on $\vvec{u}_{s+1}\vvec{v}_{s+1}$ is still the original value. Thus, during the step, its absolute value increases by at most $|\psi(\vvec{u}_s,\vvec{v}_s) -\phi(\vvec{u}_s,\vvec{v}_s)|$, which in turn is  at most $s\|\phi-\psi\|_\infty$ by an easy induction on $s$. The claim now follows since we do at most $M-1$ steps.
\end{proof}

This claim
gives that, for every $S\in\comp(D_i)$, if $S'$ is a hole of $S$ or $S$ will all holes filled then
 \begin{equation}
  \label{eq:EqualiserNorm}
 \|\FE{f_\infty,g_i}{S}\|_\infty \le 1/2.
 \end{equation} 
 Indeed, $|S'|\le (d_i+1)^d$ since
 the projection of $S'$, when viewed as a subset of $\I Z^d$, on each coordinate is at most $d_i+1$ by the diameter assumption. Thus $|\partial_ES'|\le (3^d-1)(d_i+1)^d$. Also,
 \begin{equation}\label{eq:FInftyG}
 \|f_\infty-g_i\|_\infty\leq \sum_{j=i}^\infty \|g_{j+1}-g_j\|_\infty < \frac{1/2}{(3^d-1)(d_i+1)^d}
 \end{equation}
 by~\eqref{eq:RoudingGD}. Therefore, ~\eqref{eq:EqualiserNorm} follows from Claim~\ref{cl:boundaryEqualise}.

Now, consider the real-valued flow 
 \begin{equation}\label{eq:AlternativeH}
 h:=f_{\infty}+\sum_{i=1}^\infty  \FE{f_\infty,g_i}{D_i}+\sum_{i=1}^\infty  \FO{g_i}{D_i},
 \end{equation}
 which will be used to certify that all completion steps succeed.
 This function $h$ is a flow from $A$ to $B$ since it is obtained from the $(\ind_A-\ind_B)$-flow $f_\infty$ by adding a 0-flow. 
Take any $S$ which is, for some $T\in\comp(D_i)$, either $T$ with all holes filled or a hole of $T$. By Lemma~\ref{lem:triEuler}, the graph $\triangle_{\partial_ES}$ has an Eulerian circuit; let $(\vvec{u}_1\vvec{v}_1,\ldots, \vvec{u}_t\vvec{v}_t)$ with $\vvec{u}_1,\ldots,\vvec{u}_t\in S$ be the  $\lex$-minimal one. By the definition of $\FE{f_\infty,g_i}{D_i}$, the flows 
$$
 f':=f_\infty+\FE{f_\infty,g_i}{D_i}
 $$ 
 and $g_i$ coincide on all pairs in $\partial_ES$ except possibly $\vvec{u}_{t-1}\vvec{v}_{t-1}$. 
 Also, 
 we have by~\eqref{eq:FInftyG}  that
  $$
  \Big|\,\fout{g_i}(S)-|A\cap S|+|B\cap S|\,\Big|= \Big|\,\fout{g_i}(S)-\fout{f_\infty}(S)\,\Big| \le |\partial_E S| \cdot \|f_\infty - g_i\|_\infty<1/2.
  $$
  Recall that, initially, $f$ was set to be $g_i$ on $\partial_ES$. Thus the adjustment step of making $\fout{f}(S)$ integer by adjusting its value on $\vvec{u}_{t-1}\vvec{v}_{t-1}$ by at most $1/2$ makes $\fout{f}(S)$ to be equal to $|A\cap S|-|B\cap S|$, which in turn is equal to the flow out of $S$ by $f'$ since it is obtained from the $(\ind_A-\ind_B)$-flow $f_\infty$ by adding a 0-flow. Thus $f$ after the adjustment is equal to $f'$ on the edge $\vvec{u}_{t-1}\vvec{v}_{t-1}$ (as they are already equal on every other edge in $\partial_ES$). 
  Thus when we round $f$ along the boundary of $S$, we add the same $t-2$ increments as we would do for $f'$, that is,
 $\FO{f'}{S}=\FO{g_i}{S}$. Hence, the flows $\FO{f'}{D_i}$ and $\FO{g_i}{D_i}$ coincide on~$\partial_ES$.
 We do not modify the values of $f$ on $\partial_E S$ during any other steps,
 so they are the same as the values of $f'+\FO{f'}{D_i}$ (which are in turn the same as the values of $h$ on~$\partial_ES$).
  Thus the final flow $f$ coincides with $h$ on $\partial_E D_i$ (for every $i$)  and the Integral Flow Theorem (Theorem~\ref{th:IFT}) shows that each completion step works, that is,~\eqref{eq:Complete} 
 can always be satisfied. Thus Conclusion~\eqref{it:RoundingGDCompl} holds.
 
Conclusion~\eqref{it:RoundingGDABFlow} is a direct consequence of the previous conclusion since the completion step is applied to every component obtained by removing $\bigcup_{i=1}^\infty \partial_ED_i$ from $E(G_d)$. 

 For~\eqref{it:RoundingGDNorm}, recall that the Integral Flow Theorem (Theorem~\ref{th:IFT}) guarantees that the produced integer-valued flow differs from the given real-valued flow by at most 1 on every edge. Thus, by Lemma~\ref{lem:boundary3d} and by~\eqref{eq:EqualiserNorm}, we have
  $$
  \|f\|_\infty\le \|h\|_\infty+1 \le \|f_\infty \|_\infty + \frac{3^d-2}{2} + \frac12+1\le C+ \frac{1}{2}+ \frac{3^d+1}{2} = C+\frac{3^d+2}{2},
  $$  
  as desired.

For~\eqref{it:RoundingGDLocal}, take any $\vvec{u}\in D_i$. Its component $S\in\comp(D_i)$ has diameter at most $d_i$ and thus entirely lies inside the $d_i$-neighbourhood
of~$\vvec{u}$. Since $\vvec{u}$ has access to every vertex of $S$, it also has access to the flow values on all edges in $\partial_ES$ (because, by our conventions, the flow value on an edge is encoded in each of its endpoints). Since $(D_i)_{i=1}^\infty$ is a toast sequence, the choices of how the flow $f$ is modified on $S$ during the rounding along $\partial_E S$ or any completion inside $S$ are well-defined functions of the restrictions of $g_1,\ldots,g_i$ and $D_1,\ldots,D_{i-1}$ to $S$. 
Thus the restriction of the final flow $f$ to the edges intersecting $S$ is indeed a $(d_i+1)$-local function of $g_1,\ldots,g_i$ and $D_1,\ldots,D_i$. (Note that we add $1$ to $d_i$ so that each endpoint of an affected edge, which includes points in $N_1[S]\setminus S$, can compute the new flow value on the edge.)
 \end{proof}

  \section{Proof of Theorem~\ref{th:main}}
  \label{sec:pieces}

We are now ready to prove our main result, Theorem~\ref{th:main}. 

For each $i\ge1$, let $m_i$ and $m_i'$ be the minimal non-negative integers such that
 \begin{eqnarray}
\label{eq:miDefAgain}
(3^d-1)\cdot (t_i+1)^d \cdot \left(\frac{c\,2^{1+\varepsilon}}{2^{d+\varepsilon m_i}(2^\varepsilon -1)}\right) &<& 1/2, \mbox{ and}\\
\label{eq:mi'Def}
(3^d-1)\cdot(t_i'+1)^d\cdot\left(\frac{c\,2^{1+\varepsilon}}{2^{d+\varepsilon m_i'}(2^\varepsilon -1)}\right)&<&1/2.
\end{eqnarray}
The main purpose of these bounds is to certify inequalities of the type given by \eqref{eq:RoudingGD}.

Let $f=f_{J\Hat{K}\Hat{L}}$ be the rounding of the sequence $(f_{m_i},f_{m_i},f_{m_i'})_{i=1}^\infty$ of flows on the toast sequence $(J_i,\Hat{K}_i,\Hat{L}_i)_{i=1}^\infty$. (Recall that  the sets $\Hat{K}_i$ and $\Hat{L}_i$ were defined before
Lemma~\ref{lm:FinalToast}
and  $(J_i,\Hat{K}_i,\Hat{L}_i)_{i=1}^\infty$ is a
toast sequence by the lemma; also, $f_m$ for $m\ge 0$ is the flow returned by Lemma~\ref{outline:lem:realFlows}.) We will use this flow $f$ to obtain the desired equidecomposition between $A$ and~$B$.

\begin{remark} If the reader is interested only in Parts~\ref{it:a} and~\ref{it:c} of Theorem~\ref{th:main} then it suffices to take for $f$ the flow $f_J$, which we define to be the rounding of the sequence $(f_{m_1},f_{m_2},\ldots)$ on the toast sequence $(J_1,J_2,\ldots)$.
By construction (or, formally, by Lemma~\ref{lm:f=fJ}), $f$ and $f_J$ are the same except on the set $\I T^k\setminus \bigcup_{i=1}^\infty J_i$ which is small in many respects (e.g.,\ its closure is null).\end{remark}


For reader's convenience, let us summarise some key properties of the flow $f=f_{J\Hat{K}\Hat{L}}$ in Lemma~\ref{lem:intFlow} below. This lemma is stated in a more general form so that it also applies to some other
auxiliary flows (in particular to $f_J$ when we take $(J_i,\emptyset,\emptyset)_{i=1}^\infty$ for the toast sequence $(J_i',K_i',L_i')_{i=1}^\infty)$.
Recall that, for each $\vec{p}\in \{0,\ldots,5\}^d$, the set $T_{\vec{p}}$ (defined in~\eqref{eq:Tp})
is a union of whole components of $G_d$; moreover, these sets partition~$\mathbbm{T}^k$.

\begin{lemma}
\label{lem:intFlow}
Fix any $\vec{p}\in \{0,\ldots,5\}^d$. Suppose that for each $i\ge 1$, $J_i'$, $K_i'$ and $L_i'$ are obtained  from respectively $J_i$, $K_i^{\vec{p}}$ and  $L_i^{\vec{p}}$ by removing the vertex sets of some  components so that $(J_i',K_i',L_i')_{i=1}^\infty$ is a toast sequence. Let $f'$ be the rounding of the flows $(f_{m_i},f_{m_i},f_{m_i'})_{i=1}^\infty$ on
$(J_i',{K}_i',{L}_i')_{i=1}^\infty$.
Then the following statements hold:
\begin{enumerate}
\stepcounter{equation}
\item\label{eq:fFlow} the function $f'$ is an integer-valued flow in $G_d$ from $A$ to $B$,
\stepcounter{equation}
\item\label{eq:fBounded} $\|f'\|_\infty\leq \frac{c\,2^{1+\varepsilon}}{2^d(2^\varepsilon-1)} + \frac{3^d+2}2$,
\stepcounter{equation}
\item\label{eq:fLocalK}  there is a $(2^{m_i}+t_i)$-local function of $J_1',\ldots,J_i',K_1',\ldots,K_i'$, $A$ and $B$ that coincides with $f'$ on all edges of $G_d$ intersecting $T_{\vec{p}}\cap\left(\bigcup_{j=1}^iK_j'\right)$,
\stepcounter{equation}
\item\label{eq:fLocalL} there is a $(2^{m_i'}+t_i')$-local function of $J_1',\ldots,J_i',K_1',\ldots,K_i',L_1',\ldots,L_i'$, $A$ and $B$ that coincides with $f'$ on all edges of $G_d$ intersecting $T_{\vec{p}}\cap\left(\bigcup_{j=1}^iL_j'\right)$.
\end{enumerate}
 Moreover, if for every $C\in\comp(J_i')$ and $S\in\comp(K_j')\cup\comp(L_j')$ at distance at most $2$ we have that $j\ge i$ and $N_2[C]\subseteq S$, then
 \begin{enumerate}
 \stepcounter{equation}
\item\label{eq:fLocalOnlyJ}
 the flow $f'$ coincides with $g$ on every edge intersecting $\bigcup_{i=1}^\infty J_i'$, where $g$
is the flow obtained by rounding $(f_{m_i})_{i=1}^\infty$ on $(J_i')_{i=1}^\infty$, \stepcounter{equation}
\item\label{eq:fLocalJ}
there is a $(2^{m_i}+2r_i+2q_{i-1}')$-local function of $J_1',\ldots,J_i'$, $A$ and $B$ that coincides with $f'$ on all edges of $G_d$ intersecting $T_{\vec{p}}\cap \left(\bigcup_{j=1}^i J_j'\right)$.
\end{enumerate}
\end{lemma}
\begin{proof} We just apply Lemma~\ref{lm:RoundingGD} to the flows $(f_{m_i},f_{m_i},f_{m_i'})_{i=1}^\infty$ and the toast sequence $(J_i',K_i',L_i')_{i=1}^\infty$ that were used to define~$f'$. Lemmas~\ref{lem:Ji} and~\ref{lem:KiLi} show that we can take 
\beq{eq:di}
(2r_i+2q_{i-1}', t_i, t_i')_{i=1}^\infty
\eeq
 for the sequence $(d_i)_{i=1}^\infty$ that upper bounds the diameters of finite components in our toast sequence.

By \eqref{eq:fmChange},  we have for all $m\ge0$ that
 \begin{equation}
 \label{eq:fmBound}
 \|f_m\|_\infty \leq \|f_0\|_\infty+\sum_{i=1}^{m}\|f_{i}-f_{i-1}\|_\infty \leq \sum_{i=1}^{m} \frac{2c}{2^{d+\varepsilon (i-1)}} \leq \frac{c\,2^{1+\varepsilon}}{2^d(2^\varepsilon-1)}
 \end{equation}
 that is, the flows $f_m$ have their $\ell_\infty$-norm uniformly bounded. Furthermore, Assumption~\eqref{eq:RoundingGDLim} of Lemma~\ref{lm:RoundingGD} holds by~\eqref{eq:AtoB} while
 Assumption~\eqref{eq:RoudingGD} holds (with respect to the sequence in~\eqref{eq:di}) by our choice of $m_i$ and $m_i'$ in \eqref{eq:miDefAgain} and~\eqref{eq:mi'Def}. 
 
Thus Lemma~\ref{lm:RoundingGD} applies to the flow $f'$ with $C$ being the expression in the right-hand side of~\eqref{eq:fmBound}. 

The claims in~\eqref{eq:fFlow} and~\eqref{eq:fBounded} follow from respectively Conclusions~\eqref{it:RoundingGDABFlow} and \eqref{it:RoundingGDNorm} of Lemma~\ref{lm:RoundingGD}.
 
Recall that $f_{m_i}$ is a $(2^{m_i}-1)$-local function of $A$ and $B$ by Conclusion~\eqref{eq:local}  of Lemma~\ref{outline:lem:realFlows}. Thus we can replace a local function of $f_{m_i}$ by a local function of $A$ and $B$, with just increasing the locality radius by~$2^{m_i}-1$. Thus~\eqref{eq:fLocalK} holds by Conclusion~\eqref{it:RoundingGDLocal} of Lemma~\ref{lm:RoundingGD}, according to which it is enough to increase the locality radius by the diameter bound plus~1. The same argument applied to $f_{m_i'}$ shows that~\eqref{eq:fLocalL} holds. 

It remains to show that the last two claims of the lemma hold, under the additional assumption.
For~\eqref{eq:fLocalOnlyJ}, take any edge $\vvec{u}\vvec{v}$ intersecting some~$J_i'$. 
Suppose first that $\vvec{u}\vvec{v}\in \partial_E J_i'$. Let $S$ be the (unique) component of $G\induced J_i'$ such that $\vvec{u}\vvec{v}\in\partial_E S$. The value of $g$ (as well as the value of\ $f'$) on $\vvec{u}\vvec{v}$ is the one which is assigned during the rounding of $f_{m_i}$ along the boundary of~$S$. Thus $g$ and $f'$ assume the same value on  $\vvec{u}\vvec{v}$, as desired. 
So suppose that both $\vvec{u}$ and $\vvec{v}$ are in~$J_i'$. Let $S$ be the component of $G\induced J_i'$ containing $\vvec{u}$ and~$\vvec{v}$. Consider any $T\in \comp(K_j')\cup \comp(L_j')$ which is at distance at most 2 from~$S$. By the extra assumption, we have that $j\ge i$ and $N_2[S]\subseteq T$.
Thus $\partial_E S$ separates $\vvec{u}\vvec{v}$ from $\partial_E T$. This means that when we remove the boundary edge set
$\bigcup_{j=1}^\infty(\partial_E J_j'\cup \partial_E K_j'\cup \partial_E L_j')$ from $E(G_d)$, the component containing $\vvec{u}\vvec{v}$ will be the same when we remove only $\bigcup_{j=1}^\infty \partial_E J_j'$. Thus, indeed $f'$ assumes the same value  as~$g$ on $\vvec{u}\vvec{v}$. This proves~\eqref{eq:fLocalOnlyJ}.

Finally,~\eqref{eq:fLocalJ} follows from~\eqref{eq:fLocalOnlyJ} by Conclusion~\eqref{it:RoundingGDLocal} of Lemma~\ref{lm:RoundingGD} applied directly to the flow~$g$.
\end{proof}

Observe that, for every $\vec{p}\in \{0,\ldots,5\}^d$, Lemma~\ref{lem:intFlow} (when we take the toast sequence $(J_i,\Tilde{K}_i^{\vec{p}},\Tilde{L}_i^{\vec{p}})_{i=1}^\infty$ for $(J_i',K_i',L_i')_{i=1}^\infty$) applies to the restriction of the flow $f=f_{J\Hat{K}\Hat{L}}$ to the set $T_{\vec{p}}$ (which consists of whole components of $G_d$). Indeed, $\Hat{K}_i\cap T_{\vec{p}}=\Tilde{K}_i^{\vec{p}}\cap T_{\vec{p}}$ and $\Hat{L}_i\cap T_{\vec{p}}=\Tilde{L}_i^{\vec{p}}\cap T_{\vec{p}}$ 
while $\Tilde{K}_i^{\vec{p}}$ and $\Tilde{L}_i^{\vec{p}}$ are obtained from respectively ${K}_i^{\vec{p}}$ and ${L}_i^{\vec{p}}$ by removing whole components.
We conclude by Conclusions~\eqref{eq:fFlow} and~\eqref{eq:fBounded} of Lemma~\ref{lem:intFlow} that $f$ is a bounded integer flow from $A$ to~$B$. Also, the extra assumption of Lemma~\ref{lem:intFlow} holds by  by Lemma~\ref{lm:KLKJ} and~\eqref{it:Tilde4}; in particular, the following (obvious from the construction) result is formally proved by Conclusion~\eqref{eq:fLocalOnlyJ}.

\begin{lemma}\label{lm:f=fJ} The flows $f$ and $f_J$ coincide on every edge intersecting $\bigcup_{i=1}^\infty J_i$.\qed\end{lemma}

Lemma~\ref{outline:lem:matching} can be applied to the integer $(\ind_A-\ind_B)$-flow $f$, producing
an equidecomposition between $A$ and $B$. We claim that this equidecomposition  satisfies all claims of Theorem~\ref{th:main}. Since Part~\ref{it:c} of Theorem~\ref{th:main} is already taken care of by the second claim of Lemma~\ref{outline:lem:matching}, it remains to prove the remaining two parts.

Recall that each piece returned by Lemma~\ref{outline:lem:matching} is a local function of $A$, $B$, finitely many strips and the flow $f$. 
Since the corresponding target families in Parts~\ref{it:a} and~\ref{it:b} of Theorem~\ref{th:main} form $a$-invariant algebras (that contain all strips as well as our sets $A$ and $B$), it is enough by Observation~\ref{obs:LocalRule} to prove that each set of the form $Z_{\vec{\gamma},\ell}^f$ (as defined in~\eqref{eq:Zdef}) for $\vec{\gamma}\in\{-1,0,1\}^d$ and integer $\ell$ with $|\ell|\leq \|f\|_\infty$ belongs to the corresponding algebra.

\subsection{Dimension of Boundaries}
\label{sec:Dim}

Here we prove Part~\ref{it:a} of Theorem~\ref{th:main} that the upper Minkowski dimension of the boundaries of the obtained pieces is strictly smaller than~$k$. 
Specifically, we show that, for every $\vec{\gamma}\in\{-1,0,1\}^d$ and $|\ell|\leq \|f\|_\infty$, the upper Minkowski dimension of the boundary of $Z_{\vec{\gamma},\ell}^f$ is at most $k-\zeta$, where  
\begin{equation}
\label{eq:zeta}
\zeta:=\frac{\epsilon}{4d\max\{d/\varepsilon,\, d+1\}+1}>0.
\end{equation}
(Recall that $\epsilon$, $d$ and $\varepsilon$ were defined in \eqref{eq:epsilonGap}, \eqref{eq:dBound} and~\eqref{eq:epsBound}, respectively.) One can show that $d/\varepsilon\ge d+1$ if $k\ge 2$.
\hide{First, note that $d/\varepsilon> d+1$. Indeed, using the definition of $\varepsilon$ we have 
$$
d-(d+1)\varepsilon=2d+1-\frac{d^2\epsilon+d\epsilon}k.
$$
The derivative of this quadratic polynomial of $d$ has root at $k/\epsilon -1/2$. Since $0\le d-k/\epsilon\le1$, the value of the polynomial is at least as large as that for $d=k/\epsilon+1$, which in turn is $(k^2-2\epsilon^2)/(\epsilon k)>0$, as desired. 
}

Let us pause for a moment to discuss the quantitative bound on the boundary dimension that we get when $k\geq2$ and $\boxdim(\partial A)=\boxdim(\partial B)=k-1$. In this special case, we can choose $\epsilon$ to be arbitrarily close to $1$, take $d$ to be $k+1$ and $\varepsilon$ close to $1/k$.  Thus we get that any $\zeta<\frac{1}{4 k (k+1)^2+1}$ works here, and so we can obtain pieces whose boundaries have upper Minkowski dimension arbitrarily close to  $k - \frac{1}{4 k (k+1)^2+1}$.
In particular, if $k=2$, as in the case that $A$ is a disk and $B$ is a square, then this quantity evaluates to $145/73<1.987$. Thus, this is sufficient for proving Theorem~\ref{th:circleSquare}.

Here we need to use the specific choices $r_i':=100^{2^{i+1} - 1}$ and $r_i:=100^{2^{i+1} - 2}$, as defined in~\eqref{eq:ri'specific} and~\eqref{eq:rispecific} respectively. 

Let us provide some auxiliary results first. 
A \emph{box} 
in $\mathbbm{T}^k$ is a set of the form $\prod_{i=1}^k[a_i,b_i)$ where $0\leq a_i<b_i\leq 1$ for all $1\leq i\leq k$.

\begin{lemma}[Laczkovich~{\cite{Laczkovich92b}*{Lemma~2}}; see also Schmidt~{\cite{Schmidt64}*{p.~517}}]
\label{lem:intervalLog}
For almost every choice of $\vvec{x}_1,\ldots,\vvec{x}_d$ in $\mathbbm{T}^k$ and for every $t>0$ there exists $C>0$ such that
\begin{equation}\label{eq:intervalLog}
D\left(N_r^+[\vvec{u}], I\right)\leq C\log^{k+d+t}(r)
\end{equation}
for every element $\vvec{u}\in\mathbbm{T}^k$, integer $r\ge 2$ and box $I$ in $\mathbbm{T}^k$. 
\end{lemma}

This result is needed to prove the following lemma, whose conclusion was one of the assumptions (namely, Property~\eqref{eq:XQi}) made about $\vvec{x}_1,\ldots,\vvec{x}_d$ in Section~\ref{subsec:setting}.

\begin{lemma}
\label{lem:stripLem}
There exist positive constants $c_1$ and $C_2$ such that the following holds. 
If, for $i\ge 1$, we define $\gamma_i:=c_1r_i^{-d}\log^{-2}(r_i)$, $R_i:=C_2r_i\log^{(k+d+3)/d}(r_i)$ and
\begin{equation}\label{eq:Qi}
 Q_i:=\left[0, \gamma_i\right)\times [0,1)^{k-1},
 \end{equation}
then, with positive probability with respect to uniform independent $\vvec{x}_1,\ldots,\vvec{x}_d\in\I T^k$, it holds for every integer $i\ge 1$ that the strip $Q_i$ is $r_i$-discrete in $G_d$ while $N_{R_i}[Q_{i}]=\mathbbm{T}^k$.
\end{lemma}

\begin{proof} Let $C_1$ be a constant satisfying Lemma~\ref{lem:intervalLog} with $t=1$ for at least half (in measure) of the choices of $(\vvec{x}_1,\ldots,\vvec{x}_d)\in(\I T^k)^d$. One way to see its existence (without checking whether the proof of Lemma~\ref{lem:intervalLog} gives some effective bounds on $C_1$) is as follows. For every real $C$, the set $\C X_C$ of sequences $(\vvec{x}_1,\ldots,\vvec{x}_d)\in(\I T^k)^d$ for which Lemma~\ref{lem:intervalLog} holds for this $C$ (with $t=1$) is an analytic subset, since it is the projection of the Borel subset of $(\I T^k)^d\times\Omega$ of points satisfying~\eqref{eq:intervalLog}, where $\Omega\subseteq \I T^k\times \{2,3,\ldots\}\times \I T^{2k}$ is the subspace encoding all suitable triples $(\vvec{u},r,I)$ equipped with the standard (Polish) topology. As every analytic subset is universally measurable (see e.g.~\cite{Kechris:cdst}*{Theorem~21.10}) and the countable nested union $\bigcup_{i=1}^\infty \C X_i\subseteq \I T^{k}$ has Lebesgue measure 1, by the countable additivity there is an index $i$ such that the measure of $\C X_i$ is at least $1/2$ and we can take $C_1$ to be equal to this $i$.

We will use the following estimates that hold by choosing the constants $c_1>0$ and $C_2\gg C_1$ suitably: for every integer $i\ge 1$,
 \beq{eq:gammai}
  \sum_{i=1}^\infty(2r_i+1)^d\gamma_i<1/12
  \eeq
  and
  \beq{eq:Ri}
  (2R_i+1)^d\gamma_i>C_1\log^{k+d+1}(2R_i).
  \eeq
  Let us briefly check that suitable $c_1$ and $C_2$ exist. The $i$-th summand in the left-hand side of~\eqref{eq:gammai} is at most $(3r_i)^d\cdot c_1r_i^{-d}\log^{-2}(r_i)\le  O(c_1 2^{-2i})$. This is summable, so~\eqref{eq:gammai} can be satisfied by choosing $c_1$ sufficiently small. Also, the left-hand size of~\eqref{eq:Ri} is at least $R_i^d\gamma_i= C_2^d\cdot  c_1\cdot  \log^{k+d+1}(r_i)$. As the exponent at the logarithm matches that in the right-hand side of~\eqref{eq:Ri}, this constraint can be satisfied by taking $C_2$ sufficiently large.
	
We start by applying a simple union bound argument to estimate the probability that $Q_i$ fails to be $r_i$-discrete for some $i\ge1$. For $1\leq j\leq d$, let  $\vvec{x}_j=(x_{j,1},\ldots,x_{j,d})$ with each $x_{j,i}\in [0,1)$. Of course, only the first coordinates $x_{j,1}$ matter for the lemma.

Clearly, $Q_i$ fails to be $r_i$-discrete if and only if there exists a non-zero vector $\vec{n}\in\{-r_i,\ldots,0,\ldots,r_i\}^d$ such that the sum $\sum_{j=1}^dn_j x_{j,1}$ viewed in $\mathbbm{R}$ (i.e.,\ not modulo $1$) is at distance less than $\gamma_i$ from an integer. Fix a non-zero vector $\vec{n}=(n_1,\ldots,n_d)$. By symmetry, assume that $n_d\neq0$ and, in fact, $n_d>0$. Suppose that $\vvec{x}_1,\ldots,\vvec{x}_{d-1}\in\I T^k$ have already been sampled. 
Then, since $0\leq n_dx_{j,1}<n_d$, there are at most $n_d+1$ different possible integers that the sum $\sum_{j=1}^dn_j x_{j,1}$ can be within $\gamma_i$ of and, for each of them, the probability of this event is at most $2\gamma_i/n_d$ and thus the probability of at least one happening is at most $2\gamma_i(n_d+1)/n_d\le 4\gamma_i$.
By the Union Bound over $\vec{n}$, the probability that $Q_i$ is not $r_i$-discrete is at most $4(2r_i+1)^d\gamma_i$.

Thus the probability that the set $Q_i$ is not $r_i$-discrete for some $i\ge1$ is at most
$\sum_{i=1}^\infty 4(2r_i+1)^d\gamma_i$ which is at most $1/3$ by~\eqref{eq:gammai}. 

Now, for the other assertion, take an arbitrary element $\vvec{u}\in \mathbbm{T}^k$. By the definition of $C_1$, with probability at least $1/2$ it holds for every integer $i$ that the discrepancy of $N_{R_i}[\vvec{u}]$ with respect to any strip is at most $C_1\log^{k+d+1}(2R_i)$. In particular, for the strip $Q_i$, we have
\[|N_{R_i}[\vvec{u}]\cap Q_i| \ge (2R_i+1)^d \gamma_i - C_1\log^{k+d+1}(2R_i),\]
which is strictly positive by~\eqref{eq:Ri}.
From the point of view of $Q_i$, this states that its translates by integer vectors of $\ell_\infty$-norm at most $R_i$ under the action $a:\I Z^d\actson \I T^k$ cover the whole torus, as required.

Thus the probability that a uniform $(\vvec{x}_1,\ldots,\vvec{x}_d)\in(\I T^k)^d$ satisfies the lemma is at least $1/2-1/3>0$, as desired.
\end{proof}

Let us remark that the existence of $C_1$ in the proof of Lemma~\ref{lem:stripLem} can be argued without referring to the theory of analytic sets. Informally speaking, to compute the smallest possible $C_1(\vvec{x}_1,\ldots,\vvec{x}_d)$ for given $\vvec{x}_1,\ldots,\vvec{x}_d\in\I T^k$, it is enough to consider only those boxes that have at least one integer combination of $\vvec{x}_1,\ldots,\vvec{x}_d$ on each side, plus $2^{2d}$ choices whether to include that side or not (realisable by making the box infinitesimally larger or smaller in that direction). There are countably many choices here and thus the best possible $C_1:(\I T^k)^d\to [0,\infty]$ is in fact a Borel function, as the supremum of countably many Borel functions.

\begin{lemma}\label{lm:Xi} If $\vvec{x}_1,\ldots,\vvec{x}_d$ satisfy the conclusion of Lemma~\ref{lem:stripLem},
then, as $i\to\infty$, there is a maximally $r_i$-discrete set $X_i$ which is an $r_i^{d+1}\log^{O(1)}(r_i)$-local function of the strip~$Q_i$ that was defined in~\eqref{eq:Qi}.\end{lemma}

\begin{proof} Let $c_1$ and $C_2$ satisfy Lemma~\ref{lem:stripLem} and let $R_i$ be defined as in the lemma. Thus the set $\mathcal{Q}$ of translates of $Q_i$ by vectors of the form $\sum_{i=1}^dn_i\vvec{x}_i$ with $\vec{n}\in\mathbb{Z}^d$ and $\|\vec{n}\|_\infty\leq R_i$ covers $\I T^k$.  Lemma~\ref{outline:lem:simpleDiscrete}, when applied to this covering, yields a maximally $r_i$-discrete set $X_i$ which is an $r_i(2R_i+1)^d$-local function of the sets in $\mathcal{Q}$. Every such translate is an $R_i$-local function of $Q_i$ itself. Thus, since $R_i=r_i\log^{O(1)}(r_i)$, we get that $X_i$ is an $r_i^{d+1}\log^{O(1)}(r_i)$-local function of $Q_i$.
\end{proof}

Recall that our final flow $f=f_{J\Hat{K}\Hat{L}}$ was defined at the beginning of Section~\ref{sec:pieces}. Since our forthcoming argument will only use the values of $f$ on edges intersecting $\bigcup_{i=1}^\infty I_i\subseteq \bigcup_{i=1}^\infty J_i$, the reader may equivalently use $f_J$ instead of $f$ here.

\begin{lemma}
\label{lm:booleanABQ}
For any $\vec{\gamma}\in \{-1,0,1\}^d$, $\ell\in \range(f)$ and $i\geq1$, the set $Z_{\vec{\gamma},\ell}^f\cap I_i$ can be written as a Boolean combination of at most $r_i^{d\,\max\{d/\varepsilon,\,d+1\}+o(1)}$ translates of  $A$, $B$ and $Q_1,\ldots,Q_i$.
\end{lemma} 
\begin{proof}
By \eqref{eq:miDefAgain} it holds that $2^{m_i}=\Theta(r_i^{d/\varepsilon})$. 
By Conclusion~\eqref{eq:fLocalJ} of Lemma~\ref{lem:intFlow} (applied to $f$), each set $Z_{\vec{\gamma},\ell}^f\cap I_i$ is a $O(r_i^{d/\varepsilon})$-local function of $A$, $B$ and $J_1,\dots,J_i$.
By Lemmas~\ref{lem:Ii} and~\ref{lm:JiLocality}%
, the sets $I_i$ and $J_i$ are both $O(r_i)$-local functions of $X_1,\ldots,X_i$ which in turn, by Lemma~\ref{lm:Xi}, are $r_i^{d+1+o(1)}$-local functions of $Q_1,\ldots,Q_i$. 
By adding the locality radii, we see that $Z_{\vec{\gamma},\ell}^f\cap I_i$
is a local function of  $A$, $B$ and $Q_1,Q_2,\ldots,Q_i$ of radius at most $r_i^{\max\{d/\varepsilon,\,d+1\}+o(1)}$, implying the lemma.\end{proof}

We are now in position to prove the following result.

\begin{lemma}
\label{lem:bdyDim}
For every $\vec{\gamma}\in\{-1,0,1\}^d$ and $\ell\in \range(f)$, the boundary of $Z_{\vec{\gamma},\ell}^f$ has upper Minkowski dimension at most $k-\zeta$.
\end{lemma}

\begin{proof}
By the definition of $\epsilon$ in \eqref{eq:epsilonGap}, the upper Minkowski dimension of each of $\partial A$ and $\partial B$ is less than $k-\epsilon$. Thus we can choose $\delta_0>0$ to be sufficiently small so that, for all $0 < \delta < \delta_0$,
\begin{equation}\label{eq:ABboundary}\max\big\{\lambda\left(\left\{\vvec{x}: \dist_\infty(\vvec{x},\partial A)\le \delta\right\}\right),\,
\lambda\left(\left\{\vvec{x}: \dist_\infty(\vvec{x},\partial B)\le \delta\right\}\right)\big\}< \delta^{\epsilon},\end{equation}
 where $\dist_\infty$ denotes the $\ell_\infty$-distance on the torus. Moreover, we also assume that $\delta_0$ is small enough so that
\begin{equation}
\label{eq:delta0small}
r_1<\delta_0^{-4\zeta}.
\end{equation} 
 
Let $0<\delta<\delta_0$ be arbitrary and let $i$ be the unique index so that
\begin{equation}\label{eq:riChoice}
r_i\leq \delta^{-4\zeta} < r_{i+1}.
\end{equation}
Note that such an index $i$ is guaranteed to exist by \eqref{eq:delta0small}.
Let $\mathcal{P}$ be a collection of boxes in $\mathbbm{T}^k$ of side-length $\delta$ such that each point of $\mathbbm{T}^k$ is contained in at least one and at most $2^k$ sets in $\mathcal{P}$. (For example, let $\mathcal{P}$ be the product of such 1-dimensional coverings.) Define
\[\mathcal{Z}:=\left\{Z_{\vec{\gamma},\ell}^f: \vec{\gamma}\in\{-1,0,1\}^d\text{ and }\ell\in \range(f)\right\}.\]
Since $f$ is integer-valued, the set $\mathcal{Z}$ is finite, having at most $3^d(2\|f\|_\infty+1)<\infty$ elements. 

Since we can choose $\delta_0$ arbitrarily small, let us view $\delta$ as tending to 0 (and $i\to\infty$) using the asymptotic notation accordingly. Thus, for example, $|\mathcal{Z}|=O(1)$.

Our goal is to show that at most $\left(1/\delta\right)^{k-\zeta+o(1)}$  boxes in $\mathcal{P}$ can intersect $\bigcup_{Z\in\mathcal{Z}}\partial Z$. This will be enough for bounding the upper Minkowski dimension of this set, since each box in $\mathcal{P}$ can intersect at most $2^k$ boxes of a regular $\delta$-grid.

Let $\mathcal{B}$ be the collection of boxes in $\mathcal{P}$ which intersect the topological boundary of $Z\cap I_i$ for some $Z\in \mathcal{Z}$. Recall that $i$ was chosen to satisfy~\eqref{eq:riChoice} while the set $I_i$ was defined in Section~\ref{sec:ToastConstructions}. Let $\mathcal{I}$ be the collection of sets in $\mathcal{P}$ which are not in $\mathcal{B}$ and intersect $\mathbbm{T}^k\setminus I_i$. By construction, $\mathcal{B}\cup \mathcal{I}$ contains every set in $\mathcal{P}$ which intersects the topological boundary of a set in~$\mathcal{Z}$. So, it suffices to bound $|\mathcal{B}|$ and $|\mathcal{I}|$. 

By Lemma~\ref{lm:booleanABQ}, any set of the form $Z\cap I_i$ for $Z\in\mathcal{Z}$ can be written as a Boolean combination of at most 
$r_i^{d\,\max\{d/\varepsilon,d+1\}+o(1)}$ 
 translates of $A$, $B$ and $Q_1,\ldots,Q_i$.  The union of all boxes in $\mathcal{B}$ lies within $\dist_\infty$-distance at most $\delta$ from the boundary of at least one of these translates.  Thus, 
using the volume bound, the estimate in \eqref{eq:ABboundary} and that, trivially, the measure of $\{\vvec{x}\mid \dist_\infty(\vvec{x},\partial Q_i)\le \delta\}$ is at most $4\delta$, we get that
\begin{eqnarray*}
\delta^k|\mathcal{B}|&\leq& 2^k\lambda\left(\bigcup_{Z\in\mathcal{Z}}\left\{\vvec{x}: \dist_\infty(\vvec{x},\partial (Z\cap I_i))
\leq \delta\right\}\right)\\
&\leq& 2^k\,|\mathcal{Z}|\,\left( r_i^{d\,\max\{d/\varepsilon,\, d+1\}+o(1)}(\delta^\epsilon + i\cdot 4\delta)\right).
\end{eqnarray*}
  By the definitions of $\zeta$ and $i$  (that is, by \eqref{eq:zeta} and~\eqref{eq:riChoice}) and since $i=O(\log(\log(r_i)))$, we have that $i\delta=o(\delta^\epsilon)$ and
\[
\delta^k|\mathcal{B}|\le  (\delta^{-4\zeta})^{d\,\max\{d/\varepsilon,\, d+1\}+o(1)}\cdot \delta^ {\epsilon-o(1)}
\le \delta^{\zeta-o(1)}.\]
So, indeed, $|\mathcal{B}|\leq(1/\delta)^{k-\zeta+o(1)}$. 

Since none of the boxes in $\mathcal{I}$ are contained in $\mathcal{B}$, and $\bigcup_{Z\in\mathcal{Z}}Z = \mathbbm{T}^k$, all sets in $\mathcal{I}$ must all be contained in the interior of $\mathbbm{T}^k\setminus I_i$. Thus, by Lemma~\ref{lem:Jmeas} combined with \eqref{eq:ri'specific} and \eqref{eq:rispecific}, we have that
\[
\delta^k|\mathcal{I}|\le 2^k \lambda(\mathbbm{T}^k\setminus I_i) = O(r_{i-1}'/r_i)=O\left(r_{i+1}^{-1/4}\right) =O\left((1/\delta)^{-\zeta}\right).
\]
This completes the proof of Lemma~\ref{lem:bdyDim}.
\end{proof}

\begin{proof}[Proof of Theorem~\ref{th:main}\ref{it:a}] 
Recall  that the final equidecomposition is obtained  by applying the local function of Lemma~\ref{outline:lem:matching}
to $f$ (i.e.,\ to the sets $Z_{\vec{\gamma},\ell}^f$). Clearly, the family of subsets of $\I T^k$ whose boundary has upper Minkowski dimension at most $k-\zeta$ is a translation-invariant algebra. Since this algebra contains each set $Z_{\vec{\gamma},\ell}^f$ by Lemma~\ref{lem:bdyDim}, it also contains all pieces of the equidecomposition by Observation~\ref{obs:LocalRule}.\end{proof}

\subsection{Borel Complexity}
\label{sec:BorelComplexity}

This part of the proof heavily relies on the properties of various toast sequences that we constructed  earlier. 
First,  
let us briefly recall some key related definitions and properties.  The toast sequence $(J_i,\Hat{K}_i,\Hat{L}_i)_{i=1}^\infty$, as defined before Lemma~\ref{lm:FinalToast}, is obtained by taking the toast sequence $(J_i,\Tilde{K}_i^{\vec{p}},\Tilde{L}_i^{\vec{p}})_{i=1}^\infty$ inside each ($a$-invariant) set $T_{\vec{p}}$, and then taking the union over all~$\vec{p}$. In turn, $(J_i,\Tilde{K}_i^{\vec{p}},\Tilde{L}_i^{\vec{p}})_{i=1}^\infty$ (whose properties are summarised in Lemma~\ref{lm:Tilde})
is obtained from $(J_i,{K}_i^{\vec{p}},{L}_i^{\vec{p}})_{i=1}^\infty$ by removing those components from sets ${K}_i^{\vec{p}}$ and ${L}_i^{\vec{p}}$ that ``conflict with later sets''. Finally, the sets ${K}_i^{\vec{p}}$ and ${L}_i^{\vec{p}}$ are defined by~\eqref{eq:Kip} and~\eqref{eq:Lip} (which are the same as the definitions in~\eqref{eq:Ki} and~\eqref{eq:Li}, except we take the ``$\vec{p}$-shifted" sets $Y_i$).

Recall that we have to show that all pieces of the constructed equidecomposition belong to $\bB(\bSigma(\bB (\bSigma_1^0\cup \C T_A\cup \C T_B)))$, that is, each piece can be obtained from open sets and translations of the sets $A$ and $B$ by taking Boolean combinations, then countable unions and then Boolean combinations. (In fact, our proof gives a slightly stronger version where open sets can be replaced by strips.)

For $\vec{p}\in\{0,\ldots,5\}^d$, let the flow $f_{J\Tilde{K}^{\vec{p}}\Tilde{L}^{\vec{p}}}$ be obtained by rounding the sequence of flows $(f_{m_i},f_{m_i},f_{m_i'})_{i=1}^\infty$ on the toast sequence $(J_i,\Tilde{K}_i^{\vec{p}},\Tilde{L}_i^{\vec{p}})_{i=1}^\infty$.
Since 
 the toast sequences $(J_i,\Hat{K}_i,\Hat{L}_i)_{i=1}^\infty$ and $(J_i,\Tilde{K}_i^{\vec{p}},\Tilde{L}_i^{\vec{p}})_{i=1}^\infty$ coincide inside $T_{\vec{p}}$, we have for all $\vec{n}$ and $\ell$ that
\begin{equation}\label{eq:Zunionp}Z_{\vec{n},\ell}^f = \bigcup_{\vec{p}\in\{0,\ldots,5\}^d}\left(T_{\vec{p}}\cap Z^{f_{J\Tilde{K}^{\vec{p}}\Tilde{L}^{\vec{p}}}}_{\vec{n},\ell}\right).\end{equation}
So, it suffices to show that  $T_{\vec{p}}$ and $Z^{f_{J\Tilde{K}^{\vec{p}}\Tilde{L}^{\vec{p}}}}_{\vec{n},\ell}$ for $\vec{p}\in \{0,\ldots,5\}^d$ are all in the set family
$\bB(\bSigma(\bB (\bSigma_1^0\cup \C T_A\cup \C T_B)))$. We start with the former.

\begin{lemma}
\label{lem:TpHierarchy}
Each of the sets $T_{\vec{p}}$ constructed in the proof of Lemma~\ref{lem:Licover} can be written as a Boolean combination of $G_\delta$ sets.
\end{lemma}

\begin{proof}
Here we are going to use the assumption in \eqref{eq:disjointBoundaries} which states that each component of $G_d$ intersects the topological boundary of at most one set~$Y_i$.
Thus, if the sets $R_i(\vvec{v})$ for $i\ge1$ in the proof of Lemma~\ref{lem:Licover} are re-defined in terms of the interior of $Y_i$ instead of $Y_i$ itself, then, for each $\vvec{v}\in \mathbbm{T}^k$, the sets $R_1(\vvec{v}),R_2(\vvec{v}),\ldots$ remain unchanged, except for possibly one index $i\ge1$. In particular, the set $R^*(\vvec{v})$ of their accumulation points is unchanged. So, assuming that $R_i(\vvec{v})$ is defined in terms of the interior of $Y_i$, we have that for any $b,i\ge1$ and $\vec{p}\in\{0,\ldots,5\}^d$ the set
\[\left\{\vvec{v}\in\mathbbm{T}^k: R_i(\vvec{v})\cap \prod_{s=1}^d\left(\frac{p_s-3}{3}-\frac{1}{b},\frac{p_s-2}{3}+\frac{1}{b}\right)\neq\emptyset\right\}\]
is open as well. So, for each $\vec{p}$, 
\begin{eqnarray*}
 T_{\vec{p}}'&=&\left\{\vvec{v}\in\mathbbm{T}^k:R^*(\vvec{v})\cap \prod_{s=1}^d\left[\frac{p_s-3}{3},\frac{p_s-2}{3}\right]\neq\emptyset\right\}\\
&=&\bigcap_{b=1}^\infty\bigcap_{j=1}^\infty\bigcup_{i=j}^\infty\left\{\vvec{v}\in\mathbbm{T}^k: R_i(\vvec{v})\cap \prod_{s=1}^d\left(\frac{p_s-3}{3}-\frac{1}{b},\frac{p_s-2}{3}+\frac{1}{b}\right)\neq\emptyset\right\}
 \end{eqnarray*}
is a $G_\delta$ set. Recall that, by definition (namely by~\eqref{eq:Tp}), $T_{\vec{p}}$ consists of those  $\vvec{v}$ for which $\vec{p}$ is the lexicographically smallest vector with $\vvec{v}\in T_{\vec{p}}'$.  Thus  $T_{\vec{p}}$ is  Boolean combination of these sets.
\end{proof}

Next, let us consider the sets of the form $Z^{f_{J\Tilde{K}^{\vec{p}}\Tilde{L}^{\vec{p}}}}_{\vec{n},\ell}$.

\begin{lemma}
\label{lem:Zpcomplexity}
There exists a  collection $\mathcal{W}$ of subsets of $\mathbbm{T}^k$ such that
\begin{enumerate}
\stepcounter{equation}
\item\label{eq:Wgood} every set in $\mathcal{W}$ is a countable union of Boolean combinations of strips and translates of $A$ and $B$;
\stepcounter{equation}
\item\label{eq:ZW} for every $\vec{n}\in\{-1,0,1\}^d$, $\ell\in \range(f)$ and $\vec{p}\in\{0,\ldots,5\}^d$, the set $Z_{\vec{n},\ell}^{f_{J\Tilde{K}^{\vec{p}}\Tilde{L}^{\vec{p}}}}$ can be written as a Boolean combination of some sets in $\mathcal{W}$.
\end{enumerate}
\end{lemma}

\begin{proof}
Recall that the flow $f_J$ is obtained by rounding $(f_{m_i})_{i=1}^\infty$ on $(J_i)_{i=1}^\infty$. For the proof, we need to define two further flows:
 \begin{itemize}
 \item $f_{JK^{\vec{p}}}$ is the rounding of $(f_{m_i},f_{m_i})_{i=1}^\infty$ on the toast sequence $(J_i,K_i^{\vec{p}})_{i=1}^\infty$,\mbox{ and}
 \item $f_{K^{\vec{p}}L^{\vec{p}}}$ is the rounding of $(f_{m_i},f_{m_i'})_{i=1}^\infty$ on the toast sequence $(K_i^{\vec{p}},L_i^{\vec{p}})_{i=1}^\infty$.
 \end{itemize}
 Recall that each of $(J_i,K_i^{\vec{p}})_{i=1}^\infty$ and $(K_i^{\vec{p}},L_i^{\vec{p}})_{i=1}^\infty$ is a toast sequence by Lemma~\ref{lem:JKLscaff}. Also,
 Lemma~\ref{lem:intFlow} 
 (with  the used toast sequence $(J_i',K_i',L_i')_{i=1}^\infty$ being respectively $(J_i,K_i^{\vec{p}},\emptyset)_{i=1}^\infty$ and $(\emptyset,K_i^{\vec{p}},L_i^{\vec{p}})_{i=1}^\infty$) applies to each of these two new flows.
 
For every $\vec{n}\in\{-1,0,1\}^d$, $\ell\in\range(f)$, $\vec{p}\in\{0,\ldots,5\}^d$ and $i\ge1$, define
\begin{eqnarray*}
 W^J_{\vec{n},\ell,i}&:=&\left\{\vvec{u}: f_J\left(\vvec{u},\vec{n}\cdot_a\vvec{u}\right)=\ell\mbox{ and } 
\{\vvec{u},\vec{n}\cdot_a\vvec{u}\}\cap \left(\bigcup_{j=1}^i J_j\right)\neq\emptyset \right\},\\
 W^{JK^{\vec{p}}}_{\vec{n},\ell,i}&:=&\left\{\vvec{u}: f_{JK^{\vec{p}}}\left(\vvec{u},\vec{n}\cdot_a\vvec{u}\right)=\ell \mbox{ and } 
 \{\vvec{u},\vec{n}\cdot_a\vvec{u}\}\cap \left(\bigcup_{j=1}^i K_j^{\vec{p}}\right)\neq\emptyset\right\},\\
 W^{K^{\vec{p}}L^{\vec{p}}}_{\vec{n},\ell,i}&:=&\left\{\vvec{u}: f_{K^{\vec{p}}L^{\vec{p}}}\left(\vvec{u},\vec{n}\cdot_a\vvec{u}\right)=\ell \mbox{ and } 
 \{\vvec{u},\vec{n}\cdot_a\vvec{u}\}\cap \left(\bigcup_{j=1}^i (K_j^{\vec{p}}\cup L_j^{\vec{p}})\right)\neq\emptyset\right\}.
\end{eqnarray*}
 In other words, these sets are obtained by restricting $f_J$, $f_{JK^{\vec{p}}}$ and $f_{K^{\vec{p}}L^{\vec{p}}}$ to all edges intersecting respectively $\bigcup_{j=1}^i J_j$, $\bigcup_{j=1}^i K_j^{\vec{p}}$ and $\bigcup_{j=1}^i (K_j^{\vec{p}}\cup L_j^{\vec{p}})$ and encoding
 the obtained partially defined flows by sequences of sets of vertices similarly as in~\eqref{eq:Zdef}. Recall that $J_j\subseteq K_j^{\vec{p}}$ for every $j$.
  By Lemma~\ref{lem:intFlow}, each of these newly defined sets is a local function of $A$, $B$ and finitely many of sets $J_i, K_i^{\vec{p}}, L_i^{\vec{p}}$ while, in turn, each of the sets $J_i, K_i^{\vec{p}}, L_i^{\vec{p}}$ is a finite union of strips by Lemma~\ref{lem:JKLstrips}. So each set defined above  is a Boolean combination of strips and translates of $A$ and $B$. 

 Now,  for all $\vec{n}$ and $\ell$, define
\[
 W^{J}_{\vec{n},\ell}:=\bigcup_{i=1}^\infty W^{J}_{\vec{n},\ell,i},\quad
W^{JK^{\vec{p}}}_{\vec{n},\ell}:=\bigcup_{i=1}^\infty W^{JK^{\vec{p}}}_{\vec{n},\ell,i}\quad\mbox{and}\quad W^{K^{\vec{p}}L^{\vec{p}}}_{\vec{n},\ell}:=\bigcup_{i=1}^\infty W^{K^{\vec{p}}L^{\vec{p}}}_{\vec{n},\ell,i}.\]
We define $\mathcal{W}$ to be the collection of all sets of the form 
$W^{J}_{\vec{n},\ell}$,
$W^{JK^{\vec{p}}}_{\vec{n},\ell}$ and 
$W^{K^{\vec{p}}L^{\vec{p}}}_{\vec{n},\ell}$ for any $\vec{n},\ell$ and $\vec{p}$. By above, the family $\mathcal{W}$  satisfies~\eqref{eq:Wgood}.

By construction (i.e.\ by Conclusion~\eqref{eq:fLocalOnlyJ} of 
Lemma~\ref{lem:intFlow}), the flows $f_{J\Tilde{K}^{\vec{p}}\Tilde{L}^{\vec{p}}}$ and $f_{J}$ coincide on every pair intersecting $\bigcup_{i=1}^\infty J_i$.
Also, if an edge $\vvec{u}\vvec{v}$ intersects  $K_j^{\vec{p}}$ then, by~\eqref{eq:TildeKj}, the value of $f_{J\Tilde{K}^{\vec{p}}\Tilde{L}^{\vec{p}}}$ on this edge equals the value of $f_{JK^{\vec{p}}}$ unless there exists $i>j$ such that $\vvec{u}\vvec{v}$ intersects $J_i$ when we use the value of $f_{J}$ (which may happen to coincide with the value of $f_{JK^{\vec{p}}}$ on $\vvec{u}\vvec{v}$). Furthermore, we can drop the restriction that $i>j$ here, since our rounding procedures for any $J_i$ and $J_{i'}$ produce the same value on each edge intersecting $J_i\cap J_{i'}$. 
Similarly, if an edge intersects $L_j^{\vec{p}}$, 
then the value of $f_{K^{\vec{p}}L^{\vec{p}}}$ on this edge is retained by $f_{J\Tilde{K}^{\vec{p}}\Tilde{L}^{\vec{p}}}$,
unless it is overwritten (to a different value or the same one) due to $J_i$ or $K_i^{\vec{p}}$ with $i>j$ intersecting this edge; again we can drop the restriction that $i>j$ here. 

Thus, informally speaking, our flow values are partitioned into the following three types: $J$-values, $K$-values not overwritten by $J$, and $L$-values not overwritten by~$J$ nor~$K$, that is, the order of precedence is $J,K,L$. 
Formally, we can express the above partition for every non-zero $\ell$ as
\begin{align*}
Z_{\vec{n},\ell}^{f_{J\Tilde{K}^{\vec{p}}\Tilde{L}^{\vec{p}}}} 
&= 
W^{J}_{\vec{n},\ell}\cup 
\left(W^{JK^{\vec{p}}}_{\vec{n},\ell} 
\setminus \bigcup_{s\in\range(f_J)} W^{J}_{\vec{n},s}\right)\\
&\cup
 \left(W^{K^{\vec{p}}L^{\vec{p}}}_{\vec{n},\ell} 
\setminus 
\left(\left( \bigcup_{s\in\range(f_J)} W^{J}_{\vec{n},s}\right)\cup
\left( \bigcup_{s\in\range(f_{JK^{\vec{p}}})} W^{JK^{\vec{p}}}_{\vec{n},s}\right)\right)\right).
\end{align*}
Note that each union is over a finite set (since the involved flows are integer-valued and uniformly bounded), thus satisfying~\eqref{eq:ZW}. The case $\ell=0$ is somewhat special since the union $\bigcup_{i=1}^\infty(J_i\cup K^{\vec{p}}_i\cup L^{\vec{p}}_i)$ need not cover the whole torus. However, this case also satisfies~\eqref{eq:ZW} since 
$Z_{\vec{n},0}^{f_{J\Tilde{K}^{\vec{p}}\Tilde{L}^{\vec{p}}}}$ is the complement of the
union of the sets $Z_{\vec{n},\ell}^{f_{J\Tilde{K}^{\vec{p}}\Tilde{L}^{\vec{p}}}}$ over all possible (finitely many) non-zero flow values~$\ell$.
This completes the proof of the lemma.
\end{proof}

\begin{proof}[Proof of Theorem~\ref{th:main}\ref{it:b}]
Since every strip is the difference of two open sets, Part~\ref{it:b} of Theorem~\ref{th:main} now follows by combining \eqref{eq:Zunionp} with Lemmas~\ref{lem:TpHierarchy} and~\ref{lem:Zpcomplexity}.
\end{proof}


\section*{Acknowledgements}

The authors are very grateful to the anonymous referee for the extremely careful reading of the manuscript and many very useful comments.

\medskip
\noindent For the purpose of open access, the authors have applied a Creative
Commons Attribution (CC-BY) licence to any Author Accepted Manuscript
version arising from this submission.

\appendix

\section{Laczkovich's Discrepancy Bound}
\label{app:discrep}

The purpose of this appendix is to sketch a proof of Lemma~\ref{lem:discrep}. This result is implicit in \cite{Laczkovich92b}*{Proof of Theorem~3} and its proof sketch is given in~\cite{GrabowskiMathePikhurko17}*{pp.~677--678}. Since the dependence of $d$ on $k$ and $\boxdim(\partial X)$ (which is crucial for our estimates in Section~\ref{sec:Dim}) is not explicitly calculated there, we present a slightly expanded proof sketch.

Let $X \subseteq\mathbbm{T}^k$ be a measurable set such that $k-1\le \boxdim(\partial X)<k$. The first step is to use the upper Minkowski dimension of the boundary to reduce the proof to bounding discrepancy of $N_r^+[\vvec{u}]$ relative to boxes in~$\mathbbm{T}^k$. This argument is due to Niederreiter and Wills~\cite{NiederreiterWills75}*{Kollorar~4}. Recall that $d$ is an integer such that $d>{k}/({k-\boxdim(\partial X)})$ and $\varepsilon$ is a real number satisfying $0 <\varepsilon<({d(k-\boxdim(\partial X))-k})/{k}$. Define
\[\alpha:=\frac{(1+\varepsilon)k}{d}.\]
By the definition of upper Minkowski dimension and the fact that $\alpha<k-\boxdim(\partial X)$, there exists $\delta_0\in (0,1)$ such that
\begin{equation}\label{eq:distDelta}\lambda\left(\left\{\vvec{x}: \dist_\infty(\vvec{x},\partial X)\leq \delta\right\}\right)\leq \delta^\alpha\end{equation}
for all $0<\delta<\delta_0$. Now, for $r\in \mathbbm{N}$, choose $\delta\in (0,\delta_0)$ small with respect to $r$, where the dependence is clarified below. For convenience, let us assume that $\delta^{-1}$ is an integer. 

Let $\mathcal{P}$ be the partition of $\mathbbm{T}^k$ into a grid of $\delta^{-k}$ boxes, each with side-length $\delta$. Let $\mathcal{B}$ be the elements of $\mathcal{P}$ which intersect $\partial X$. We have 
\[\delta^k|\mathcal{B}|\leq\lambda\left(\left\{\vvec{x}: \dist_\infty(\vvec{x},\partial X)\leq \delta\right\}\right)\leq  \delta^\alpha\]
by \eqref{eq:distDelta} and so $|\mathcal{B}|\leq \delta^{-k+\alpha}$. 

Let $\mathcal{I}$ be the elements of $\mathcal{P}$ contained in the interior of $X$. We then let $\mathcal{I}^*$ be the collection of boxes obtained by starting with $\mathcal{I}$ and iteratively merging two boxes if they have the same projection onto the first $k-1$ coordinates and their closures share a $(k-1)$-dimensional face. For any two distinct boxes in $\mathcal{I}^*$ which have the same projection onto the first $k-1$ coordinates, there must be at least one element in $\mathcal{B}$ ``between'' them which prevents them from merging. Conversely, each element of $\mathcal{B}$ prevents at most one potential merging. Therefore,
\[|\mathcal{I}^*|\leq \delta^{-k+1} + |\mathcal{B}|\leq \delta^{-k+1}+\delta^{-k+\alpha}< 2\delta^{-k+\alpha},\]
since $\alpha<1$ by our assumption that $\boxdim(\partial X)\ge k-1$. Thus, $\mathcal{I}^*\cup\mathcal{B}$ is a covering of $X$ with at most $3\delta^{-k+\alpha}$ boxes such that each of the boxes in $\mathcal{B}$ has measure at most $\delta^{k}$. 

Now, given any finite set $F\subseteq\mathbbm{T}^k$, by the triangle inequality applied to the partition $X=(\cup_{I\in \mathcal{I}^*}I)\cup (\cup_{I\in \mathcal{B}} (I\cap X))$, the discrepancy of $F$ relative to $X$ can be bounded as follows:
\begin{equation}\label{eq:Dsplit}D(F,X)\leq \sum_{I\in \mathcal{I}^*}D(F,I) + \sum_{I\in \mathcal{B}}D(F,I\cap X)\end{equation}

We apply Lemma~\ref{lem:intervalLog} to the first sum on the right side of \eqref{eq:Dsplit}, say with $t:=1$, and obtain that 
\[\sum_{I\in \mathcal{I}^*}D\left(N_r^+[\vvec{u}], I\right)\leq C\log^{k+d+t}(r)\,|\mathcal{I}^*| \leq 2C\log^{k+d+t}(r)\delta^{-k+\alpha}.\]
 The contribution of $I\in\mathcal B$ to the second sum is
 at most 
 \hide{$(r+1)^d\delta^k$  if $$(r+1)^d\mu(X\cap I)-|N_r^+[\vvec{u}]\cap I\cap X|\ge 0$$ and at most $|N_r^+[\vvec{u}]\cap I|\le (r+1)^d\delta^k+C\log^{k+d+t}(r)$ otherwise.
 Thus}
 $$
 \max\left\{\,|N_r^+[\vvec{u}]\cap (I\cap X)|,\,\, |N_r^+[\vvec{u}]|\, \mu(X\cap I)\,\right\}\le
 \max\left\{|N_r^+[\vvec{u}]\cap I|, (r+1)^d\mu(I)\right\},
 $$
 and each term can be bounded from above by $(r+1)^d\delta^k + C\log^{k+d+t}(r)$. Thus, by $|\mathcal{B}|\leq \delta^{-k+\alpha}$, we have
\begin{eqnarray*} \sum_{I\in\mathcal{B}} D\left(N_r^+[\vvec{u}], I\right)
\leq \left((r+1)^d\delta^k + C\log^{k+d+t}(r)\right)\delta^{-k+\alpha}.
\end{eqnarray*}
So, if we set $\delta:=(r+1)^{-d/k}$ for $r\to\infty$, then we get
\[D\left(N_r^+[\vvec{u}],X\right)\leq (r+1)^{d-\alpha d/k+o(1)} = (r+1)^{d-1-\varepsilon+o(1)}.
\]
The extra $o(1)$ term in the exponent can clearly be taken care of by choosing $\varepsilon$ sufficiently close to $\frac{d(k-\boxdim(\partial X))-k}{k}$ in the beginning. Thus, Lemma~\ref{lem:discrep} follows from Lemma~\ref{lem:intervalLog}.

\bibliography{oleg,sets,misc,ramsey,enum,number,posets,sat,ex,matroid,design,random,graph,general,geometry,algorithm,Analysis,limits}

\end{document}